\documentclass[a4paper,11pt]{amsart}
\usepackage{amssymb,amsthm,bbm,epic,eepic,graphics,amsbsy,url,a4wide,verbatim}

\usepackage{amssymb,amsthm}
\usepackage[curve]{xypic}
\usepackage[applemac]{inputenc}
\usepackage{longtable}
\usepackage{array}
\newcommand{\la}{\lambda}

\newtheorem{theo}[equation]{Theorem}
\newtheorem*{theor}{Theorem}
\newtheorem{prop}[equation]{Proposition}

\newtheorem{coro}[equation]{Corollary}
\newtheorem{defi}[equation]{Definition}
\newtheorem{lemma}[equation]{Lemma}
\newtheorem{remark}[equation]{Remark}
\numberwithin{equation}{section}

\newcommand{\mmm}{{\raisebox{1mm}{\rule{1.5mm}{.2mm}}}} 
\newcommand\ii{\mathbf{i}}
\newcommand\jj{\mathbf{j}}
\newcommand\iid{\mathrm{i}}
\newcommand\into{\hookrightarrow}
\newcommand\onto{\twoheadrightarrow}

\newcommand\Id{{\mathrm{Id}}}
\renewcommand{\cong}{\simeq}
\DeclareMathOperator{\Ad}{\mathrm{Ad}}
\DeclareMathOperator{\Aut}{\mathrm{Aut}}
\DeclareMathOperator{\End}{\mathrm{End}}

\DeclareMathOperator{\Hom}{\mathrm{Hom}}
\DeclareMathOperator{\Ext}{\mathrm{Ext}}
\DeclareMathOperator{\Tor}{\mathrm{Tor}}
\DeclareMathOperator{\Image}{\mathrm{Im}}
\DeclareMathOperator{\Irr}{\mathrm{Irr}}
\DeclareMathOperator{\Ker}{\mathrm{Ker}}

\DeclareMathOperator{\Res}{\mathrm{Res}}
\DeclareMathOperator{\GL}{\mathrm{GL}}
\DeclareMathOperator{\GU}{\mathrm{GU}}
\DeclareMathOperator{\SU}{\mathrm{SU}}
\DeclareMathOperator{\PGU}{\mathrm{PGU}}
\DeclareMathOperator{\PSU}{\mathrm{PSU}}
\DeclareMathOperator{\PGL}{\mathrm{PGL}}
\DeclareMathOperator{\Sp}{\mathrm{Sp}}
\DeclareMathOperator{\PSp}{\mathrm{PSp}}

\DeclareMathOperator{\SL}{\mathrm{SL}}
\DeclareMathOperator{\Ze}{\mathrm{Z}}
\DeclareMathOperator{\Mat}{\mathrm{Mat}}
\newcommand{\C}{{\mathbbm C}}
\newcommand{\F}{{\mathbbm F}}
\newcommand{\kk}{{k}}
\newcommand{\kkb}{{\mathbf{k}}}
\newcommand{\kb}{\overline{\kk}}
\newcommand{\kt}{\kk^{\times}}
\newcommand{\Q}{{\mathbbm Q}}
\newcommand{\Z}{{\mathbbm Z}}

\def\lexp#1#2{\kern\scriptspace\vphantom{#2}^{#1}\kern-\scriptspace#2}

\newcommand{\GAP}{{\tt GAP}}
\newcommand{\GAPt}{{\tt GAP3}}
\newcommand{\GAPq}{{\tt GAP4}}
\newcommand{\MACAULAY}{{\tt Macaulay 2}}
\newcommand{\CHEVIE}{{\tt CHEVIE}}
\def\OOO{{\mathcal{ O}}}

\title{On ternary quotients of cubic Hecke algebras}
\author{Marc Cabanes \& Ivan Marin}
\address{Institut de Math\'ematiques de Jussieu, Universit\'e Paris 7,
175 rue du Chevaleret, F-75013 Paris}

\date{February 4, 2012}               

\def\eps{\varepsilon}

\newcommand{\AAA}{{K}}
\def\bb{\mathbf{b}}
\def\cc{\mathbf{c}}
\def\qq{\mathbf{q}} 
\def\la{{\lambda}}

\def\ab{{\rm ab}}

\def\inn{\subseteq}

\def\tri{\triangleleft}
\def\smd{\rtimes}

\def\te{\otimes}

\begin{document}
\maketitle

{\bf Abstract.}  We prove that the quotient of the group
algebra of the braid group introduced by L. Funar in \cite{FUNAR} collapses in
characteristic distinct from 2. In characteristic 2 we define
several quotients of it, which are connected to the classical Hecke
and Birman-Wenzl-Murakami quotients, but which admit in addition a
symmetry of order 3. We also establish conditions on the
possible Markov traces factorizing through it.

\bigskip

\section{Introduction}

Let $B_n$ be the braid group on $n$ strings ($n\geq 2$), that is the group
defined by $n-1$ generators $s_1,\dots ,s_{n-1}$ submitted to the relations
$s_is_j=s_js_i$ whenever $i-j\geq 2$, and
$s_{i}s_{i+1}s_{i}=s_{i+1}s_{i}s_{i+1}$ for any $i=1,\dots n-2$
(see e.g. \cite{BIRMAN} or \cite{KURPITA} for basic results on these groups).

This paper grew out as an attempt to understand the mysterious
`cubic Hecke algebras' defined by L. Funar and used in \cite{FUNAR} and \cite{FUNARBELLIN}.
In \cite{FUNAR}, an algebra $K_n(\gamma)$ for $\gamma \in \kk$ is defined over a commutative ring $\kk$
as the quotient of the group algebra $\kk B_n$ of the braid group
$B_n$ on $n$ strands, by the relations $s_i^3 = \gamma$,
and $s_{i+1} s_i^2 s_{i+1} + s_i s_{i+1}^2 s_i + s_i^2 s_{i+1} s_i +
s_i s_{i+1} s_i^2 + s_i^2 s_{i+1}^2 + s_{i+1}^2 s_i^2 + \gamma s_i + \gamma
s_{i+1} = 0$. Notice that the relations are equivalent to
$s_1^3 = \gamma$, $s_{2} s_1^2 s_{2} + s_1 s_{2}^2 s_1 + s_1^2 s_{2} s_1 +
s_1 s_{2} s_1^2 + s_1^2 s_{2}^2 + s_{2}^2 s_1^2 + \gamma s_1 + \gamma
s_{2} = 0$. The striking property of this algebra is that the latter
relation involves only $s_1,s_2$ and that, as proved in \cite{FUNAR},
it is a finitely generated $\kk$-module (hence finite dimensional over $\kk$
if $\kk$ is a field). Although many finite-dimensional
cubic quotients of the (group algebra of the) braid groups have been defined,
to our knowledge it is the only one which is not a quotient of the
classical Birman-Wenzl-Murakami algebra and which can be defined from
relations in $\kk B_3$. Notice that, whenever $\gamma$ admits an invertible
third root $\alpha \in \kk$ with $\alpha^3 = \gamma$, we have $K_n(\gamma)
\simeq K_n(1)$ under $s_i \mapsto \alpha^{-1} s_i$ -- and in particular always $K_n(-1) \simeq K_n(1)$.
Moreover, $K_n(1)$ is a quotient of the group algebra $\kk \Gamma_n$,
for $\Gamma_n = B_n/<s_1^3>$. This group $\Gamma_n$ is a semidirect product
$\Gamma_n^0 \rtimes C_3$, with $C_k$ denoting the cyclic group of order $k$,
and the defining ideal of $K_n(1)$ has the remarkable property to be generated by a $C_3$-invariant
ideal in $\Z \Gamma_3^0$ -- thus deserving the name ternary used
in the title.

By a theorem of Coxeter, $\Gamma_n$ is finite if and only if $n \leq 5$.
Moreover, in this case it is a finite complex reflection group,
and, as was conjectured by Brou\'e, Malle and Rouquier,
$\kk  \Gamma_n$ for $n \leq 5$ admits a flat deformation similar to the presentation of the ordinary
Hecke algebra as a deformation of $\kk \mathfrak{S}_n$. 
This has been proved in \cite{BM}, Satz 4.7 for $n = 3,4$, and recently
in \cite{G32} for $n=5$.
Partly stimulated by this conjecture,
the authors of \cite{FUNARBELLIN} constructed a deformation
of $K_n(\gamma)$ (still finitely generated).

The main motivation in \cite{FUNAR} and \cite{FUNARBELLIN} is to construct
link invariants. In \cite{FUNAR} it is claimed that $K_n(-1)$
admits a Markov trace with values in $\Z/6\Z$. A more general
statement is claimed in \cite{FUNARBELLIN}, that the constructed
deformation provides a link invariant with values in some extended ring.
Around 2004-2005, S. Orevkov pointed out a gap in a part of \cite{FUNARBELLIN}
devoted to the proof of the invariance of the trace under Markov
moves, which originates in \cite{FUNAR}. In 2008, the second author of the present paper noticed that, when $\kk$ is a field
of characteristic 0,
the 'tower of algebras' $K_n(1)$ collapsed, more precisely
that $K_n(1) = 0$ for $n \geq 5$ (see theorem \ref{theodif23} below). However, when $\kk = \Z$, this
tower does not collapse. This can be seen from the fact that
the natural group morphisms $\Gamma_n \onto C_3$
induce morphisms $\Z \Gamma_n \onto
\Z C_3 \onto (\Z/8\Z)C_3$ which factorize through $K_n(1)$.

\subsection{Statement of the main results}

Letting $\AAA_n = K_n(1)$ we prove (see corollary \ref{corfg} and theorems \ref{theodif23} and \ref{theocol3})
\begin{theor} When $\kk = \Z$,
\begin{enumerate}
\item $\AAA_n$ is a finite $\Z$-module for $n \geq 5$.
\item The exponent (as an abelian group) of $\AAA_n$ has the form $2^r 3^s$
for some $r,s$ (depending on $n$) when $n \geq 5$.
\item The exponent of $\AAA_n$ is a power of $2$ (\emph{not} depending on $n$)
when $n \geq 7$.
\end{enumerate}
\end{theor}

When $\kk$ is a field, in order to get a stably nontrivial structure,
we thus need to assume that $\kk$ has characteristic 2.

\begin{theor} Assume $\kk$ is a field of characteristic 2. For all $n$,
there exists a quotient $\mathcal{H}_n$ of $\AAA_n$,  which has dimension
$3(n!-1)$ and which embeds inside a product of three Hecke algebras.
This algebra $\mathcal{H}_n$ is the quotient of $\kk \Gamma_n$
by the relation $s_1 s_2^{-1} + s_2 s_1^{-1} + s_1^{-1} s_2 + s_2^{-1} s_1 = 0$.
\end{theor}

We call this algebra the ternary Hecke algebra, as it can be defined as
the quotient of $\kk \Gamma_n$ by the intersection of the
three ideals whose corresponding quotients define the three possible Hecke algebras at
third roots of 1.

Taking $\kk = \Z$, we let $\AAA_{\infty}$ denote the direct limit of the
$\AAA_n$ under the natural morphisms $\AAA_n \to \AAA_{n+1}$, and we similarly
define $\mathcal{H}_{\infty}$. Using the second definition
above, $\mathcal{H}_{\infty}$ can be defined over $\Z/4\Z$.

We recall that a Markov trace on $\AAA_{\infty}$ is a $\Z$-module morphism
$t : \AAA_{\infty} \to M$, where $M$ is some $\Z[u,v]$-module, which
satisfies $t(xy) = t(yx)$ for all $x,y \in \AAA_{\infty}$, $t(xs_n) = u\,t(x)$
and $t(x s_n^{-1}) = v t(x)$ for all $x$ in (the image of) $\AAA_n$.
It can be shown that such a Markov trace, if it exists, is uniquely
determined by the value $t(1)$, and takes values in $\Z[u,v] t(1) \subset M$.

\begin{theor}
\begin{enumerate}
\item If $t : \AAA_{\infty} \to \Z[u,v] t(1)$ is a Markov trace, then
$16t(1) = 0$, $4uv. t(1)= 4t(1)$, $4u^3.t(1) = 4v^3.t(1) = -4t(1)$ and
$(3u^3 + 3 v^3 - 3uv + 1)t(1) = 0$.
\item If $4t(1) = 0$, then $t$ factors through $\mathcal{H}_{\infty}$
(defined over $\Z/4\Z$)
\item There exists a Markov trace $\bar{t} : \mathcal{H}_{\infty} \to (\Z/4\Z)[u,v]$
with $\bar{t}(1) = \bar{1} \in \Z/4\Z$, which originates from the Markov
traces on ordinary Hecke algebras.
\end{enumerate}
\end{theor}

Modulo 4, the most general link invariant that can be defined this way is thus
given by the following operation : take the Homfly polynomial,
consider the three possible specialisations `at third roots of 1',
and reduce these three values modulo 4.

Finally, we investigate another quotient of $K_n$, that we denote
$\mathcal{BMW}_n$ and which is obtained from the usual
Birman-Wenzl-Murakami algebras by a similar `ternary' operation.
Computer calculations seems to indicate that this quotient
is asymptotically very close to $K_n$. However, the study of this
quotient is more delicate, and we get only partial results on it.
This nethertheless shows that, over a field of characteristic 2, $K_n$
is actually \emph{larger} than all the quotients of $\kk \Gamma_n$
by relations on 3 strands that have been defined so far.

\subsection{Open problems}

The work leaves for now the following questions open :
\begin{enumerate}
\item Over $\Z/4\Z$, and even over $\Z$, does $\mathcal{H}_n$
coincide with the quotient of the group algebra of $\Gamma_n$
by the ideal generated by $s_1 s_2^{-1} - s_1^{-1} s_2 + s_2 s_1^{-1} - s_2^{-1} s_1$ ?
\item Which are the Markov traces on $K_n(1)$ with $4 t(1) \neq 0$ ? Are there
non-obvious ones~? (Notice that the natural projection $\Gamma_n \onto C_3 = <s>$
obviously induces a Markov trace $t : K_n \onto (\Z/8)C_3$ with $u=s$ and $v = s^2$.)
\item What is the minimal $r$ ($r \geq 3$) such that $2^r K_{\infty} = 0$ ? Note that
$2^5 K_{\infty} = 0$ by proposition \ref{propnew}.
\item We lack a nice description of the intersection of the defining
ideals of the `two Temperley-Lieb algebras', at third roots of 1 and in
characteristic 2. This would help understanding $\mathcal{BMW}_n$ (see Definition \ref{defbmw}).
\item Do we have $\mathcal{BMW}_{\infty} = K_{\infty}(1)$, over a field of characteristic 2  ? 
\item Are there `nice generators' for the defining ideal of $\mathcal{BMW}_n$ ?
\item We did not study here the deformation of $K_n$ proposed in \cite{FUNARBELLIN},
although we hope our work now provides a firmer ground for it. See \cite{G32}
for the characteristic $0$ case.
\item Does $\mathcal{H}_n$ admit a `nice' deformation, and a related
Markov trace ?
\item Is there a nice description of the algebra $K_4(1)$ in characteristic 3 ?
\item What are $K_5, K_6$ as modules over the ring $\Z_3$ of 3-adic integers ?
\item Are the natural morphisms $\Gamma_n \to \Gamma_m$ injective for $6 \leq n \leq m$ ?
\end{enumerate}

\subsection{Notations}

Let $G$ be a group. We denote 
by $\Ze (G)$, resp. $(G,G)$, the center, resp. derived subgroup of $G$, and 
we denote by $G^\ab $ the quotient $G/(G,G)$. If $H$ is a group on which $G$ acts by group automorphisms, we denote by $H\smd G$ the associated semi-direct product.

If $A$ is a ring and $G$ acts on $A$ by ring automorphisms, we denote by $A\smd G$ the semi-direct product ring, that is the free $A$-module $\oplus_{g\in G}Ag$ endowed with multiplication $(ag).(a'g')=a(g.a')gg'$ for $a,a'\in A$, $g,g'\in G$.

If $n\geq 1$ is an integer, one denotes by $C_n$ the cyclic group with $n$ 
elements. 

 For $\kk$ a field
we let $\overline{\kk}$ denote an algebraic closure of $\kk$.

If $G$ is a finite group, we denote by $\Irr (G)$ the set of irreducible 
characters of $G$, that is trace characters of simple $\C G$-modules.

If $A$ is a ring and $n\geq 1$ is an integer, one denotes by $\Mat_n(A)$ the 
ring of $n\times n$ matrices with coefficients in $A$. We will also use the more
general notation $Mat_I(A)$ for $I$ an arbitrary finite set.
One denotes by 
$\Id_n$ the identity matrix. One denotes matrix transposition by $M\mapsto {}^tM$.

Let $q$ be a power of a prime. We denote by $\F_q$ the field with $q$ 
elements. We denote by $\GL_n(q) = \GL_n(\F_q)$, resp. $\SL_n(q) = \SL_n(\F_q)$ the general and special 
linear groups in $\Mat_n(\F_q)$. One denotes by $\Sp_{2n}(q) = \Sp_{2n}(\F_q)$ the
multiplicative group of matrices $M\in \Mat_{2n}(\F_q)$ satisfying 
$$^tM \left(\begin{array}{cc}0 & \Id_n \\-\Id_n & 0\end{array}\right) 
M=\left(\begin{array}{cc}0 & \Id_n \\-\Id_n & 0\end{array}\right).$$ Let us denote by $a\mapsto \overline{a}  =a^q$ the field automorphism of $\F_{q^2}$ order 2, which extends as a ring automorphism of $\Mat_n(\F_q)$ denoted in the same fashion. One denotes by $\GU_n(q)$ the subgroup of matrices $M\in \GL_n(q^2)$ such that $$^t\overline{M}M=\Id_n.$$ Denote $\SU_n(q)=\GU_n(q)\cap \SL_n(q^2)$. When $m\leq n$ we always consider $\GU_m(q)$ as the subgroup of $\GU_n(q)$ fixing the last $n-m$ elements of the canonical basis of $\F_{q^2}^n$.

\subsection{Acknowledgements} 
The second author benefited from the ANR grant ANR-09-JCJC-0102-01, corresponding to the ANR project `RepRed'.

\section{Groups}

\subsection{The groups $\Gamma_n$}

Let $\Gamma_n$ be the quotient of $B_n$ obtained by adding the extra relations 
$s_i^3=1$ for any $i=1,\dots n-1$.

The following is due to Coxeter \cite{COX} (see also \cite{ASSCOX}).

\begin{theo}  $\Gamma_2$, $\Gamma_3$, $\Gamma_4$, $\Gamma_5$  are finite (complex) 
reflection groups, respectively denoted by $G(3,1,1)\cong C_3$, 
$G_4\cong Q\smd C_3$ where $Q$ is the quaternion of order 8 and $C_3$ acts 
by any automorphism of order 3, $G_{25}\cong \GU_3(2)$, 
$G_{32}\cong C_3\times\Sp_4(3)$ in the Shephard-Todd classification. 
Their orders are respectively $3$, $24$, $648$ and $155,520 = 2^7.3^5.5$.
For $n \geq 6$, $\Gamma_n$ is infinite.
\end{theo}

\rm
\medskip

The following is due to Assion \cite{ASS}.

\medskip

\begin{theo} \label{theoassion}
\begin{enumerate}
\item[(i)] Every non-trivial normal subgroup of $\Gamma_5$ contains either $\left((s_1s_2)^3.(s_3s_4)^3\right)^3$ or $s_3.s_1.s_1^{(s_2s_3)^3}.s_1^{(s_2s_3)^3(s_3s_4)^3}$.

\item[(ii)] Let $U(m)$ be the quotient of $\Gamma_{m+1}$ obtained by imposing 
the extra relation $\left((s_1s_2)^3.(s_3s_4)^3\right)^3=1$. Then it is isomorphic 
with $\GU _m(2)$ except when $m=2$ mod 3 in which case  
$$U(m) =Y_{m-1}\smd \GU_{m-1}(2)$$ where 
$Y_{m-1}  =\{ (x,V)\mid x\in\F_4, V\in \F_4^{m-1} 
x+\overline{x}+{}^t\overline{V}.V=0\}$ is endowed with the 
multiplication $(x,V).(x',V')=(x+x'+^t\overline{V}.V', V+V')$ and the action of $\GU_{m-1}$ is by $(x,V)^A=(x,A^{-1}V)$. 
\item[(iii)] For $n \geq 5$, the quotient of $\Gamma_n$
by the relation $s_3.s_1.s_1^{(s_2s_3)^3}.s_1^{(s_2s_3)^3(s_3s_4)^3}=1$
is a finite group, isomorphic to $\Sp_{n-1}(\F_3)$ if $n$ is odd,
and to the stabilizer of one vector in $\Sp_{n+1}(\F_3)$ if $n$ is even.
\end{enumerate}
\end{theo}

\medskip

\rm

\begin{remark} \label{centralizersassion} In \cite{ASS}, the group $U(m)$ for $m=2 \mod 3$ is defined in 
the projective unitary group $\PGU_{m+1}(2)$ as the centralizer of 
$\Id_{m+1}+E_{m+1}$ with $E_{m+1}$ the matrix 
$E=\left(\begin{array}{cc}1 & 1 \\1 & 1\end{array}\right)$ in dimensions 
$m,m+1$ (zeros elsewhere). This clearly amounts to the subgroup of 
$\GU_{3k}(2)$ of elements fixing the sum $e_{3k-1}+e_{3k}$ of the two last
elements of an orthonormal basis of $\F_4^{3k}$.
\end{remark}

For $n \leq m$, the classical embeddings $B_n \into B_m$
induce morphisms $\varphi_{n,m} : \Gamma_n \to \Gamma_m$. The length function
$B_n \onto \Z$ defined by $s_i \mapsto 1$ induces morphisms
$l_n : \Gamma_n \onto C_3$ such that $l_m \circ \varphi_{n,m} = l_n$.
In particular, the finite index subgroup $\Gamma_n^0 = \Ker l_n$ of $\Gamma_n$
is mapped to $\Gamma_{n+1}^0$ under $\varphi_{n,n+1}$.

Recall from \cite{BIRMAN,KURPITA} that $\Ze(B_n)$
is infinite cyclic, and generated for $n \geq 3$ by
$$
z_n = (s_1 s_2 \dots s_{n-1})^n.$$

We gather here a few additional results on these groups. For explicit computations
in the finite groups $\Gamma_n$ for $n \leq 5$, we used the
development version of the \CHEVIE\ package for \GAPt\ : in this package, the finite
complex reflection groups $G_4,G_{25},G_{32}$ are represented as permutation
groups on a set of `complex roots', which makes some computations easy
to do. This development version can be found at \url{http://www.math.jussieu.fr/~jmichel/chevie/index.html}.

\begin{theo}  \label{theoprelimgamma}
\begin{enumerate}
\item[(i)] The image of $z_5$ in $\Gamma_5$ has order 6 and generates $\Ze(\Gamma_5)$. Under the
isomorphism $\Gamma_5 \simeq C_3 \times \Sp_4(\F_3)$, $C_3$
is generated by $z_5^3$, while $z_5^2 \in \Ze(\Sp_4(\F_3))$.
\item[(ii)] Under $B_5 \onto \Gamma_5$, $z_5^2$ is identified
with $s_3.s_1.s_1^{(s_2s_3)^3}.s_1^{(s_2s_3)^3(s_3s_4)^3}$ and $z_5^3$ with $\left((s_1s_2)^3.(s_3s_4)^3\right)^3$. 
\item[(iii)] The natural morphisms $\Gamma_n \to \Gamma_m$ are injective
for $n \leq 5$.
\item[(iv)] The morphism $\Gamma_5 \to \Gamma_6$ admits a retraction,
i.e. there exists a morphism $p : \Gamma_6 \onto \Gamma_5$
with $p \circ \varphi_{5,6} = \Id_{\Gamma_5}$. In particular,
$\Gamma_6 = \Gamma_5 \ltimes \Ker p$. It is given by $p(s_5) = z_4^2 z_5^2$.
\item[(v)] For every $n$, $\Gamma_n$ is a semidirect product
$\Gamma_n^0 \rtimes C_3$, and $\Gamma_n^0$ is the commutator
subgroup of $\Gamma_n$.
\item[(vi)] For $n \geq 2 $, $\Gamma_{n+1}$ is normally
generated by $\varphi_{n,n+1}(\Gamma_n)$ ; For $n \geq 3$, $\Gamma_{n+1}^0$ is normally
generated by $\varphi_{n,n+1}(\Gamma_n^0)$.
\end{enumerate}
\end{theo}
\begin{proof}
Parts (i) and (ii) are easily checked by direct computations in $\Gamma_5 = G_{32}$
using \CHEVIE\ (and in addition part (i) consists in well-known properties of the
group $G_{32}$, also denoted $3 \times 2.S_4(3)$ in Atlas notation, see \cite{ATLAS} p. 26).
For part (iii), the case $m \leq 5$ follows from
the identification of $\Gamma_2,\Gamma_3,\Gamma_4$ with parabolic
subgroups of $G_{32}$
(see e.g. \cite{BMR}). We thus can assume $n=5$. Let $K = \Ker \varphi_{5,m}$.
We have $K \subset \Ker l_5$ since
$l_m \circ \varphi_{5,m} = l_5$. Since $l_5(z_5) = 5\times (5-1) \mod 3$ we
get $\Ker l_5 = \Sp_4(\F_3)$ and $K \tri \Sp_4(\F_3)$. Since $\Sp_4(\F_3)$
is quasisimple we have $K = \{ e \}$ or $K = \Ze(\Sp_4(\F_3))
= < z_5^3 >$ or $K = \Sp_4(\F_3)$. The third case is excluded because
$\Gamma_m$ is nontrivial and generated by conjugates of $\varphi_{2,m}
(s_1)$, the case $K = \Ze(\Sp_4(\F_3))$ would imply
the finiteness of $\Gamma_m \simeq \Sp_{m-1}(\F_3)$ by Assion's theorem,
contradicting Coxeter's theorem. This proves (iii). Proving (iv) amounts to
saying that $z_4^2 z_5^2 \in \Gamma_5$ has order 3, commutes with
the $s_i$
for $i \leq 3$, that is with $\Gamma_4$, which is clear, and that $s_4 (z_4^2 z_5^2) s_4
= (z_4^2 z_5^2) s_4 (z_4^2 z_5^2)$, which is easily checked using \CHEVIE ; this
proves (iv).
The first statement of part (v) is trivial, as the subgroup $<s_1 >$
generated by $s_1$ provides a complement to $\Gamma_n^0$ in $\Gamma_n$ ;
then, clearly $(\Gamma_n,\Gamma_n) \subset \Ker l_n = \Gamma_n^0$,
as $C_3$ is abelian. In order to prove that $\Gamma_n^0 \subset
(\Gamma_n,\Gamma_n)$, we consider the abelianization morphism
$\pi_{\ab} : \Gamma_n \to \Gamma_n^{\ab}$. From the braid relations
we have $\pi(s_i) = \pi(s_{i+1})$ for all $i$, hence
$\pi(g) = \pi(s_1^{l_n(g)})$ for all $g \in \Gamma_n$ ; this proves
$\pi(\Gamma_n^0) = \{ 1 \}$ hence (v).
Rewriting
the braid relation $s_i s_{i+1} s_i = s_{i+1} s_i s_{i+1}$
as $s_{i+1} = (s_i s_{i+1}) s_i (s_i s_{i+1})^{-1}$
we get that $\varphi_{n,n+1}(\Gamma_n)$ normally generates
$\Gamma_{n+1}$. Now
recall that, when $G$ is a group generated by elements $a_1,\dots,a_r$,
$H$ a subgroup of $G$ and $S \subset G$ a set of representatives of $G/H$
with set-theoretic section $G/H \to S$ denoted $x \mapsto \bar{x}$,
then $H$ is generated by the $ya_i \overline{ya_i}^{-1}$
for $i \in [1,r]$ and $y \in S$ (see e.g. \cite{MKS}). 
It follows that $\Gamma_n^0$ is generated by the $s_i s_1^{-1}$,
$s_1 s_i s_1^{-2} = s_1 s_i s_1$, $s_1^2 s_i = s_1^{-1} s_i$,
by taking $S = \{ 1, s_1, s_1^2 \}$ as a set of representatives
of $\Gamma_n/\Gamma_n^0 \simeq C_3$. Using $s_{i+1} = (s_i s_{i+1})
s_i (s_i s_{i+1})^{-1}$ we get that, for $i \geq 3$,
$s_{i+1} s_1^{-1} = (s_i s_{i+1}) s_i s_1^{-1} (s_i s_{i+1})^{-1}$,
$s_1 s_{i+1} s_1 = (s_i s_{i+1}) s_1 s_i s_1 (s_i s_{i+1})^{-1}$
and $s_1^{-1} s_{i+1} = (s_i s_{i+1}) s_1^{-1} s_i (s_i s_{i+1})^{-1}$.
Thus, for $n\geq 4$, the generators of $\Gamma_{n+1}$ involving
$s_n$ are conjugates of elements in $\varphi_{n,n+1}(\Gamma_n)$,
and this proves (vi) for $n \geq 4$. The case $n =3$ is easily checked
by hand.

\end{proof}

\begin{remark} Part (iii) of Assion's theorem has been generalized by
Wajnryb in \cite{WAJNRYB} ; the question of whether the natural morphisms
$\Gamma_n \to \Gamma_m$ are injective for $n \geq 6$ seems to be open ;
part (vi) is clearly false for $n=2$, as $\Gamma_2^0 = \{ 1 \}$.
\end{remark}

\subsection{Additional preliminaries on the groups $\Gamma_n, n \leq 5$}

The group $\Gamma_3\cong G_4$ is a semi-direct product $Q\smd C_3$ 
where 
$Q$ is the quaternion group of order 8, $C_3$ is the cyclic group of 
order 
$3$, and the semi-direct product is associated to any automorphism of $Q$ 
of order $3$. Writing classically $Q={<}\ii,\jj{>}$ with $\ii^2=\jj^2=z$ the 
central element of order $2$, $\kkb = \ii \jj$ and $C={<}s{>}$ with $s$ acting on $Q$ by the 
permutation $(\ii,\jj,\kkb)$, an isomorphism is obtained by $s_1\mapsto s$ and 
$s_2\mapsto \ii^3s$ (so that $s_1s_2^2 \mapsto \ii$).

Using the above morphisms we identify $\Gamma_3$ and therefore $Q$
to subgroups of $\Gamma_5$. In the sequel we will need to use the Atlas
character tables on elements of $Q$. For this we need to identify a few
conjugacy classes in $\Gamma_5 = C_3 \times \Sp_4(\F_3)$.
In Atlas notations, $\Sp_4(F_3) = 2.U_4(2)$ contains 2 classes
of order 2. One of the two being central (hence corresponding to $z_5^3$),
the value of the other one on any Brauer character in characteristic not $2$
lies in the column labelled 2a of \cite{ATLASBRAUER}.
Among the three classes of order 4 in $\Sp_4(\F_3)$, two are deduced one from
the other by multiplication by $z_5^3$. It is easily checked that,
if $x \in \Gamma_3 \subset \Gamma_5$ has order $4$, then it is
not conjugated to $x z_5^3$. It follows that the column of the ordinary or Brauer character
table corresponding to $x$ is the one labelled 4a in \cite{ATLASBRAUER}. We can thus read
on the tables the values taken by elements of $Q$ on ordinary and Brauer characters
in characteristic prime to $2$.

The group $\Gamma_5 \simeq
C_3 \times \Sp_4(\F_3)$ and therefore $\Sp_4(\F_3)$ contains another useful quaternion subgroup $Q_0$,
characterized up to $\Gamma_5$-conjugacy by $\Ze(Q_0) = < z_5^3>$. For later
computations, an explicit
description
of this subgroup in terms of the generators will turn out useful. A 2-Sylow subgroup
of $\Gamma_5$ is generated by the elements 
$ a_1 = s_2^{-1} s_3 s_1 s_2^{-1} s_3 s_1 s_2^{-1} s_1^{-1},
a_2 = s_3^{-1} s_2 s_{3}^{-1} s_1 s_2 s_3 s_1,
a_3 = s_4^{-1} s_3 s_4^{-1} s_3,
a_4 = s_4 s_3^{-1} s_4 s_2 s_3 s_1 s_2^{-1} s_1 s_3 s_1$.
Two generators of such a $Q_0$ are then given by
$\ii_0 = a_4^{-1} a_2 a_3 a_2$, $\jj_0 = a_4^2a_1$.

\subsection{The groups $Y_m$}
\label{sectionYm}

For $1 \leq r \leq m-1$ we let $e_r$ denote the $r$-th vector
of the canonical basis of $\F_4^{m-1}$, and we let
$\pi : Y_m \to \F_4^{m-1}$ denote the canonical
projection $(x,V) \mapsto V$. We choose $\alpha \in \F_4 \setminus
\F_2$, and let $i_r = (e_r,\alpha)$, $j_r = (\alpha e_r, \alpha)$.
Then $i_r,j_r$ have order 4 and generate a quaternion
subgroup $Q_r$ of $Y_m$. It is easily checked that
$Y_m$ is a central product of the $Q_r$, namely the
quotient of $Q_1 \times \dots \times Q_{m-1}$ by the
identification of the centers of $Q_1,\dots,Q_{m-1}$.
If $z$ denotes the generator of $\Ze(Y_m)$, the
elements of $Y_m$ can be uniquely written in the
form $i_{r_1} i_{r_2} \dots i_{r_u} j_{s_1} \dots j_{s_v} z^{\epsilon}$
with $\epsilon \in \{ 0, 1 \}$ and $r_1,\dots r_u$, $s_1,\dots, s_v$ distinct
indices. 

In particular, the group $Y_m$ is an extra-special 
group
of type $2^{1+2(m-1)}$ (see \cite{GORENSTEIN} \S\ 5.5). In characteristic distinct from $2$, such a group admits
$m-1$ linear characters and a $2^{m-1}$-dimensional
irreducible representation, afforded by the tensor
product of the 2-dimensional irreducible representations
of the $Q_r$.

We need to recall some basic facts on the representations
of the quaternion group. When $\kk$ is a field of characteristic
$p \neq 2$, the 1-dimensional representations are clearly defined over
$\kk$. When $\kk$ contains a primitive fourth root of unity $\omega$,
then the 2-dimensional representation can be defined
over $\kk$, through 
$
\ii \mapsto \left(\begin{array}{cc}0 & 1 \\-1 & 
0\end{array}\right)\ ,\ \  \jj\mapsto \left(\begin{array}{cc}\omega & 0 \\0 
& -\omega \end{array}\right).$
It is also defined over $\kk = \F_3$, through
$\ii \mapsto 
\left(\begin{array}{cc}0 & 1\\-1 & 
0\end{array}\right), \jj \mapsto \left(\begin{array}{cc}1 & 1\\1 & 
-1\end{array}\right).$ It follows that, under
these conditions on $\kk$, the $2^{m-1}$-dimensional
representation of $Y_m$ can be explicitly defined over $\kk$.

\section{Reminder on projective representations}

Let $G$ be a group and $\kk$ a field. An action of $G$
as algebra automorphisms of $\Mat_n(\kk)$ yields
a projective representation $\bar{\rho} : G \to \PGL_n(\kk)$ by the
Skolem-Noether theorem, hence a 2-cocycle $c : G \times
G \to \kt$ defined by $c(g_1,g_2) = \tilde{\rho}(g_1 g_2)
\tilde{\rho}(g_2)^{-1}\tilde{\rho}(g_1)^{-1}$
where $\tilde{\rho} : G \to \GL_n(\kk)$ is a set-theoretic lifting
of $\bar{\rho}$. It is always possible to choose $\tilde{\rho}(e) = \Id_n$,
which we always assume from now on. Then the cocycle satisfies $c(e,g) = 1$
for all $g \in G$ ; we say that such a cocycle is \emph{normalized}.
The corresponding class $[c] \in H^2(G,\kt)$ is trivial if and only if
we can lift $\bar{\rho}$ to a linear representation $\rho : G \to \GL_n(\kk)$.
In that case, if $c = \mathrm{d} \alpha$ for some
$\alpha : G \to \kt$,
i.e. $c(g_1,g_2) = \alpha(g_1 g_2) \alpha(g_2)^{-1} \alpha(g_1)^{-1}$,
then $\rho(g) = \alpha(g)^{-1} \tilde{\rho}(g)$ provides such a lifting.
Under our assumption, such an $\alpha$ satisfies $\alpha(e) = 1$.

We recall the short exact sequences in
low-dimensional group cohomology, provided by
\begin{enumerate}
\item the universal coefficients exact sequence :
$$0 \to \Ext(H_1 G ,\kt) \to H^2(G,\kt) \to \Hom(H_2G, \kt) \to 0.
$$ 
\item the Künneth exact sequence :
$$0 \to \Tor(H_0 G,H_1 K)\oplus\Tor(H_1 G,H_0 K) \to H_2(G \times K)
\to H_2 K \oplus (H_1 G \otimes H_1 K) \oplus
H_2 G \to 0$$
\end{enumerate}

We recall that, when $G$ is finite, then $H_2 G$ is
the so-called Schur multiplier of $G$.

\begin{lemma} \label{calculsH2}
\begin{enumerate}
\item $H^2(\Gamma_3,\kt) \simeq \Ext(C_3, \kt)$ hence $H^2(\Gamma_3,\kt) = 0$
when $char. \kk = 3$.
\item We have a short exact sequence
$0 \to \Ext(C_3, \kt) \to H^2(\GU (4,2),\kt) \to \Hom(C_2,\kt)\to 0.$
If $\kk$ is a finite field of characteristic 3, then $H^2(\GU(4,2),\kk) = C_2$.
\item Let $C_0 \simeq C_2 \times C_2$ denote the image
of $Q_0 \subset \Sp_4(\F_3)$ inside $\SU(4,2) \simeq \PSU(4,2) \simeq \PSp_4(\F_3)$,
and assume $\kk$ is a finite field of characteristic $3$.
Then the restriction morphism $H^2(\GU(4,2), \kt) \to H^2(C_0 , \kt)$
is injective.
\end{enumerate}
\end{lemma}
\begin{proof}
It is known that $H_2 \Gamma_3 = 0$ (see \cite{KARPI} table 8.3), whence (i).
We have $\GU(4,2) = C_3 \times \SU(4,2)$, and it is known that
$H_2 \SU(4,2) \simeq C_2$ (see \cite{KARPI} table 8.5) hence $H_2 \GU(4,2) = C_2$
by Künneth, since $\SU(4,2)$ is perfect and $H_2 C_3 = 0$. Then the short
exact sequence is the universal coefficients exact sequence. When $\kk$ has characteristic 3,
$\Hom(C_2, \kt) \simeq C_2$ since $-1 \neq 1$ in $\kk$,
and $\kt$ is 3-divisible hence
$\Ext(C_3,\kt) = 0$, which proves (ii). The group $\Gamma_5 = C_3 \times \Sp_4(\F_3)$
provides a nonsplit central extension of $\GU(4,2)$,
hence the nontrivial element of $H^2(\GU(4,2),C_2) \simeq H^2(\GU(4,2),\kt)$.
For $g_1,g_2 \in \GU(4,2)$ and arbitrary preimages $\tilde{g}_1,\tilde{g}_2, \widetilde{g_1 g_2}$
in $\Gamma_5 = C_3 \times \Sp_4(\F_3)$, it can be defined by $c(g_1,g_2) = 1$
if $\widetilde{g_1 g_2} = \tilde{g}_1 \tilde{g}_2$ and $c(g_1,g_2) = -1$ otherwise.
Restricting it to $C_0$ yields the cocycle associated to the extension
$1 \to \Ze(Q_0) \to Q_0 \to C_0 \to 1$ which is not split, hence (iii).
\end{proof}

\begin{lemma} \label{H2cycliques}
\begin{enumerate}
\item[(i)] Let $x,y$ be generators of $C_2^2$ and let $c : (C_2^2)^2 \to \F_3^{\times}$ be a normalized 2-cocycle.
Its  class $[c]$ is trivial in $H^2(C_2^2,\F_3^{\times})$ if and only if $c(x,y) = c(y,x)$
and $c(x,x) = c(y,y) = 1$.
\item[(ii)] Let $g$ be a generator of $C_3$ and let $c : C_3^2 \to \F_4^{\times}$
be a normalized 2-cocycle. Its class $[c]$ is trivial in $H^2(C_3,\F_4^{\times})$ if
and only if $c(g,g)c(g,g^{-1}) = 1$.
\end{enumerate}
\end{lemma}
\begin{proof}
The group $H^2(C_2^2,\F_3^{\times}) \simeq H^2(C_2^2,C_2)$ is an extension of $\Hom(H_2 C_2, \F_3^{\times})
= \Hom(C_2,\F_3^{\times})$ by $\Ext((C_2)^2,\F_3^{\times}) = \Ext((C_2)^2,C_2)
\simeq (C_2)^2$. We check that the normalized cocycles associated to the
nonsplit extensions of $C_2$ by $C_2 \times C_2$ satisfy $c(x,y) = - c(y,x)$
when the extension is not abelian, and $|\{ c(x,x) , c(y,y) \}| = 2$
when it is. Conversely, all coboundaries satisfy $c(x,y) = c(y,x)$ and
$c(x,x) = c(y,y)$, which proves (i). The proof of (ii) is similar and
left to the reader.
\end{proof}

We will use the following in several instances.

\begin{prop}\label{nontwisted}
Let $G$ be a finite group, $\kk$ a commutative ring and $A$ a $\kk$-algebra. 
Let $f\colon G\to A^\times$ be a group morphism. This induces an action of $G$ on $A$ by conjugacy. Then the 
associated semi-direct product $A\rtimes G$ (defined by multiplication $ag.a'g'=af(g)a'f(g)^{-1}gg'$) is isomorphic 
with the (commutative) tensor product $A\otimes \kk G$. 
\end{prop}

\begin{proof}
 The map is $a\otimes g\mapsto af(g^{-1}).g$ since $(af(g^{-1}).g).bf(h^{-1}).h=
af(g^{-1})f(g)bf(g^{-1})f(h^{-1}).gh=abf((gh)^{-1}).gh$ which is the image of $ab\otimes gh$. A reverse map 
is clearly afforded by $a.g\mapsto af(g)\otimes g$.
\end{proof}

The following essentially consists in making explicit a Morita equivalence summing up Mackey-Wigner's method of ``little groups" (see \cite{SERRE}~\S~8.2 and \cite{CE} ex. 18.6).

\begin{prop} \label{wigner}
Let $G$ a finite group (left) acting transitively on a set $X$. Let $k$ be a commutative ring, and let $A$ be the 
$\kk$-algebra $G\ltimes \kk^X$ where $\kk^X=\oplus_{x\in X}\kk\epsilon_x$ is endowed with the product law
($\epsilon_x\epsilon_{x'}=\delta_{x,x'}\epsilon_x$) and the action of $G$ is induced by the one on $X$. Then any choice 
of $x_0\in X$ with stabilizer $G_0\subseteq G$ and any choice of a ``section" $s\colon X\to G$ such that $s(x).x_0=x$ for all $x\in X$, 
define a unique isomorphism $$A\longrightarrow {\rm Mat}_X(k G_0)$$ sending each $\epsilon_{x}\in k^X$ ($x\in X$) to 
$\theta (\epsilon_x):=E_{x,x}$, and each $g\in G$ to $$\theta (g):=\sum_{x\in X}s(gx)^{-1}g.s(x) E_{gx,x}$$ 
(where $E_{x,y}\in{\rm Mat}_X(k )$ is the elementary matrix corresponding to $x,y\in X$).
\end{prop}

\begin{proof}
 Note that indeed $s(gx)^{-1}g.s(x)\in G_0$ since $s(gx).x_0 = gx=g.s(x).x_0$.

We assume $k={\Bbb Z}$. The general case is deduced by tensor product $-\otimes_{\Bbb Z} k$.

Note that we are below actually checking explicitly that, denoting $i=\epsilon_{x_0}$, one has $A\cong \End_{iAi}(Ai)^{\rm opp}$ where $Ai$ is a $A$-bimodule-$iAi$ isomorphic with $(iAi)^X$ as right $iAi$-module, with moreover $iAi=kG_0$ and $AiA=A$.

To check that the proposed formulae define a morphism between our algebras and in view of the law on $A$, it suffices to check that $\theta (g)\theta (g')=\theta (gg')$, $\theta (g)\theta (\epsilon_x)=\theta (\epsilon_{gx})\theta (g)$ and $\theta (\epsilon_x)\theta (\epsilon_{x'})=\delta_{x,x'}\theta (\epsilon_x)$ for each $g,g'\in G$ and $x,x'\in X$.

We have $\theta (g)\theta (g')= \sum_{x,x'\in X}s(gx)^{-1}g.s(x).s(g'x')^{-1}g'.s(x') E_{gx,x}  E_{g'x',x'}$. The product 
$E_{gx,x}  E_{g'x',x'}$ is $E_{gx,x'}$ whenever $x=g'x'$, and is zero otherwise. When $x=g'x'$, we also have $s(x).s(g'x')^{-1}=1$, so that $\theta (g)\theta (g')= \sum_{x'\in X}s(gg'x')^{-1}gg'.s(x') E_{gg'x',x'}=\theta (gg')$.

Samely, $\theta (g)\theta (\epsilon_x)= \sum_{x'\in X} s(gx')^{-1}g.s(x')   E_{gx',x'}E_{x,x}=s(gx)^{-1}g.s(x)   
E_{gx,x}$, while $\theta (\epsilon_{gx})\theta (g)=\sum_{x'\in X} s(gx')^{-1}g.s(x')  
 E_{gx,gx}E_{gx',x'}=s(gx)^{-1}g.s(x)   E_{gx,x}$ since $gx=gx'$ if and only if $x=x'$.

The morphism is now clearly surjective by the equation above (with $x=x_0$) since any elementary matrix is then reached 
up to an element of $G_0$, and the elements of $G_0\epsilon_{x_0}\subseteq A$ surject on ${\Bbb Z} G_0.E_{x_0,x_0}$.

Isomorphism follows by noting that we have a surjective morphism between free ${\Bbb Z}$-modules of equal rank. 
Since it has to be split, it is an isomorphism.
\end{proof}

\section{Algebras}

We define and study a quotient of the group algebra of the groups 
$\Gamma_n$.

\begin{defi} \label{defcq}  We define $\qq$ to be the sum of elements in $Q$,
and $\cc =\qq s_1$ (or equivalently $s_1 s_2 \cc = \qq$), that is
$$
\begin{array}{lcl}
\qq &=& 1+s_1s_2^2+s_2s_1^2+s_1^2s_2+s_2^2s_1+s_1s_2s_1+s_1^2s_2^2s_1^2+s_1s_2^2s_1s_2^2\in \Z \Gamma_3\\
\cc &=&  s_{2} s_1^2 s_{2} + s_1 s_{2}^2 s_1 + s_1^2 s_{2} s_1 +
s_1 s_{2} s_1^2 + s_1^2 s_{2}^2 + s_{2}^2 s_1^2 + s_1 + 
s_{2} \\
\end{array}
$$
and $I_n =\Z \Gamma_n.\qq.\Z\Gamma_n = (\qq) = (\cc)$ be the two-sided ideal it generates in $\Z\Gamma_n$ for 
any $n\geq 3$. Let $\AAA_n =\Z \Gamma_n/I_n$.
\end{defi}

Note that $\AAA_n$ is the algebra $K_n(1)$ of the introduction.

If $R$ denotes a (unital) commutative ring, we let $R \AAA_n$ denote the quotient of
$R \Gamma_n$ by $R I_n = R \Gamma_n .q. R \Gamma_n \subset R \Gamma_n$. We have $R\AAA_n \simeq \AAA_n \otimes_{\Z} R$.

\subsection{First results}

As proved by L. Funar, for every $n$ this algebra is a finitely
generated $\Z$-module. For the convenience
of the reader, we provide another (shorter) proof
of the following result of \cite{FUNAR}.

\begin{prop} (Funar)  \label{propfunar} Let $\overline{A}_n$ denote the image of
the natural morphism $\AAA_n \to \AAA_{n+1}$. One has
$\AAA_{n+1} = \overline{A}_n + \overline{A}_n s_n \overline{A}_n + \overline{A}_n s_n^2 \overline{A}_n$.
\end{prop}
\begin{proof} The case $n = 2$ is trivial, so we can proceed
by induction. Let $C_n = \overline{A}_n + \overline{A}_n s_n \overline{A}_n + \overline{A}_n s_n^2 \overline{A}_n$.
We have $1 \in C_n$ so we only need to prove
that $C_n$ is a left ideal. Since $\overline{A}_n C_n \subset C_n$
this amounts to saying $s_n C_n \subset C_n$, that
is $s_n \overline{A}_n s_n^{\eps} \overline{A}_n \subset C_n$ for $\eps \in \{ 0, 1, 2 \}$.
By the induction assumption $\overline{A}_n s_n^{\eps}
\overline{A}_n = \overline{A}_{n-1} s_n^{\eps} \overline{A}_n
+ \overline{A}_{n-1} s_{n-1} \overline{A}_{n-1} s_n^{\eps} \overline{A}_n
+ \overline{A}_{n-1} s_{n-1}^{\eps} \overline{A}_{n-1} s_n^{\eps} \overline{A}_n$.
Since $s_{n}$ commutes with $\overline{A}_{n-1}$ we get
$\overline{A}_n s_n^{\eps} \overline{A}_n
= s_n^{\eps} \overline{A}_n + \overline{A}_{n-1} s_{n-1} s_n^{\eps}
\overline{A}_n + \overline{A}_{n-1} s_{n-1}^2 s_n^{\eps} \overline{A}_n$.
Now $s_n s_n^{\eps} \overline{A}_n = s_n^{\eps + 1} \overline{A}_n
\subset C_n$,
$s_n \overline{A}_{n-1} s_{n-1}s_n^{\eps} \overline{A}_n
= \overline{A}_{n-1} s_n s_{n-1} s_n^{\eps} \overline{A}_n$
and $s_n \overline{A}_{n-1} s_{n-1}^2 s_n^{\eps} \overline{A}_n
= \overline{A}_{n-1} s_n s_{n-1}^2 s_n^{\eps} \overline{A}_n$.
It is thus sufficient to show that $s_ns_{n-1}s_n^{\eps} \in C_n$
and $s_n s_{n-1}^2 s_n^{\eps} \in C_n$ for $\eps \in \{0,1,2 \}$.
The case $\eps = 0$ is obvious, $s_n s_{n-1} s_n = s_{n-1} s_n s_{n-1}
\in C_n$, $s_n s_{n-1} s_n^2 = s_{n-1}^2 s_n s_{n-1} \in C_n$,
$s_{n} s_{n-1}^2 s_n^2 = s_{n-1}^2 s_n^2 s_{n-1} \in C_n$, and there only remains
to show that $s_n s_{n-1}^2 s_n \in C_n$. But $c \equiv 0$ in $\overline{A}_{n+1}$
implies, under conjugation by $\Gamma_{n+1}$, that 
$s_n s_{n-1}^2 s_n + s_{n-1} s_n^2 s_{n-1} +
s_{n-1}^2 s_n s_{n-1} + s_{n-1} s_n s_{n-1}^2 + s_{n-1}^2 s_n^2
+ s_n^2 s_{n-1}^2 +s_{n-1} + s_{n} = 0$ in $\overline{A}_{n+1}$,
hence $s_n s_{n-1}^2 s_n \in C_n$, and this concludes the proof.

\end{proof}

\begin{coro} \label{corfg} For all $n$, $\AAA_n$
is a finitely generated $\Z$-module.
\end{coro}

\medskip

The following lemma will be useful.

\begin{lemma} \label{simplesbrauer} Let $p$ be a prime, $H$ a finite group, $S$ 
is a simple $\overline{\F_p}H$-module, and $\phi$ is its Brauer character. 
Let $Q$ is a $p'$-subgroup, that is a subgroup whose order is not divisible by $p$, then $q:=\sum_{t\in Q}t$ annihilates $S$ if 
and only if $\phi (q)=0$. The same holds in characteristic $0$
for arbitrary $H$ and $Q$ and $\phi$ the ordinary character of a simple $\overline{\Q} H$-module.
\end{lemma}
\begin{proof}

We may replace $H$ by $Q$ itself and assume $\phi$ is the Brauer character of an arbitrary
finite dimensional ${\overline \F}_p Q$-module $S$. 
Since $Q$ is a $p'$-group, this module lifts to an $\OOO Q$-module ${\widehat S}$ where $\OOO$ is a finite extension of 
$\Z_p$. Then $\phi$ is the ordinary character of ${\widehat S}$. Since $q$ is an idempotent 
up to an invertible scalar of $\OOO$, we have $\phi (q)=0$ if and only if 
$q{\widehat S}=0$, and this is equivalent to $qS=0$.
The characteristic zero case is included in the above reasoning.
\end{proof}

The structure of $\AAA_3$ and $\AAA_4$ as $\Z$-modules can
be obtained by computer means, as $\Gamma_3$ and $\Gamma_4$ are
small enough : $\AAA_n$ is the quotient of $\Z \Gamma_n \simeq \Z^{|\Gamma_n|}$ by
the submodule spanned by the elements $g_1 \qq g_2$ for $g_1,g_2 \in \Gamma_n$. Using the
algorithms implemented in \GAPq\ for computing the Smith normal form, we get the following.

\begin{theo} As $\Z$-modules, $ \AAA_3 \simeq \Z^{21}$ and
$$
\AAA_4 \simeq \Z^{183} \oplus (\Z/2\Z)^{54} \oplus (\Z/3\Z)^{48}
\oplus (\Z/9\Z)^{18} $$
\end{theo}

The size of $\Gamma_5$ is too large for the same kind of computations
to settle the case of $\AAA_5$. However, we manage to get the following

\begin{prop} \label{dimA5p2} The algebra $\F_2 \AAA_5$ has dimension $3 \times 863 = 2589$.
\end{prop} 

\begin{proof}
For computing this dimension we cannot rely on usual high-level mathematical
software, and needed instead to write our own code.
The computation is done as follows. Since $\qq \in \F_2 \Gamma_5^0$,
we can content ourselves with computing the subspace spanned
by the $g_1 \qq g_2$ for $g_1,g_2 \in \Gamma_5^0 = \Sp_4(\F_3)$.
We can assume $g_1 \in \Gamma_5^0/N_{\Gamma_5^0}(Q_8)$ and
$g_2 \in Q_8 \backslash \Gamma_5^0$. Taking representatives
in $\Gamma_5^0$ of these cosets, this leaves 90
possibilities for $g_1$ and 6480 for $g_2$. Encoding each
entry on one bit, each vector in $\F_2 \Gamma_5^0$
occupies $6480$ bytes, and a basis of $\F_2 \Gamma_5^0$
occupies about 330 MBytes. The encoding of elements
of $\Gamma_5^0$ as matrices in $\Sp_4(\F_3)$ is
more economic than encoding them as permutations, and
the time-consuming procedures such as finding
$90 \times 6480 \times 8$ times the position of an element in the list
of the $51840$ elements of $\Gamma_5^0$ can be
optimized by using a numerical key and ordering these
elements. Each time a new sequence of 8 elements is
computed and converted into a new line vector,
a Gauss elimination is performed (using {\tt xor} operations
on 4-bytes words) with respect to the precedingly obtained
free family.  We wrote a {\tt C} program based on these ideas and computed
the dimension of this submodule (this lasts a few hours on todays PCs).
One gets $50977$, hence $\dim \F_2 \AAA_5 = 3 \times (51840-
50977) = 3 \times 863$.
\end{proof}

\begin{theo} 
If $k$ is an algebraically closed field
with $k=2k=3k$ (i.e. its 
characteristic is $\not=2,3$), then $k\AAA_3\cong 
\Mat_2(k)\times\Mat_2(k)\times\Mat_2(k)\times\Mat_3(k)$, 
$k\AAA_4\cong \Mat_2(\kk)^3 \times \Mat_3(\kk) \times \Mat_9(\kk)^2$ 
\end{theo} 
\begin{proof}
For the case $n=3$, we first just assume $2k=k$ and $k$ contains a primitive fourth root of 
unity $\omega$. Then one has $\kk Q=\kk \times \kk \times \kk \times \kk \times \Mat_2(\kk)$ by the only 
$k$-algebra map such that $$\ii\mapsto (1,-1,1,-1, \left(\begin{array}{cc}0 & 1 \\-1 & 
0\end{array}\right))\ ,\ \  \jj\mapsto (1,1,-1,-1, \left(\begin{array}{cc}\omega & 0 \\0 
& -\omega \end{array}\right)).$$

In $kQ$, $e_Q=\qq/8$ is a central idempotent acting by $1$ on the first coordinate above 
and by 0 on the others. So $kQ/kQe_Q\cong k^3\times \Mat_2(k)$ by the same map as above 
deleting the first coordinate, and $k\Gamma_3/k\Gamma_3e_Q$ is a semi-direct product 
$ [k^3\times \Mat_2(k)]\smd C_3$ where the generator of $C_3$ permutes cyclically the 
first three coordinates and acts on the summand $\Mat_2(k)$ according to 
$\ii\mapsto \jj\mapsto \ii\jj \mapsto \ii$, that is $\left(\begin{array}{cc}0 & 1 \\-1 & 
0\end{array}\right)\mapsto \left(\begin{array}{cc}\omega & 0 \\0 & -\omega 
\end{array}\right)\mapsto \left(\begin{array}{cc}0 & -\omega \\-\omega & 
0\end{array}\right)\mapsto \left(\begin{array}{cc}0 & 1 \\-1 & 0\end{array}\right)$. 
This last action is by conjugacy by $\left(\begin{array}{cc}-1 & \omega\\1 & 
\omega\end{array}\right)$, so Proposition \ref{nontwisted} implies that the corresponding semi-direct product is isomorphic with $\Mat_2(k)\otimes kC_3=\Mat_2(kC_3)$.

Note that when moreover $3k=k$ and $k$ contains a third root of unity, then $kC_3\cong k^3$ and $\Mat_2(kC_3)\cong \Mat_2(k)^3$.

The other semi-direct product $k^3\smd C_3$ is isomorphic with $\Mat_3(k)$ by identifying $k^3$ with diagonal matrices and sending the generator of $C_3$ to the permutation matrix of the appropriate cycle of order 3.

This gives the claim about $k\AAA_3$.  

We notice that the primes dividing the orders of $\Gamma_3$ and $\Gamma_4$ are $2,3$. It follows
that $\kk \Gamma_4$ is semisimple and that $\kk \AAA_4$ is a direct sum of $B_{\chi} = \Mat_{\chi(1)}(\kk)$
among all irreducible Brauer characters $\chi$ corresponding to modules $S$ with $\qq S \neq 0$,
that is $\chi(\qq) \neq 0$ by Lemma \ref{simplesbrauer}. Equivalently, $\chi(\qq) \neq 0$
means that the restriction of $S$ to $\Gamma_3$ does not contain any $1$-dimensional
component. The ordinary character and induction tables of $\Gamma_3 = G_4$ and
$\Gamma_4 = G_{25}$ are easily accessible using \CHEVIE, so this readily provides the set
of such characters and the conclusion.

\end{proof}

\subsection{Characteristic distinct from $2$ and $3$}

\begin{theo}  \label{theodif23}
If $k$ is a field with $k=2k=3k$ (i.e. its 
characteristic is $\not=2,3$), then $k \AAA_n = 0$ for $n \geq 5$.
\end{theo}
\begin{proof}
In order to prove $\kk \AAA_n = 0$ for $n \geq 5$, it is sufficient
to show that $\kk \AAA_5 = 0$, as $\kk \AAA_n$ is generated by conjugates
of the image of the natural morphism $\AAA_5 \to \AAA_n$. Since $\kk \AAA_5$
is a quotient of $\kk \Gamma_5$ it is finite dimensional, so
we can assume $\kk = \overline{\kk}$, as $\overline{\kk} \AAA_5 = \kk \AAA_5 \otimes_{\kk} \overline{\kk}$. The ordinary character table
and elements of the complex reflection group $G_{32} = \Gamma_5$
are easy to deal with using \CHEVIE. We get that no irreducible
character of $\Gamma_5$ vanishes on $\qq$, hence proving that
$\kk \AAA_5$ has no simple module by Lemma \ref{simplesbrauer},
hence $\kk \AAA_5 = 0$, provided that the characteristic of $\kk$ is
not $2,3$ or $5$. For $p = 5$ we use that $\Gamma_5 = C_3 \times \Sp_4(\F_3)$,
with $Q \subset \Sp_4(\F_3) \subset \Gamma_5 $, hence $\kk \AAA_5
= \kk C_3 \otimes (\kk \Sp_4(\F_3)/(\qq))$. We check that 
no 5-modular Brauer character of $\Sp_4(\F_3)$ vanishes on $\qq$
by using the table of Brauer
characters provided by \cite{ATLASBRAUER}, and
the conclusion follows again from Lemma \ref{simplesbrauer}.

\end{proof}

\subsection{Characteristic 3}

This section is devoted to the proof of the following. 

\begin{theo} \label{theocol3}
If $k$ is a field of characteristic $3$, then 
\begin{enumerate}
\item[(i)] $k\AAA_3\cong  \Mat_3(k)\times \Mat_2(kC_3)$.
\item[(ii)] $k\AAA_5\cong k\AAA_6 \cong {\rm Mat}_{25}(kC_3)$,
\item[(iii)] $k\AAA_n = 0$ for $n\geq 7$.
\end{enumerate}
\end{theo}

\subsubsection{\ } The case $n=3$ has been treated at the start of the proof of Theorem 7
provided that $\kk$ contains a 4-th root of 1. In the case $char. \kk = 3$
we remove that assumption. The irreducible representations of $Q $
are defined over $\kk$. This is clear for the 1-dimensional ones, and
the 2-dimensional one is given by $\ii \mapsto 
\left(\begin{array}{cc}0 & 1\\-1 & 
0\end{array}\right)$, $\jj \mapsto \left(\begin{array}{cc}1 & 1\\1 & 
-1\end{array}\right)$. The rest of the argument remains valid,
provided that the cocycle given by the projective representation
$C_3 \to \Aut(\Mat_2(\kk)) = \PGL_2(\kk)$ is zero in $H^2(C_3,\kt)$.
Since $H^2(C_3,\kt) \simeq \Ext(C_3,\kt) = 0$ when $char. \kk = 3$
this concludes the proof.
\medskip

\subsubsection{\ }  \label{ssecn5p3} Case $n=5$. \rm Let us look at $\Gamma_5^0= \Sp_4(3)$ whose group algebra contains 
$q$ since $\Gamma_5^0$ contains all $2$-elements of $\Gamma_5$. We have $\AAA_5= \kk C_3\otimes  A'_5$ where 
$A'_5=k\Gamma^0_5/I'_5$ and $I'_5$ is the two-sided ideal of $\kk \Gamma^0_5$ generated by $\qq$. In order to 
show that $\kk \AAA_5\cong {\rm Mat}_{25}(\kk C_3)$, it suffices to check that $\kk\Gamma^0_5/I'_5\cong  
{\rm Mat}_{25}(\kk)$. We first assume that $\kk$ is algebraically closed.

A first step is to check that all simple $k\Gamma_5^0$-modules except one are annihilated by 
$I_5'$. Using the table of Brauer characters of $\Sp_4(3)= 2.S_4(3)$, it is easy to check 
that only the simple $k\Gamma_5^0$-module $M$ of dimension 25 is such that its Brauer 
character (with values in $k$) $\tau_M$ satisfies $\tau_M(\qq) = 0$. So we have $\qq M=0$, by 
Lemma \ref{simplesbrauer}, $A'_5\not= 0$, and the only simple $k\Gamma_5^0$-module which gives rise to a 
$A'_5$-module is this module $M$ of dimension 25. 
Moreover, this module has no 
self-extension
as $k\Gamma_5^0$-module
by \cite{BENSON27} 12.2 (vi).
So this unique simple $A'_5$-module has no selfextension, so is 
projective, hence $A'_5\cong {\rm Mat}_{25}(k)$ as claimed.

In case $\kk \neq \kb$, we get from the above that $\kb A'_5 \simeq \Mat_{25}(\kb)$.
We prove that the $25$-dimensional irreducible representation of $\Sp_4(\F_3)$
is defined over $\F_3$, which provides a nontrivial surjective morphism
$\kk A'_5 \to \Mat_{25}(\kk)$, hence an isomorphism (e.g. by equality of dimensions).
The proof goes as follows. We let $\kk = \F_3$. The $4$-dimensional reflection
representation of $G_{32}$ is defined over $\Z[j]$, where $j = \exp(2 \iid \pi/3)$, hence defines,
after tensorisation by a suitable linear character, a $4$-dimensional
irreducible representation $\rho_0$ of $\Sp_4(\F_3)$ over $\Z[j]$.
We let $\bar{\rho}_0 : \Sp_4(\F_3) \to \GL_4(\F_3)$ denote its reduction modulo the ideal $(3,j+1)$
(which is isomorphic to the standard representation of $\Sp_4(\F_3)$ over $\F_3$).
We use the character table and the decomposition matrix of $\Sp_4(\F_3)$,
as provided by \cite{BENSON27} (or by the package {\tt CTblLib} of \GAPq) to show the following :
\begin{itemize}
\item $S^2 \rho_0$, $\Lambda^2 \rho_0$, $\Lambda^2 (S^2 \rho_0)$ are absolutely irreducible,
as well as $S^2 \bar{\rho}_0$.
\item The composition factors over $\overline{\F_3}$ of the 45-dimensional representation
$\Lambda^2 (S^2 \bar{\rho}_0)$ are $S^2 \bar{\rho}_0$ (twice) and
the 25-dimensional irreducible (once).
\end{itemize}
Since $S^2 (\Lambda^2 \bar{\rho}_0)$ and $\Lambda^2 \bar{\rho}$
are defined over $\F_3$, the same thus holds for our 25-dimensional representation. 

\medskip

Let us extract from the above the following proposition for future reference.
\medskip

\begin{prop} \label{prop25} Let $\kk$ be a field of characteristic 3.
Under the isomorphism $\Gamma_5\cong C_3\times \Sp_4(3)$, one has $\qq\in \kk \Sp_4(3)$ and the only simple $\kk \Sp_4(3)$ annihilated by
$\qq$ is the only simple $\kk \Sp_4(3)/Z(\Sp_4(3))=\kk \SU_4(2)$-module of dimension 25.
\end{prop}

\medskip

\subsubsection{\ } \label{old224}  Case $n=6$. \rm From the above, note that $z_5^3 - 1 \in I_5'$,
since the isomorphism $\kk A'_5/I'_5 \to \Mat_{25}(\kk)$ is given by the 25-dimensional
simple representation of $\Gamma'_5= 2.S_4(3)$, which factorizes through $S_4(3)$ (see \cite{ATLASBRAUER})
hence has the center $<z_5^3 >$ of $\Gamma'_5$ in its kernel.

Therefore, by Theorem \ref{theoassion} (and Theorem \ref{theoprelimgamma} (ii)),

$\kk \AAA_6$ is a quotient of the group algebra of $U(5)=Y_4\smd \GU_4(2)$, the $\GU_4(2)$ term
corresponding to $\Gamma_5/\Ze (\Gamma_5^0)$. Note that $\qq$ is a sum of elements of that group.
Let us show that the simple $k\AAA_6$-modules are all annihilated by $\qq$ except the one which corresponds with the 25-dimensional Brauer character of $\SU_4(2)$.
By Proposition~\ref{prop25}, we are looking for the simple $kU(5)$-modules whose restriction to $\SU_4(2)$ annihilates $\qq$, hence has all its composition factors isomorphic to the 25-dimensional representation singled out above. 

From the description of $U(5)$ recalled in Theorem \ref{theoassion} (ii), we have
$\kk U(5)=\kk Y_4.\GU_4(2)$ where $Y_4$ is clearly an extra-special group of type $2^{1+8}$ (notation 
of \cite{ATLAS}).
We have $\Irr (Y_4)=\Irr (kY_4)=\Irr (Y_4^\ab )\cup\{\chi_0\}$ where $\chi_0$ is the irreducible character of 
degree 16 (see \cite{GORENSTEIN} \S 5.5 on the characters of the extra-special groups). 

If $\la\in\Irr (Y_4^\ab)$, let $e_\lambda$ be the sum of idempotents of $kY_4$ associated with elements of the
orbit $U.\lambda \subset \Irr (Y_4^\ab)$. Let us abbreviate $U=\GU_4(2)$ and let $U_\lambda$ denote the stabilizer of $\lambda$ in $U$ by conjugacy.

By  Proposition \ref{wigner},  $kY_4U.e_\lambda\cong \Mat_{(U:U_\la )}(kU_\la )$, so the simple $kY_4U.e_\lambda$-modules are of dimensions $(U:U_\la )$ times the dimension of some simple $kU_\la$-module.

If $\la = 1$, then $U_\la=U$, so we find a block isomorphic to $kU$ and the quotient by the ideal generated by $\qq$ is $\Mat_{25}(k)$.

To study other stabilizers, note that $Y_4^{\rm ab}\cong\Irr (Y_4^\ab )$ by the hermitian form. This is $U$-equivariant, so we may consider those subgroups $U_\la$ as stabilizers of non trivial elements $V$ in the natural representation space $\F_4^4$.

If $^t\overline{V}V\not= 0$, then $\F_4^4 =\F_4.V\oplus V^\perp   $ and $U_\la$ then identifies with the 
unitary group on $V^\perp$, isomorphic with $\GU_3(2)$. By computing its Brauer character table (e.g.
using \GAPq), we get that its simple modules over $k$ have dimensions 1,2,3, so we get dimensions 1,2,3$\times (\GU_4(2):\GU_3(2))$ which is never a multiple of 25.

If $^t\overline{V}V= 0$, then $V$ can be taken as the sum of last two vectors of an orthonormal basis, 
so that the computation of its stabilizer is similar to the one of Remark \ref{centralizersassion}. Then $U_\la$ identifies with a 
semi-direct product $Y_2\smd \GU_2(2)$. By the discussion used above for $Y_4\smd \GU_4(2)$, one can sort out 
the dimensions of the simple $k[Y_2\smd \GU_2(2)]$-modules as follows. 
We first have $\GU_2(2)\cong (C_3\times C_3)\smd C_2$ with trivial Schur multiplier. So the simple projective representations of this group and the simple representations of its subgroups are of degree 1, hence the simple
$k[Y_2\smd \GU_2(2)]$-modules are of dimensions dividing $18$.  They are prime to 5, so that once multiplied with $(\GU_4(2):\GU_2(2))=1440$ they give dimensions not a multiple of 25.

 Let now $e_0$ be the idempotent of $kY_4$ corresponding to the only non linear character of the extra-special group $Y_4$. It is central in $kY_4U$, $e_0.kY_4\cong \Mat_{16}(k)$ by semi-simplicity and we have an action of $ U$ on the latter.
 
 The following shows that $e_0k\AAA_6=0$, thus establishing our claim.

\medskip

\begin{prop}
\begin{enumerate}
\item The above action of $\GU_4(2)$ on $\Mat_{16}(k)$ is induced by a morphism $\GU_4(2)\to \GL_{16}(k)$ and conjugacy.  \rm
\item One has $e_0\in kY_4U.\qq.kY_4U.$
\end{enumerate}
\end{prop}
\begin{proof} This action defines a projective representation $\GU(4,2) \to \PGL_{16}(\kk)$,
and we need to show that it is linearizable, meaning that the
induced element of $H^2(\GU(4,2),\kt)$ is zero. By Lemma \ref{calculsH2} it is sufficient
to compute its image in $H^2(C_0,\kt)$ where $C_0 \simeq (C_2)^2$ is
the image of the quaternion group $Q_0 \subset \Gamma_5$
in $\GU(4,2) \simeq C_3 \times \PSp_4(\F_3)$. We compute it explicitly
as follows. Since $\kk$ has characteristic 3, we can assume that
$\kk = \F_3$ and that the 16-dimensional representation
$\psi$ of $Y_4$ is defined over $\F_3$ by the matrix models given in Section \ref{sectionYm}.
For $x,y$ two generators of $C_0 \subset \GU(4,2)$, their actions
on $Y_4$ define twisted representations $\psi_x = \psi \circ \Ad x$,
$\psi_y = \psi \circ \Ad y$ of $Y_4$, which provides intertwinners
$P_x,P_y \in \GL_{16}(\kk)$ and a normalized cocycle. We check that they satisfy $P_x P_y = P_y P_x$
and $P_x^2 = P_y^2 = \Id_{16}$. From Lemma \ref{H2cycliques} it follows
that this cocycle is a coboundary, which concludes (i).

We let $U' = \SU(4,2) \subset \GU(4,2) = U$. 
From the above and Proposition \ref{nontwisted} we get that $e_0 \kk Y_4 \rtimes U' \simeq \Mat_{16}(\kk) \rtimes U'$
is isomorphic to $\Mat_{16}(\kk) \otimes \kk U'$. If $\rho : Q \to \GL_{16}(\kk)$
denotes the restriction to $Q \subset \GU(4,2)$ of the representation defined above, 
$\qq$ is mapped to $M = \sum_{g \in Q} \rho(g) \otimes g \in \Mat_{16}(\kk) \otimes \kk U'$
under this isomorphism. Then the ideal $e_0 \kk Y_4 U' \qq Y_4 U'$ of $e_0 \kk Y_4 U'$ is mapped to the ideal
generated by $M$ inside $\Mat_{16}(\kk) \otimes \kk U' \simeq \Mat_{16}(\kk U')$. Every ideal
of $\Mat_{16}(\kk U')$ being isomorphic to $\Mat_{16}(I)$ for some ideal $I$ of $\kk U'$,
we get that this ideal is $\Mat_{16}(I)$ for $I$ generated by the entries $(m_{ij})$
of the matrix $M$. In order to compute it we need to explicitly lift the representation $\bar{\rho} : Q
\to \PGL_{16}(\kk)$ afforded by the intertwinners to a linear representation $\rho$. It is clearly
sufficient to lift the generators $\ii,\jj$ of $Q$. Although any lifting will do, as $\kk^{\times}
= \{ -1 , 1 \}$ hence the set of the $\rho \otimes \chi$ for $\chi$ a linear character
of $Q$ covers all the possible liftings of the generators, we find that the four possible liftings
are not equivalent as representations of $Q$, hence only one is the restriction
of the linear representation of $U'$
providing the isomorphism. Nevertheless, computing the entries of $M$ in the four cases,
we find that $\ii\jj(z-1)$ belongs to all four possible vector subspaces of $\kk Q$ spanned by the entries
of $M$, where $z = \ii^2=\jj^2$. It follows that $z-1$ belongs to $I$. Since $U'=\SU(4,2)$ is simple, the conjugates of
$z \in Q \subset U'$ generate $U'$ hence $I$ contains the augmentation ideal of $\kk U'$.
As a consequence the quotient of $e_0 \kk Y_4 \rtimes U'$ by $\qq$ is either zero or isomorphic
to $\Mat_{16}(\kk)$. Since the image of $\kk U' \subset e_0 \kk Y_4 \rtimes U'$ factorize through $\Mat_{25}(\kk)$ it
has to be $0$. Since it generates $(e_0 \kk Y_4 \rtimes U')/(\qq)$ we get $e_0 \kk Y_4 U' = 
e_0 \kk Y_4 U' \qq Y_4 U'$.
\end{proof}

\medskip

\subsubsection{\ }  Case $n\geq 7$. \rm  It suffices to show that $\qq$ generates $\kk \Gamma_7$ as a two-sided ideal, 
to get the same in any $k\Gamma_n$ for any $n\geq 7$. By the argument at the start of \ref{old224} above,
$z_5^3 - 1$ belong to the ideal generated by $\qq$ in $\kk \Gamma_5$, hence to the ideal
generated by $(\qq)$ in $\kk \Gamma_7$, and 
Theorem \ref{theoassion} (ii) then implies that $\kk \AAA_7$ is a quotient of $\kk \GU_6(2)$ by the two-sided ideal generated by 
$\qq\in \kk\SU_4(2)$.

Assume that $\kk \AAA_7\neq 0$ and let $S$ be a simple $\kk \AAA_7$-module. We see it as a simple $\kk \GU_6(2)$-module such 
that $\qq S = 0$.
Since the restriction of $S$ to $\SU_4(2)$ is a module annihilated by $\qq$, all its composition factors are 
isomorphic to the same 25-dimensional simple $\kk \SU_4(2)$-module. Its Brauer
character $\phi_S$ then satisfies $\Res^{\GU_6(2)}_{\SU_4(2)}\phi_S=m.\phi_{25}$ where $m\geq 1$ is an integer and $\phi_{25}$ is a Brauer $3$-modular character of degree 25. 

In the table of Brauer characters of $\GU_6(2)$ (denoted by $3.U_6(2).3$ in the notations 
of \cite{ATLASBRAUER}), it should then appear as a character of degree $25m$ and with values in $m\OOO$ for 
the classes of elements of $\SU_4(2)\subset \GU_6(2)$ ($\OOO$ denotes the ring of integers of the $3$-adic ring 
of a splitting field of $\GU_6(2)$).

Since the publication of \cite{ATLASBRAUER}, this table has been computed and made available in \GAPq\ 
(package {\tt CTblLib 1.1.3}), so we can check that
only two characters match the condition on degree, and it is for $m=111$ and $154$. But the condition on values is satisfied in neither case.
(see table \ref{3U63m3}).

\begin{table}
\begin{center}
\resizebox{12cm}{!}{\includegraphics{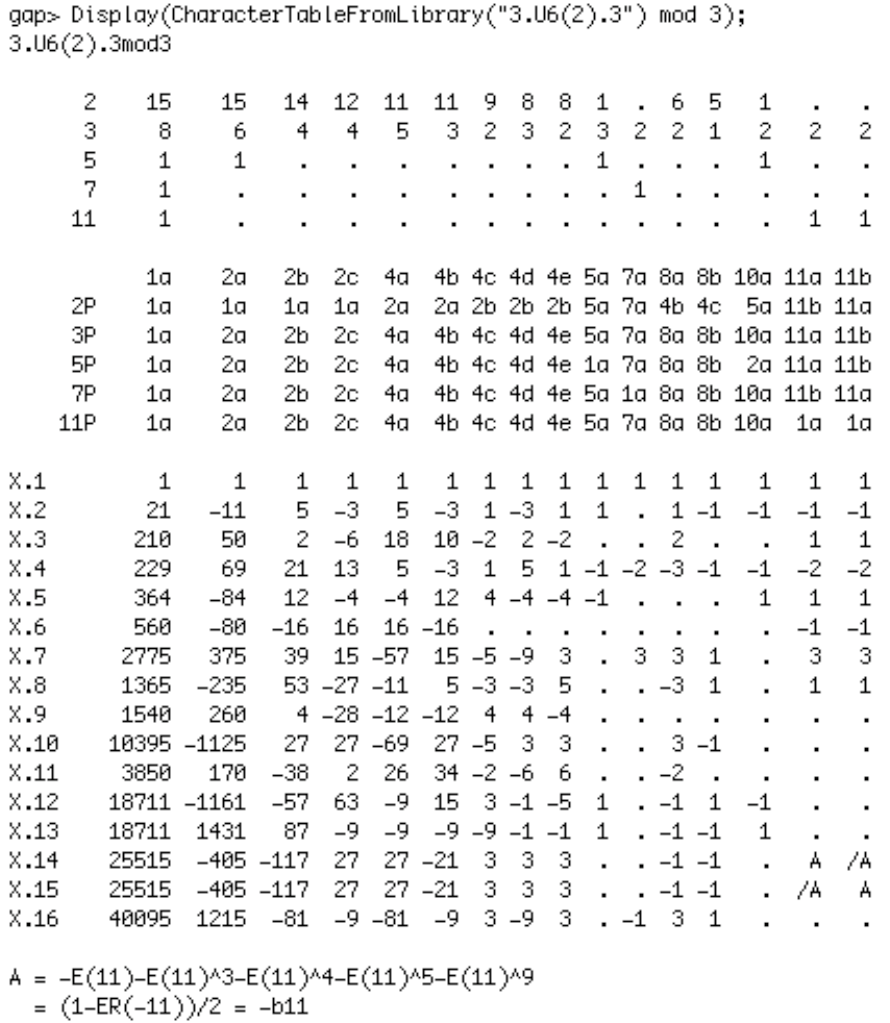}}
\end{center}
\caption{Brauer character table for $3.U_6(3).3$, after \GAPq}
\label{3U63m3}
\end{table}

\bigskip

\subsection{Even characteristic}

Here we choose another equivalent description of $\AAA_n$ and introduce a new element $\bb\in \Z\Gamma_3$ that will prove important 
to our study of characteristic 2.

\begin{defi} \label{defb} Let $$\bb = s_1 s_2^{-1} + s_2^{-1} s_1 + s_1^{-1} s_2 + s_2 s_1^{-1}$$

\end{defi}

Note that $\bb + s_1 s_2^{-1} \bb = \qq$, and in particular $(\qq) \subset (\bb)$.

In characteristic 2, we will not get a complete description of $\kk \AAA_n$. This section is devoted to
the description of $\kk \AAA_n$ for $n \in \{ 3, 4 \}$, and to preliminary results on the ideal generated by $b$.
We will prove the following, letting $z \mapsto \bar{z}$
denote the element $z \mapsto z^2$ of $\mathrm{Gal}(\overline{\F_2}/\F_2)$.

\begin{theo}
If $k$ is a field of characteristic $2$, then 
\begin{enumerate}
\item[(i)] $\kk \AAA_3\cong \kk \Gamma_3 / J(\kk \Gamma_3)^4 \cong (\kk Q / J(\kk Q)^4) \rtimes C_3  $.
\item[(ii)] When $\kk \supset \F_4$, $\kk \AAA_4\cong \kk \AAA_3 \oplus \Mat_3(\kk \Gamma_3/I_q) \oplus \Mat_3(\kk \Gamma_3/\overline{I_q})$
with $I_q = M_q C_3 \subset \kk \Gamma_3$, $M_q$ a 4-dimensional ideal of
$\kk Q$ with $J(\kk Q)^3 \subset M_q \subset J(\kk Q)^2$,
$M_q + \overline{M_q} = J(\kk Q)^2$.
\end{enumerate}
\end{theo}

For $n = 3$ this is a consequence of the following.
\begin{prop}\label{radicals} Keep $k$ of characteristic 2. Then $J(k\Gamma_3)^4=(\qq)\subset (\bb)=J(k\Gamma_3)^3$.
\end{prop}

\begin{proof}

As before, we let $z = (s_1 s_2)^3$, $\ii = s_2 s_1^{-1} z^{-1}$,
$s_2^2 s_1= \kkb = \ii\jj \in Q$. We have $\bb = [s_1,s_2^2]+[s_1^2,s_2] = (\ii + \ii\jj)(1+z)$, hence
$\sigma_{\jj} = \sum_{ x\in {<} \jj { >}} x = \ii \bb \in (\bb)$ and similarly
$\sigma_{\ii} = (\ii\bb)^{s_1^2}, \sigma_{\ii \jj} = (\ii\bb)^{s_1} \in (\bb)$. Let $K = \kk \sigma_{\ii} \oplus \kk \sigma_{\jj}
\oplus \kk \sigma_{\ii\jj} \subset \kk Q$. It is easily checked to be
a 2-sided ideal, stable under $s_1$-conjugation. Since $Q $ is a 2-group,
the Jacobson radical $J(\kk Q )$ is the 7-dimensional augmentation
ideal, and in particular $1+\ii \in J(\kk Q )$. By Jennings theorem (see
\cite{BENSONBOOK} thm. 3.14.6) one easily gets that
$\sum_{r \geq 0} t^r \dim_{\kk} J(\kk Q )^r / J(\kk Q )^{r+1} = 1 + 2 t + 2 t^2 + 2t^3 + t^4$
hence $J(\kk Q )^5 = 0$,
$\dim J(\kk Q )^4 = 1$, $\dim J(\kk Q )^3 = 3$ and $\dim J(\kk Q )^2 = 5$. In particular
$J(\kk Q )^4$
coincides with the simple submodule $\kk \qq$. 
We have $\sigma_x = (1+x)^3$ for $x \in \{\ii, \jj , \ii\jj \}$, so $K \subset
J(\kk Q )^3$ hence $K = J(\kk Q )^3$ by equality of dimensions.
The ideal $J(\kk Q )$ of $\kk Q $ being stable under $s_1$-conjugation,
we get that $J(\kk Q )C_3 = C_3 J(\kk Q )$ is an ideal of $\kk \Gamma_3
= \kk Q  \ltimes C_3$ with $(J(\kk Q )C_3)^5 = 0$ hence $J(\kk Q )C_3
\subset J(\kk \Gamma_3)$. We have $\dim J(\kk Q )C_3 = 21$
and $\dim J(\kk \Gamma_3) = 24 -3=21$ because $\kk \Gamma_3$
admits 3 simple 1-dimensional modules (when $\kk \supset \F_4$),
hence $J(\kk Q )C_3 = J(\kk \Gamma_3)$.

From $J(\kk Q )C_3 = C_3 J(\kk Q )$
and $C_3 C_3 = C_3$ we get that $J(\kk \Gamma_3)^n = J(\kk Q )^n C_3$.
It follows that the ideal $(\bb)$ in $\kk \Gamma_3$ is
$K C_3 = J(\kk Q )^3 C_3 = J(\kk \Gamma_3)^3$, and the one generated
by $\qq$ is $\kk \qq. C_3 = J(\kk Q )^4 C_3 =  J(\kk \Gamma_3)^4$.
\end{proof}

We now consider the case $n = 4$. We let $K$ denote the kernel of
the natural morphism $\Gamma_4 \onto \Gamma_3$. It is the
extra-special group $3^{1+2}$ with exponent 3.
Generators are given by $a=s_1 s_3^{-1}$,$u= (s_1 s_2^2 s_1)^{-1}(s_3 s_2^2 s_3)$,
$\zeta = (s_1 s_2 s_3)^4 \in Z(\Gamma_4)$, and we have $(a,u) = aua^{-1}u^{-1} = \zeta$.
The action of $\Gamma_3$ on $K$ is given by $s_1 u s_1^{-1} = a u$,
$s_2 a s_2^{-1} = u^{-1} \zeta a$, $(s_1,a) = (s_2,u) = (s_1,\zeta) = (s_2,\zeta) = 1$.

We assume $\kk \supset \F_4$. The irreducible representations of $K$ are
defined over $\kk$. Choosing $j \in \F_4 \setminus \F_2$, an irreducible
3-dimensional representation $R : K \to \GL_3(\F_4)$ is given by
$$
a \mapsto \begin{pmatrix} 1 & 0 & 0 \\ j & j^2 & 0 \\ 1 & 1 & j
\end{pmatrix}
u \mapsto \begin{pmatrix} 1 & 1 & 1 \\ 0 & 1 & 1 \\ 0 & 1 & 0
\end{pmatrix}
\zeta \mapsto j^2
$$
and another one is afforded by its Galois conjugate $\bar{R}$. Then
$\kk K = \kk^9 \oplus \Mat_3(\kk) \oplus \Mat_3(\kk)$,
and $\kk \Gamma_4 = (\kk^9 \rtimes \Gamma_3) \oplus (\Mat_3(\kk) \rtimes_{R} \Gamma_3)
\oplus (\Mat_3(\kk) \rtimes_{\bar{R}} \Gamma_3)$.
We will prove that $\Mat_3(\kk) \rtimes \Gamma_3 \simeq \Mat_3(\kk \Gamma_3)$
and describe an explicit isomorphism. For $g \in \Gamma_3$ we
denote $R^g : x \mapsto R(gxg^{-1})$. From $R^g \simeq R$ for every
$g \in \Gamma_3$ we get a projective representation $\rho : \Gamma_3 \to
\PGL_3(\kk)$ ; by explicit computations we check that this $\rho$
can be lifted to a linear representation $\tilde{\rho} : \Gamma_3 \to \GL_3(\kk)$
given by
$$
s_1 \mapsto \begin{pmatrix} 1 & 0 & 0 \\
j & j^2 & 0 \\ j & j & 1 \end{pmatrix} \ \ 
s_2 \mapsto \begin{pmatrix} j^2 & 0 & j \\
0 & 1 & 0 \\ 0 & 0 & 1 \end{pmatrix} \ \ 
$$
Then, an explicit isomorphism $\Mat_3(\kk) \rtimes \Gamma_3 \to \Mat_3(\kk) \otimes \kk \Gamma_3$
is given by $1 \otimes g \mapsto \tilde{\rho}(g)\otimes g$. The ideal generated
by $\qq \in \kk \Gamma_3$ in $\Mat_3(\kk) \rtimes \Gamma_3$
then corresponds to $\Mat_3(I_q)$ with $I_q$ the ideal generated in $\kk \Gamma_3$
by the entries of $\sum_{x \in Q } \rho(x)$. By computer we find that
$I_q$ has dimension 12, and is generated by $s_1^{-1} s_2 s_1 + j^2 s_1 s_2^{-1} s_1 +
j^2 s_2 s_1^{-1} s_2 + j s_2^{-1} s_1^{-1} + j^2 s_1^{-1} s_2^{-1}$.
We also check that $I_q = M_qC_3$ with
$M_q$ the 4-dimensional ideal of $\kk Q $ generated by
$1 + j s_1 s_2 s_1 + j^2 (s_1 s_2)^3 + j s_1 s_2^{-1} + j s_2^{-1} s_1$,
that $J(\kk Q )^3 \subset M_q \subset J(\kk Q )^2$.

Similarly, we get that the ideal generated by $\bb$ in $\Mat_3(\kk) \rtimes \Gamma_3$
corresponds to $\Mat_3(I_b)$ with $I_b$ an ideal of dimension 21
that contains $s_1^{-1} s_2 + 1$. since $(\kk \Gamma_3)/
(s_1^{-1} s_2 + 1) = \kk (\Gamma_3/s_1^{-1} s_2) \simeq \kk C_3$
we get $I_b = (s_1^{-1} s_2 + 1)$ and $\Mat_3(\kk \Gamma_3) / \Mat_3(I_b)
\simeq \Mat_3(\kk C_3)$.
\begin{lemma} \label{lemr1r2} The images of the elements $r_1 = s_2 s_3^2 + s_1^2 s_2 + s_1 s_2^2 + s_3 s_1^2 + s_2^2 s_3 + s_1 s_3^2$
and $r_2 = s_2^2 s_3 + s_1 + s_2 + s_2 s_3 s_1^2 + s_2^2 s_3 s_1 + s_1^2 s_3^2$
of $\kk \Gamma_4$ inside $\Mat_3(\kk) \rtimes \Gamma_3$ 
lie inside the image of $(\bb)$.
\end{lemma}
\begin{proof}
We first write $r_1,r_2$ inside $\kk K \rtimes \Gamma_3$. We
get $r_1 = u^{-1} \zeta a s_2 s_1^2 + s_1^2 s_2 + s_1 s_2^2 + a^{-1} + 
au^{-1} a s_2^2 s_1 + a$
and $r_2 = \zeta^{-1}ua s_2^2s_1^2 + s_1 + s_2 + aua s_2 + au^{-1} a s_2^2 s_1^2 + a s_1$.
We need to prove that they map to $0$ through the composite of the morphisms
$\kk K \rtimes \Gamma_3 \to \Mat_3(\kk) \rtimes \Gamma_3 \to \Mat_3(\kk) \otimes \kk \Gamma_3 \onto \Mat_3(\kk C_3)$,
that is that $ R(u^{-1} \zeta a) \tilde{\rho}(s_2 s_1^2) + \tilde{\rho}(s_1^2 s_2) + \tilde{\rho}(s_1 s_2^2) + R(a^{-1}) + 
R(au^{-1} a) \tilde{\rho}(s_2^2 s_1) + R(a) = 0$ and $R(\zeta^{-1}ua) \tilde{\rho}(s_2^2s_1^2) + \tilde{\rho}(s_1) +
\tilde{\rho}(s_2) + R(aua) \tilde{\rho}(s_2) + R(au^{-1} a)
\tilde{\rho}(s_2^2 s_1^2) + R(a) \tilde{\rho}(s_1) = 0$. This follows from a straightforward computation.
\end{proof}

The case of the other 3-dimensional representations is similar
and can moreover be deduced from the first one by Galois action :
letting $x \mapsto \bar{x}$ denote the nontrivial element of
$\mathrm{Gal}(\F_4/\F_2)$, $(\qq)$
corresponds to the ideal $\overline{I_q} = \overline{M_q} C_3$,
and we check $M_q + \overline{M_q} = J(\kk Q )^2$.

We now turn to the 1-dimensional representations $\rho_{\alpha,\beta} : K
\to \F_4^{\times}$ defined by $a \mapsto j^{\alpha}, u \mapsto j^{\beta},
\zeta \mapsto 1$
for $\alpha,\beta \in \{ 0,1,2 \}$. We have $\rho_{\alpha,\beta}^{s_1}(a)
= j^{\alpha}$, $\rho_{\alpha,\beta}^{s_1}(u)
= j^{\alpha+\beta}$, $\rho_{\alpha,\beta}^{s_2}(a)
= j^{\alpha-\beta}$, $\rho_{\alpha,\beta}^{s_2}(u)
= j^{\beta}$. Identifying the possible $(\alpha,\beta)$ with $\F_3^2$,
the $\Gamma_3$-action on the classes of representations thus corresponds to
the identification of $\Gamma_3$ with $\SL_2(\F_3)$ given by
$$
s_1 \mapsto \begin{pmatrix} 1 & 0 \\ 1 & 1 \end{pmatrix} \ \ 
s_2 \mapsto \begin{pmatrix} 1 & -1 \\ 0 & 1 \end{pmatrix}
$$
It follows that there are two orbits, of cardinalities 1 and 8. We have $\kk^9 \rtimes \Gamma_3 =
\kk \Gamma_3 \oplus \kk^{\Gamma_3/C} \rtimes \Gamma_3$ with $C$
the stabilizer of a nonzero vector in $\F_3^2$.

We apply Proposition \ref{wigner}  with $C = < s_1 >$
and $ Q $ making a representative system of $\Gamma_3/C$. Then, under the isomorphism
$\kk^{\Gamma_3/C} \rtimes \Gamma_3 \simeq \Mat_{Q }(\kk C_3)$,
$g \in Q $ is mapped to $\sum_{u \in Q } E_{gu,u}$. In particular,
$\qq \in \kk Q $ is mapped to
$$
\sum_{v \in Q } \sum_{u \in Q } E_{vu,u} = \sum_{u,v \in Q } E_{u,v}.
$$
The ideal of $\kk^{\Gamma_3/C} \rtimes \Gamma_3$
generated by $\qq$ is then mapped to $\Mat_8(I)$ for $I$ the ideal
of $\kk C_3$ generated by $1$, hence is the full block
$\kk^{\Gamma_3/C} \rtimes \Gamma_3$. Since $\qq \in (\bb)$, the same holds
for the ideal generated by $\bb$.

\begin{prop} \label{proprelsGamma4} The elements $r_1, r_2 \in \kk \Gamma_4$ of Lemma \ref{lemr1r2} belong
to the 2-sided ideal generated by $\bb$.
\end{prop}
\begin{proof}
We showed that $\kk \Gamma_4 / (\bb)$ is isomorphic to
$(\kk \Gamma_3/(\bb)) \oplus \Mat_3(\kk C_3) \oplus \Mat_3(\kk C_3)$.
The images of $r_1,r_2$ in both $\Mat_3(\kk C_3)$ is $0$ by Lemma
\ref{lemr1r2}, and it is readily checked that $r_1 \mapsto \bb$
and $r_2 \mapsto 0$ through $\kk \Gamma_4 \onto \kk \Gamma_3$. The conclusion follows.
\end{proof}

\subsection{A finer description of $K_5$ as a $\Z$-module}

Similar algorithms as the ones used in the proof of proposition \ref{dimA5p2} enabled us,
using several months of CPU time, to determine the structure of
$(\Z/32 \Z) K_5$ as a $(\Z/32\Z)$-module. Combined with our study of odd characteristic, this implies the
following.

\begin{prop} \label{propnew} As a $\Z$-module, $K_5 \simeq (K_5^0)^3$
with
$$
K_5^0 \simeq (\Z/2\Z)^{744} \times (\Z/4\Z)^{38} \times (\Z/8\Z)^{80} \times
(\Z/16\Z) \times G
$$
where $G$ is an abelian $3$-group with $\dim_{\F_3} G \otimes \F_3 = 25^2= 625$.
\end{prop}

\section{A ternary Hecke algebra in characteristic 2}

We assume that $k$ is a field of characteristic 2 with $k \supset \F_4 = \{ 0, 1, j,j^2 \}$. Recall that $\bb$ and $\qq$ are defined in 
Definitions \ref{defb} and \ref{defcq}.

\begin{defi} For $\alpha ,\beta\in k$, and $n\geq 3$, we define the following.

Let $J_n(\alpha,\beta) = k 
\Gamma_n . (s_1- \alpha)(s_1-\beta) . k \Gamma_n$ 

Let
$H_n(\alpha,\beta)= k \Gamma_n/J_n(\alpha,\beta)$. 

Let  $J_n = J_n(1,j) \cap J_n(1,j^2) \cap J_n(j,j^2)$.
\end{defi}

The aim of this section is essentially to prove the following. In particular, we see that $k\AAA_n$ never collapses 
and actually has dimension $\geq 3(n!-1)$. Recall that $\qq\in k\Gamma_n.\bb.k\Gamma_n$ (see Proposition~\ref{radicals})

\begin{theo} Let $n\geq 3$. Then $J_n=k\Gamma_n.\bb.k\Gamma_n$ as a 2-sided 
ideal of $k\Gamma_n$ and $k\Gamma_n/J_n$ has dimension $3(n!-1)$.

\end{theo}

Notice that $J_n(\alpha ,\beta )$ contains
$(s_i - \alpha)(s_i - \beta)$ for arbitrary $1\leq i< n $.

\begin{lemma} \label{lemAtoH}  Assume $n \geq 3$. We have $I_n \subset J_n(\alpha,\beta)$ 
whenever $\alpha^3 = \beta^3 = 1$ and $\alpha \neq \beta$.
\end{lemma}
\begin{proof}
We need to show that $\cc \equiv 0$ modulo $J_n(\alpha,\beta)$. From 
$s_1^2 \equiv (\alpha  + \beta) s_1 - \alpha \beta$
we get
$s_2 s_1^2 s_2 \equiv (\alpha  + \beta) s_2s_1s_2 - \alpha \beta s_2 ^2 \equiv
(\alpha  + \beta) s_2s_1s_2 - \alpha \beta (\alpha + \beta) s_2 + (\alpha \beta)^2
$
and symmetrically
$
s_1 s_2^2 s_1 \equiv  (\alpha  + \beta) s_1s_2s_1 - \alpha \beta (\alpha + \beta) s_1 + (\alpha \beta)^2
$, thus
$s_1 s_2^2 s_1 + s_2 s_1^2 s_2 =  2(\alpha + \beta) s_1 s_2 s_1 - \alpha \beta (\alpha+ \beta)(s_1 + s_2) + 2 \alpha^2 \beta^2$.
From the same equation we get $s_1^2 s_2 s_1 \equiv (\alpha  + \beta) s_1s_2 s_1 -\alpha \beta s_2 s_1$
and $s_1 s_2 s_1^2 \equiv (\alpha  + \beta) s_1 s_2 s_1 - \alpha \beta s_1 s_2$ hence
$s_1^2 s_2 s_1 + s_1 s_2 s_1^2 \equiv 2 \alpha \beta s_1 s_2 s_1 - \alpha \beta (s_1 s_2+s_2 s_1)$. 
Finally
$
s_1 ^2 s_2^2 \equiv \left( (\alpha + \beta) s_1 - \alpha \beta \right)\left( (\alpha + \beta) s_2 - 
\alpha \beta \right)
\equiv (\alpha + \beta)^2 s_1 s_2 - \alpha \beta(\alpha + \beta) (s_1 + s_2) + (\alpha \beta)^2
$
and symetrically  $s_2^2 s_1^2 \equiv (\alpha + \beta)^2 s_2 s_1 - \alpha \beta(\alpha + \beta)(s_1 + s_2) + (\alpha \beta)^2$.
Altogether this yields
$$
\cc \equiv 4 (\alpha + \beta) s_1 s_2 s_1 + \left((\alpha+\beta)^2 - \alpha \beta \right) (s_2 s_1 + s_1 s_2) +
\left(1- 3 \alpha \beta(\alpha + \beta)\right) (s_1 + s_2) + 4 \alpha^2 \beta^2.
$$
Since $(\alpha + \beta)^2 - \alpha \beta = \alpha^2 + \alpha \beta + \beta^2 = 0$ and
$\alpha \beta(\alpha + \beta) = (\beta/\alpha) + (\alpha/\beta) = -1$
we get $\cc \equiv 4(\alpha + \beta)s_1 s_2 s_1 + 4 (s_1 + s_2) + 4 \alpha^2 \beta^2$.
Since $4 = 0$ this concludes the proof.
\end{proof}

Recall $J_n = J_n(1,j) \cap J_n(1,j^2) \cap J_n(j,j^2)$. From the above lemma, $J_n\supset I_n$ and obviously $k \AAA_n$ surjects onto $k \Gamma_n / J_n$,
while $k \Gamma_n / J_n$ embeds into $H_n(1,j) \times H_n(1,j^2) \times H_n(j,j^2)$.

In order to deal with quotients of an intersection of three ideals we
will need, here and later on, the following two lemmas.

\begin{lemma} \label{lemIIIdistrib} Let $A$ be a (possibly non-commutative) unital ring, $I_1,I_2,I_3$ three 2-sided ideals, such that $A = I_1 + I_2+I_3$.
Then $I_1 + I_2 \cap I_3 = (I_1 + I_2) \cap (I_1 + I_3)$.
\end{lemma}
\begin{proof}

Denote $I=(I_1+I_2)\cap (I_1+I_3)$. Then the inclusion $I_1+(I_2\cap I_3)\subseteq I$ is trivial. 
On the other hand, we have $I=I(I_1+I_2+I_3)=I(I_1+I_2)+I.I_3\subseteq (I_1+I_3).(I_1+I_2)+(I_1+I_2).I_3\subseteq I_1+I_3I_2+I_2I_3\subseteq I_1+(I_2\cap I_3)$.

\end{proof}

\begin{lemma} \label{lemchinoistop} Let $A$ be an abelian group, $I,J,K$ subgroups of $A$ with
$I+J+K = A$. We define morphisms
{}
$$
A/I \cap J \cap K  \stackrel{d_1}{\longrightarrow} A/I \times A/J \times A/K \stackrel{d_2}{\longrightarrow} 
A/(J+K) \times A/(I + K) \times A/(I+J)
$$

where $d_2$ is induced by $(a,b,c) \mapsto (b-c,a-c,a-b)$ and $d_1$
is the natural (injective) map. Then $d_2 \circ d_1 = 0$, $d_2$
is surjective and $\Ker d_2 / \Image d_1 \simeq (K+I ) \cap (K+J) /
K+ I \cap J$.
\end{lemma}

\begin{proof} $d_2 \circ d_1 = 0$ is clear. $\Image d_2$ contains
$A/(I+J) = 0 \times 0 \times A/(I+J) $,
as $A/(I+J) = (I+J+K)/(I+J)$ is clearly $d_2(K/I \times 0 \times 0)$,
where $K/I$ denotes the image of $K$ in $A/I$, hence $d_2$ is surjective
by symmetry. An element of $\Ker d_2$ is the class of
a triple $(a,b,c) \in A^3$ with $a-b = i+j$, $b-c = j'+k$, $a-c = i' + k'$
for some $i,i' \in I$, $j,j' \in J$, $k,k' \in K$, hence
of a $(a-i,b+j,c) = (a',a',c)$ with $a' = a-i = b+j$. One has $a'-c =
b-c+j = a-i - c \in (K+I) \cap (K+J)$. Conversely,
the class of any $(a,a,c)$ with $a-c \in (K+I) \cap (K+J)$ belongs to
$\Ker d_2$.

On the other hand, such a triple $(a,a,c)$ originates from $A$ iff
there exists $i \in I$, $j \in J$, $k \in K$ such that $a+i = a+j = c+k$,
which means $c \in K + I \cap J$. This proves $\Ker d_2 / \Image d_1
\simeq (K+I) \cap (K+J) /(K + I \cap J)$ under $(a,a,c) \mapsto a-c$.
\end{proof}

When $\alpha, \beta \in \mu_3(\kk)$ with $\alpha \neq \beta$,
we let $q_{\alpha} : H_n(\alpha,\beta) \to \kk$ denote the
natural morphism sending each $s_i$ to $\alpha$.
$$
\xymatrix{
H_n(j,j^2) \ar[d]_{q_{j^2}} \ar[dr]^{q_{j}} & H_n(1,j^2) \ar[dl]^{q_{j^2}} \ar[dr]_{q_{1}} & H_n(1,j) \ar[dl]_{q_{j}} \ar[d]^{q_1} \\
\kk & \kk & \kk 
}
$$

\begin{defi}
We let $\mathcal{H}_n$ denote the subalgebra of $H_n(j,j^2) \oplus H_n(1,j^2) \oplus
H_n(1,j)$ made of the triples $(x_{1},x_{j},x_{j^2})$
such that $q_{\alpha}(x_{\alpha'}) = q_{\alpha}(x_{\alpha''})$
whenever $\{ \alpha,\alpha', \alpha'' \} = \mu_3(\kk)$.

We also denote $U_n=k\Gamma_n/(\bb)$ (see Definition \ref{defb}).

\end{defi}

Recall that $J_n = J_n(j,j^2) \cap J_n(1,j^2) \cap J_n(1,j)$. 

\begin{prop}\label{ternary} The image of the natural embedding 
$\kk \Gamma_n / J_n \into H_n(j,j^2) \oplus H_n(1,j^2) \oplus
H_n(1,j)$ is $\mathcal{H}_n$. We have $\dim \mathcal{H}_n = 3(n!-1)$ and
$\kk \AAA_n \onto \mathcal{H}_n$.
\end{prop}
\begin{proof}

We want to apply lemmas \ref{lemIIIdistrib} and \ref{lemchinoistop} to the ideals
$J_n(\alpha,\beta)$ for $\alpha,\beta$ distinct third roots of 1. If $\mu_3(\kk)
= \{ \alpha,\beta,\gamma \}$, we first prove  $J_n(\alpha , \beta )+ J_n(\alpha , \gamma )
=k\Gamma_n (s_1-\alpha )k\Gamma_n$ : under $\varphi$, we can assume for this $\alpha = 1$,
$\beta = j$, $\gamma = j^2$ and we have $(s_1-1)(s_1-j)-(s_1-1)(s_1-j^2) = (1+2j)(s_1-1)$
with $1+2j$ invertible.
This implies at once that the sum of the three $J_n(\alpha,\beta)$
have for sum $\kk \Gamma_n$. We also get that the map $d_2$ of \ref{lemchinoistop}
defines $\mathcal{H}_n$ as its kernel. We thus get the first statement of our
proposition by applying Lemma \ref{lemchinoistop} in a case where Lemma \ref{lemIIIdistrib}
ensures that the sequence of maps is exact. This exactness also implies the claim on dimensions
since each Hecke algebra has dimension $n!$.
\end{proof}

\begin{remark} Lemma \ref{lemAtoH} and Proposition \ref{ternary} hold true with $\kk$ replaced by
$(\Z/4\Z)[j] = (\Z/4\Z)[x]/(x^2+x+1)$, with the same proofs (except, of course, for the statement on the dimension).
\end{remark}

\begin{remark}  We have natural morphisms $\mathcal{H}_n \onto H_n(\alpha,\beta)$,
hence every simple $H_n(\alpha,\beta)$-module $M$ provides a simple module
for $\mathcal{H}_n$, and for $\kk \Gamma_n$. Since $char. \kk = 2$ and $s_1 \in \Gamma_n$
has order 3, the action of $s_1$ is semisimple, and so
we can assume that the induced morphism $\kk \Gamma_n \to
\End_{\kk}(M)$ either factorizes through $K_n(\alpha)$ (up to exchanging $\alpha$
and $\beta$) or does not factorize through any of the $J_n(u,v)$
for $\{u,v \} \neq \{ \alpha, \beta \}$. It follows that
a collection of non-isomorphic simple modules for $\mathcal{H}_n$
is afforded by the simple $H_n(\alpha,\beta)$-modules of dimension
at least 2 and the three 1-dimensional modules defined by $s_1 \mapsto \alpha$
for a given $\alpha \in \mu_3(\kk)$. By the same argument, the same holds for
indecomposable modules as well.
\end{remark}

\def\ss{s}

\begin{prop} For $n \geq 3$, $J_n$ is generated as a 2-sided
ideal by $\bb$ (see Definition \ref{defb}).
\end{prop}

It is easily checked that, for $n \geq 3$, $\bb \in J_n$. Indeed,
we have $\bb \in J_n(\alpha,\beta)$ because the image
of $\bb= \ss_1 \ss_2^{2} +\ss_2^{2} \ss_1 + \ss_1^{2} \ss_2 + \ss_2 \ss_1^{2}$
in $H_n(\alpha,\beta) = \kk \Gamma_n / J_n(\alpha,\beta)$
is $2(\alpha + \beta) (\ss_1 \ss_2 + \ss_2 \ss_1) + 2 \alpha \beta(\ss_1 + \ss_2) \equiv 0$,
using $s_i^2 \equiv (\alpha + \beta) s_i + \alpha \beta$ and char. $\kk$ = 2.

Recall $U_n = \kk \Gamma_n / (\bb)$. In view of Proposition \ref{ternary}, in order to prove $J_n = (\bb)$,
it is enough to check that $\dim U_n \leq \dim \kk \Gamma_n / J_n
= 3 (n! - 1)$. We need a lemma, where we abuse notations
by letting $U_k$ denote the image of $U_k$ in $U_{n+1}$.

\begin{lemma} \label{lemUnUnp1}For $n \geq 2$, one has
\begin{enumerate}
\item $U_{n+1} = U_n + U_n \ss_n U_n + U_n \ss_n^2 U_n$
\item $U_{n+1} = U_n + U_n \ss_n U_n + U_n \ss_n^2$
\item If $k < n$, $r,t \in \{0,1,2 \}$ we have
$s_k^r s_1^t s_n^2 \equiv s_1^{r+t} s_n^2$ modulo $U_n + U_n s_n$.
\item $U_{n+1} = U_n + U_n \ss_n U_n + U_2 \ss_n^2$
\end{enumerate}
\end{lemma}
\begin{proof}
Item (i) is a consequence of $\qq \in (\bb)$ by Proposition \ref{propfunar}.

We now prove (ii). One has
$\ss_2^2 \ss_1 = \ss_1 \ss_2^2 + \ss_1^2 \ss_2 + \ss_2 \ss_1^2 + \bb$.
By (1), $U_{n+1}$ is spanned by $U_n$, $U_n \ss_n U_n$ and the
$w_1 \ss_n^2 w_2$ for $w_1,w_2$ positive words in the $\ss_i$
for $i \leq n-1$. We let $l(w_2)$ denote the length of
$w_2$ with respect to these generators. 
If, as a word, $w_2 = \ss_r w_2'$ with $r\leq n-2$, then
$w_1 \ss_n^2 w_2 = w_1 \ss_n^2 \ss_r w_2' = 
w_1 \ss_r \ss_n^2 w_2'$ with $l(w'_2) < l(w_2)$.
If $w_2 = 1$ is the empty word, then $w_1 \ss_n^2 w_2 \in U_n \ss_n^2$.
Otherwise, we have $w_2 = \ss_{n-1} w_2'$. By conjugating $\bb$,
we get $\ss_n^2 \ss_{n-1} \equiv \ss_{n-1} \ss_n^2 + \ss_{n-1}^2 \ss_n +
\ss_n \ss_{n-1}^2 \mod (\bb)$ hence
$$w_1\ss_n^2 \ss_{n-1}w'_2 \equiv w_1\ss_{n-1} \ss_n^2w'_2 + w_1\ss_{n-1}^2 \ss_n w'_2+
w_1\ss_n \ss_{n-1}^2 w'_2 \mod (\bb)
$$
On the other hand, $l(w'_2) < l(w_2)$ and $w_1\ss_{n-1}^2 \ss_n w'_2+
w_1\ss_n \ss_{n-1}^2 w'_2  \in U_n \ss_n U_n$, so we
can conclude by induction on the length of $w_2$.

We first note that (iii) is trivial for $n = 2$, so we assume $n \geq 3$. It is
also trivial for $r = 0$, so we can assume $r \in \{ 1 , 2 \}$.
We first deal with the case $t = 0$.
We let $V_n = U_n + U_n s_n$ and
we use that, in $\kk \Gamma_4/(\bb)$,
$s_2 s_3^2 = 
(s_1^2 s_2 + s_1 s_2^2) + (s_1^2s_3  + s_2^2 s_3) + s_1 s_3^2$
and
$s_2^2 s_3^2 = 
(s_1 + s_2) + (s_2 s_1^2s_3  + s_2^2 s_1s_3 ) + s_1^2 s_3^2$
(see Proposition \ref{proprelsGamma4}).
Here and in the following, all congruences are modulo additive subgroups.
By conjugation we thus get 
$s_{n-1} s_n^2 \equiv (s_{n-2}^2 s_{n-1} + s_{n-2} s_{n-1}^2) +
(s_n s_{n-2}^2 + s_{n-1}^2 s_n) + s_{n-2} s_n^2 \mod (\bb)$ 
whenever $n \geq 3$,
and in particular $s_{n-1} s_n^2 \equiv s_{n-2} s_n^2$ modulo $V_n$ and
also $s_{n-1}^2 s_n^2 \equiv s_{n-2}^2 s_n^2$ modulo $V_n$. We need to
prove that
$s_k^r s_n^2 \equiv s_1^r s_n^2 \mod V_n$ for all $k < n$ and
$r \in \{ 1, 2 \}$, or, equivalently, that
$s_{k}^r s_n^2 \equiv s_{k+1}^r s_n^2 \mod V_n$ for all $k < n-1$ and $r \in \{ 1, 2 \}$.
We prove this by decreasing induction, the case $k = n-2$ being already
known. Let now $k < n-2$. Notice that $V_n$ is both a left $U_n$-module
and a $U_{n-1}$-bimodule. Modulo $V_n$, we have by the induction hypothesis
and the commutation relations that
$s_{k+1}^a s_k^b s_n^2 \equiv s_{k+1}^a s_n^2 s_k^b \equiv
s_{k+2}^a s_n^2 s_k^b \equiv s_k^b s_{k+2}^a s_n^2 \equiv
s_k^b s_{k+1}^a s_n^2$
for all $a,b \in \{1,2 \}$. On the other hand, $s_k s_{k+1} s_k = s_{k+1}
s_k s_{k+1}$ hence
$s_k s_{k+1} s_k s_n^2 \equiv s_k^2 s_{k+1} s_n^2 \equiv s_{k+1} s_k^2 s_n^2$
is equal modulo $V_n$ to
$s_{k+1} s_k s_{k+1} s_n^2 \equiv s_{k+1}^2 s_k s_n^2$. Multiplying
on the left by $s_k^{-1} s_{k+1}^{-1}$ we thus get
$s_k s_n^2 \equiv s_k^{-1} s_{k+1} s_k s_n^2 \equiv 
s_k^{-1} s_k s_{k+1} s_n^2 \equiv s_{k+1} s_n^2$ hence the conclusion for $t = 0$.
For arbitrary $t$, we then have
$s_k^r s_1^t s_n^2
\equiv s_k^r s_k^t s_n^2 \equiv s_k^{r+t} s_n^2 \equiv s_1^{r+t} s_n^2$.

\medskip

Notice that (ii) and (iv) are the same statement for $n = 2$, so
we can again assume $n \geq 3$. Since $U_n + U_n s_n \subset U_n + U_n s_n U_n$,
(iii) implies $U_n s_n^2 \subset U_2 s_n^2 + U_n + U_n s_n U_n$
hence (iv) follows from (ii).

\end{proof}

For $0 \leq k \leq n$, we let $s_{n,k} = s_n s_{n-1} \dots s_{n-k+1}$
with the convention that $s_{n,0} = 1$ and $s_{n,1} = s_n$.
We let $U_n^k = U_n s_{n,k}$ (hence $U_n^0 = U_n$). Similarly,
we let $x_{n,k} = s_n s_{n-1} \dots s_{n-k+2} s_{n-k+1}^2$
for $1 \leq k \leq n$,
with the convention $x_{n,1} = s_n^2$.

\begin{lemma} \label{lems1Un}
\begin{enumerate}
\item If $r \leq n-1$, $1 \leq k \leq n$ and $c \in \{ 0, 1 , 2 \}$, then $s_r s_1^c x_{n,k} \in
s_1^{c+1} x_{n,k} + U_n^0 + \dots + U_n^k$.
\item For $w \in \Gamma_n$, $w x_{n,k} \in s_1^{l(w)}x_{n,k} + U_n^0 + \dots + U_n^n$,
where $l : \Gamma_n \onto \Z/3\Z$ is $s_i \mapsto 1$.
\end{enumerate}
\end{lemma}
\begin{proof}
We first deal with (i). Notice that the statement is trivial for $n \leq 2$, so we
can assume $n \geq 3$ and in particular $s_1 s_n = s_n s_1$.
We prove the statement by induction on $k$, for all $n$. The case
$k = 1$ being known by the previous lemma, we can assume $k \geq 2$.
Let $r \leq n-1$. We first consider the case $r \leq n-2$. Then
$s_r s_1^c x_{n,k} = s_r s_1^c s_n s_{n-1} \dots s_{n-k+1}^2
= s_n s_r s_1^c s_{n-1} \dots s_{n-k+1}^2 = s_n s_r s_1^c  x_{n-1,k-1}$. 
By the induction hypothesis we have $s_r s_1^c x_{n-1,k-1}
\equiv s_1^{c+1} x_{n-1,k-1}$ modulo $U_{n-1}^0 + \dots + U_{n-1}^{k-1}$
hence $s_n s_r s_1^c x_{n-1,k-1} \equiv s_n s_1^{c+1} x_{n-1,k-1}$
modulo $s_nU_{n-1}^0 + \dots + s_nU_{n-1}^{k-1}$.
Noticing that $s_n U_{n-1}^j = s_n U_{n-1} s_{n-1,j} = U_{n-1} s_n s_{n-1,j}
= U_{n-1} s_{n,j+1} \subset U_n s_{n,j+1}$ we get that
$s_n s_r s_1^c x_{n-1,k-1} \equiv s_n s_1^{c+1}  x_{n-1,k-1}\equiv s_1^{c+1} s_n x_{n-1,k-1}
\equiv s_1^{c+1} x_{n,k}$ modulo $U_n + U_n^1 + \dots
+ U_n^k$.

We now consider the case $r = n-1$. For clarity, we let $b
= n-k+1$. Then, using
$s_{b+1} s_b^2 = s_{b+1}^2 s_b + s_b s_{b+1}^2 + s_b^2 s_{b+1}$,
we get that $s_{n-1} s_1^c x_{n,k} = A + B + C$ with
$$
\left\lbrace
\begin{array}{lcl}
A &=& s_{n-1} s_1^c s_n s_{n-1} \dots s_{b+2} s_{b+1}^2 s_b \\
B &=& s_{n-1} s_1^c s_n s_{n-1} \dots s_{b+2} s_b s_{b+1}^2 \\
C &=& s_{n-1}s_1^c  s_n s_{n-1} \dots s_{b+2} s_b^2 s_{b+1}
\end{array} \right.
$$

First note that $C =  s_{n-1} s_1^c s_b^2 s_n s_{n-1} \dots s_{b+2}  s_{b+1}
\in U_n s_{n,k-1} = U_n^{k-1}$.
By the induction hypothesis, $A = (s_{n-1} s_1^c s_n s_{n-1} \dots s_{b+2} s_{b+1}^2) s_b$
is congruent to $(s_{1}^{c+1} s_n s_{n-1} \dots s_{b+2} s_{b+1}^2) s_b = s_1^{c+1} x_{n,k-1} s_b$
modulo $(U_n + U_n^1 + \dots + U_n^{k-1}) s_b \subset
U_n + U_n^1 + \dots + U_n^{k-2} + U_n^k$. Now $s_1^{c+1} x_{n,k-1} s_b = s_1^{c+1}
 s_{n,k-2} s_{b+1}^2 s_b$
and using again $s_{b+1}^2 s_b =  s_{b+1} s_b^2 + s_b s_{b+1}^2 + s_b^2 s_{b+1}$
we get that $$s_1^{c+1} x_{n,k-1} s_b = s_1^{c+1} x_{n,k} + s_1^{c+1} s_n \dots s_{b+2} s_b s_{b+1}^2
+ s_1^{c+1} s_n \dots s_{b+2} s_b^2 s_{b+1}.$$ We have
$s_1^{c+1} s_n \dots s_{b+2} s_b^2 s_{b+1} = s_1^{c+1}  s_b^2 s_{n,k-1} \in U_n^{k-1}$. Moreover, $s_1^{c+1} s_n \dots s_{b+2} s_b s_{b+1}^2 = s_1^{c+1} s_b x_{n,k-1}$, and
by the induction hypothesis, we have $s_b x_{n,k-1} \in s_1 x_{n,k-1} + U_n + \dots 
+ U_n^{k-1}$. Hence
$A \in s_1^{c+1} x_{n,k} + s_1^{c+2} x_{n,k-1} + U_n + \dots + U_n^k$.

We now consider $B$. We have $s_b s_{b+1}^2 \in s_1 s_{b+1}^2 + U_{b+1} + U_{b+1} s_{b+1}$ by Lemma \ref{lemUnUnp1}.
Moreover $s_{n-1} s_1^c s_n \dots s_{b+2} U_{b+1} = s_{n-1} U_{b+1} s_n \dots s_{b+2} \subset
U_n s_n \dots s_{b+2}$ and similarly $s_{n-1} s_1^c s_n \dots s_{b+2} U_{b+1}s_{b+1} \subset U_n s_n \dots s_{b+1}$,
hence $B \in s_{n-1} s_1^c \dots s_{b+2} s_1 s_{b+1}^2 +U_n^{k-2} + U_n^{k-1}$
i.e. $B \in s_{n-1} s_1^{c+1} x_{n,k-1} + U_n^{k-2} + U_n^{k-1}
\subset s_1^{c+2} x_{n,k-1} + U_n + U_n^{1} + \dots +  U_n^{k-1}$ by the induction hypothesis.
Altogether this yields $A+B+C \in s_1^{c+1} x_{n,k} + U_n + \dots + U_n^k$ and the conclusion
for (i).

Part (ii) is an immediate consequence of (i), as we have $s_r x_{n,k} \equiv s_1 x_{n,k}$
and $s_r^2 x_{n,k} = s_r s_r x_{n,k} \equiv s_r s_1 x_{n,k} \equiv s_1^2 x_{n,k}$
modulo $U_n^0+U_n^1+\dots +U_n^n$ whenever $r <n$, and the $s_r$ for $r <n$ generate $\Gamma_n$.
\end{proof}

\begin{prop} Let $n \geq 2$. Then $\dim U_n = 3(n!-1)$ and
$$
U_{n+1} = U_n \oplus U_n^1 \oplus \dots \oplus U_n^n \oplus U_2 x_{n,1} \oplus \dots \oplus U_2 x_{n,n}
$$
\end{prop}
\begin{proof}
We first prove that
$U_{n+1} = U_n + U_n^1 + \dots + U_n^n + U_2 x_{n,1} + \dots + U_2 x_{n,n}$
by induction on $n$. Assuming this to be true for $n$, we have
$U_{n+2} = U_{n+1} + U_{n+1} s_{n+1} U_{n+1} + U_2 x_{n+1,1}$
by Lemma \ref{lemUnUnp1}, and
$$U_{n+1}s_{n+1} U_{n+1} \subset U_{n+1} s_{n+1} (U_n + \dots + U_n^n) +
U_{n+1} s_{n+1} U_2 x_{n,1} + \dots + U_{n+1} s_{n+1} U_2 x_{n,n}.
$$
But, for $k=1,\dots , n$, $U_{n+1} s_{n+1} U_n^k = U_{n+1} s_{n+1} U_n s_{n,k} = 
U_{n+1}  U_n s_{n+1} s_{n,k} = U_{n+1}^{k+1}$
and  $U_{n+1} s_{n+1} U_2 x_{n,k} \subset U_{n+1} s_{n+1} x_{n,k} = U_{n+1} x_{n+1,k+1}$.
Therefore $U_{n+2}\subset \sum_{k=1}^nU_{n+1}x_{n+1,k+1}+\sum_{k=1}^nU_{2}x_{n+1,k}$.
On the other hand, $U_{n+1} x_{n+1,k+1} \subset U_2 x_{n+1,k+1} + U_{n+1} + \dots + U_{n+1}^{n+1}$
by Lemma \ref{lems1Un}. It follows that 
$U_{n+2} \subset U_{n+1} + U_{n+1}^1 + \dots + U_{n+1}^{n+1} + U_2 x_{n+1,1} + \dots + U_2 x_{n+1,n+1}$
and we conclude by induction.

We then prove that $\dim U_n \leq 3(n!-1)$, again by induction on $n$.
Since $U_{n+1} = U_n + U_n^1 + \dots + U_n^n + U_2 x_{n,1} + \dots + U_2 x_{n,n}$,
we get $\dim U_{n+1} \leq (n+1) \dim U_n + 3n \leq 3 (n+1)! - 3(n+1) + 3n = 3( (n+1)! - 1)$.

Finally, since $U_n$ maps onto $\mathcal{H}_n$ we know $\dim U_n \geq 3( n! - 1)$ hence
$\dim U_n = 3(n! -1)$. It follows that all inequalities above are equalities and the sum is direct, which
concludes the proof.
\end{proof}

\section{A ternary Birman-Wenzl algebra}

\begin{figure}
\begin{center}
\resizebox{12cm}{!}{\includegraphics{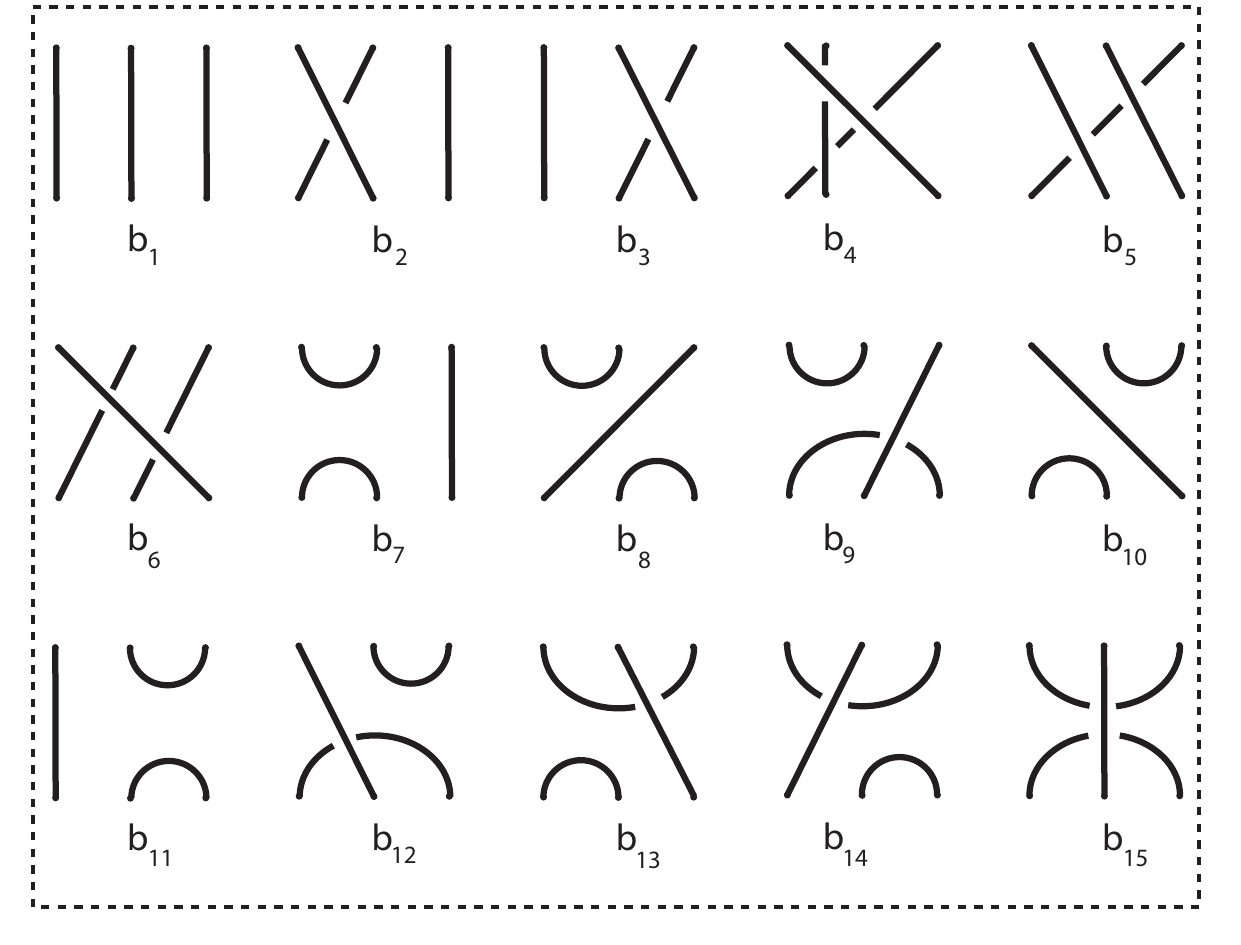}}
\end{center}
\caption{Basis for $BW_3$}
\label{pictsBW3}
\end{figure}

\begin{table}
$$
\begin{array}{|lcl|lcl|}
\hline
1 & \mapsto & b_1 & s_1 s_2^{-1} & \mapsto & b_2 + b_6 + b_8 + b_{11} + b_{14} \\
\hline
s_1^{-1} s_2 & \mapsto & b_3 + b_6 + b_7 + b_8 + b_9 & s_2 s_1^{-1} & \mapsto & b_3 + b_5 + b_{13} \\ 
\hline
s_2^{-1}s_1 & \mapsto & b_2 + b_5 + b_{12} & s_1 s_2 s_1 & \mapsto & b_4 \\
\hline
\end{array}
$$
{}
$$
\begin{array}{|lcl|}
\hline
s_1^{-1}s_2^{-1}s_1^{-1} & \mapsto & b_2 + b_3 + b_4 + b_5 + b_6 + b_9 + b_{10} + b_{14} + b_{15} \\
\hline
s_2^{-1}s_1 s_2^{-1} s_1 & \mapsto & b_1 + b_2 + b_3 + b_5 + b_6 + b_7 + b_{10} + b_{11} + b_{12} + b_{13} + b_{15}  \\
\hline
\end{array}
$$
\caption{The map $\F_4 Q_8 \onto BW_3$}
\label{tablemapQ8BW3}
\end{table}

{}

\begin{table}
$$
\begin{array}{|lcl|lcl|lcl|}
\hline
b_1 s &=& b_2 & b_2 s & = & b_1 + b_2 + b_7 & b_3 s &=& b_5 \\
b_4 s &=& b_4 + b_6 + b_{10} & b_5 s &=& b_3 + b_5 + b_{13} &
b_6 s &=& b_4 \\
b_7 s &=& b_7 & b_8 s &=& b_9 & b_9 s &=& b_7 + b_8 + b_9 \\
b_{10} s &=& b_{10} & b_{11} s &=& b_{12} & b_{12} s &=& b_{10} + b_{11}
+ b_{12}  \\
b_{13} s &=& b_{13} & b_{14} s &=& b_{15} & b_{15} s &=& b_{13} + b_{14} + b_{15} \\
\hline
\end{array}
$$
\caption{Multiplication by $s$ in $BW_3$}
\label{tableMultsBW3}
\end{table}

\subsection{Birman-Wenzl algebras}

If $\kk$ is a ring and $x,\la,q \in \kk^{\times}$, $\delta \in \kk$ with $\delta = q-q^{-1}$, and $x \delta = \delta-\la + \la^{-1}$  the Birman-Wenzl algebra $BW_n$
is defined by generators $s_1^{\pm},\dots,s_{n-1}^{\pm}$, $e_1,\dots,e_{n-1}$ and relations
\begin{enumerate}
\item $s_i s_{i+1} s_i = s_{i+1} s_i s_{i+1}$
\item $s_i s_j = s_j s_i$ for $|j-i| \geq 2$
\item $e_i s_{i-1}^{\pm 1} e_i = \la^{\mp 1} e_i$
\item $s_i - s_i^{-1} = \delta(1-e_i)$
\item $s_i e_j = e_j s_i$ for $|j-i| \geq 2$
\item $e_i e_j = e_j e_i$ for $|j-i| \geq 2$
\item $s_i e_i = e_i s_i = \la e_i$
\item $s_i s_j e_i = e_j e_i = e_j s_i s_j$ for $|j-i| = 1$
\item  $e_i^2 = x e_i$
\item  $e_i e_{i\pm 1} e_i = e_i$.
\end{enumerate}
It is a free $\kk$-module of dimension $(2n-1).(2n-3).\dots.3.1$, isomorphic to Kaufmann's tangle algebra (see \cite{MW}).
In case $\delta$ is invertible, the $e_i$ can be expressed in terms of the $s_i$.
This algebra can then be described as the quotient of the group algebra $\kk B_n$ with relations
(3), (7) ,(8), (9), (10) where $e_i$ is defined as $1 - \delta^{-1}(s_i - s_i^{-1})$.
Relation (7) is then equivalent to (7') :  $(s_i - \la)(s_i + q^{-1})(s_i - q) = 0$,
and a straightforward calculation shows that it implies (9). Now notice that the pair $(s_i,s_{i+1})$
is conjugated in $B_n$ to the pair $(s_{i+1}, s_i)$, hence (3) can be rewritten as
$e_i s_j^{\pm 1} e_i  = \la^{\mp 1} e_i$ whenever $|j-i| = 1$. Then (10) is easily seen to
be a consequence of (3) and (9), hence of (3) and (7). The relation (8) can be shown to be
implied by (3) and (7') (see \cite{WENZL} \S 3).
Finally, note that conjugation in the braid groups shows that (3) is equivalent to
(3') : $e_1 s_2^{\pm 1} e_1 = \la^{\mp 1} e_1$.

A natural quotient of $BW_n$ is obtained by adding the relation $e_i = 0$, or
equivalently $s_i - s_i^{-1} = \delta$. This quotient is naturally
isomorphic to the Hecke algebra $\kk B_n / (s_i -q)(s_i+q^{-1})$.

We now specialize to the specific instance we are interested in, by taking
$\kk = \F_4$, $q = j \in \F_4 \setminus \F_2$ hence $\delta = 1$, and
$x = 1$. Then relation (7') is $s_i^3 = 1$, which means that
$BW_n$ is the quotient of $\kk \Gamma_n$ by the relations $(3')$, which we
split as the two relations $(3'_{\pm}) : e_1 s_2^{\pm} e_1 = e_1$. 
It can be checked (e.g. by computer) that the ideals generated by $(3'_+)$ and
$(3'_-)$ have dimension 8 in $\kk \Gamma_3$, while their sum
has dimension 9, as is known by $\dim BW_3 = 15$. Note that
the relations $(3'_{\pm})$ can also be rewritten in our
case $e_1(s_2^{\pm} + 1) e_1 = 0$.

\begin{defi} Let $r_w^{\pm} = e_1(s_2^{\pm} + 1) e_1 \in \kk \Gamma_3$,
that is, writing $\Gamma_3 = Q_8 \rtimes C_3$,
$$
\left\lbrace \begin{array}{lcl}
r_w^+ &=& (1+\ii z + \jj z + \kkb z)(1+s+s^2) \\
r_w^- &=& (1+\ii+\jj+\kkb)(1+s+s^2)
\end{array} \right.
$$
\end{defi}

We notice for future use that the three 1-dimensional
representations of $\kk \Gamma_n$ factor through $BW_n$. Indeed,
the two non-trivial ones factor through the Hecke algebra
$H_n(j,j^2)$, which is a quotient of $BW_n$,
while $BW_n$ admits the representation $s_i \mapsto 1$,
$e_i \mapsto 1$, which induces the trivial representation of $\Gamma_n$.

\subsection{Another quotient of $\AAA_n$ in characteristic 2}

We use the representation of $BW_n$ in terms of tangles, taking
for convention that the product $xy$ of the tangles $x$ and $y$
is obtained by putting $y$ \emph{below} $x$. Following \cite{MW}, a basis
for $BW_n$ is given by a basis of the algebra of Brauer diagrams and an arbitrary
choice of over and under crossings. The basis chosen for $BW_3$ is pictured
in figure \ref{pictsBW3}, with $s_1 = s = b_2$ ; 
the morphism $\kk Q_8 \to BW_3$
and the multiplication on the right by $s = b_2$ are
tabulated in tables \ref{tablemapQ8BW3} and \ref{tableMultsBW3},
respectively.

Let $\varphi \in \Aut(\kk B_n)$ be defined by $\varphi(s_i) = j s_i$.
It induces an automorphism of $\kk \Gamma_n$ of order 3, and
$\varphi^3 = \Id$. Let $\mathcal{B}_1^n$ be the kernel
of $\kk \Gamma_n \onto BW_n$, namely the ideal
$(3'_+) + (3'_-)$. We have $\varphi(\qq) = \qq$, and we let
$\mathcal{B}_j^n = \varphi(\mathcal{B}_1^n)$,
$\mathcal{B}_{j^2}^n = \varphi^2(\mathcal{B}_1^n)$,

\begin{prop} \label{propBBBsomme}
\begin{enumerate}
\item The natural morphism $\kk \Gamma_n \onto BW_n$ factors through $\AAA_n$
\item When $n = 3$, its kernel is contained in $J(\kk \Gamma_3)$ 
\item We have $\qq \in \mathcal{B}^n = \mathcal{B}_1^n \cap \mathcal{B}_j^n\cap \mathcal{B}_{j^2}^n$
\item $(\qq) = \mathcal{B}^3$
\item $1+z_3 \in \mathcal{B}_+ =   \mathcal{B}_1 + \mathcal{B}_j + \mathcal{B}_{j^2}$
\item $\kk \Gamma_n/\mathcal{B}_+ = \kk C_3$ for $n \geq 5$.
\item For $n = 3$, $\mathcal{B}_+ = J(\kk Q_8)^2 C_3 = J(\kk \Gamma_3)^2$.
\end{enumerate}
\end{prop}
\begin{proof}
Part (i) means that $q$ is mapped to $0$, which can be checked on
table \ref{tablemapQ8BW3}.
Part (ii) is because $\kk \Gamma_3$ has only 3 simple representations,
all coming from $BW_n$, hence all annihilated by the kernel.
Part (iii) follows from $\varphi(\qq) = \qq$. For part (iv),
notice that $\mathcal{B}^3$ is stabilized by $\varphi$,
hence $\mathcal{B}^3 = I_1 \oplus I_j \oplus I_j^2$
with $I_{\alpha} = \Ker(\varphi- \alpha)$, as $char. \kk \neq 3$.
On the other hand, $\kk \Gamma_3 = \bigoplus_{i=0}^2 (\kk Q_8) s^i$
and $\varphi(s^i) = j^is^i$. Since $\mathcal{B}^3$ is an ideal, $I_j = I_1 s,
I_{j^2} = I_1 s^2$, hence $\mathcal{B}_3 = I_1 \rtimes C_3$ with
$I_1$ an ideal of $\kk Q_8$. This ideal is the kernel of
the natural map $\kk Q_8 \to BW_3$, and it is easy to determine
from table \ref{tablemapQ8BW3}. We find $I_1 = \kk \qq$, hence
(iv).
In $Q_8$, we have $1+z = (1+z\ii+z\jj+z\kkb)  + z(1+\ii+\jj+\kkb)$, hence
$(1+z)S \in \mathcal{B}_0$ with $S = 1+ s + s^2$. Since
$S,\varphi(S),\varphi^2(S)$ span $\kk C_3$, we have $1+z \in \mathcal{B}_+$,
which proves (v). We have $\Gamma_5 = \Gamma_5^0 \rtimes C_3$,
with $\Gamma_5^0$ the normal subgroup of $\Gamma_5$ generated by $z_3$.
Since $\mathcal{B}_+$ is invariant under $\varphi$, we have
$\mathcal{B}_+ = I \rtimes C_3$ with $I \subset \kk \Gamma_5^0$.
Now $1 + z_3 \in \mathcal{B}_+$ hence $1+z_3 \in I$,
and $\kk \Gamma_5^0 /I = \kk (\Gamma_5^0/\ll z_3 \gg) = \kk$.
The ideal $I$ is then the augmentation ideal of $\kk \Gamma_5^0$,
$\kk \Gamma_5 / \mathcal{B}_+ = \kk C_3$. This implies
$\kk \Gamma_n / \mathcal{B}_+ = \kk C_3$ for all $n \geq 5$, hence (vi).
For $n=3$, we have similarly $\mathcal{B}_+ = I C_3$ for
some ideal $I$ of $\kk Q_8$. We have $(1+\ii)(1+\jj) = 1+\ii+\jj+\kkb$,
and $1+z\ii+z\jj+z\kkb = z(1+\ii)(1+\jj) + (1+\ii)^2$ hence $I \subset J(\kk Q_8)^2$.
We know that $\dim J(\kk Q_8)^2 = 5$ and we compute that $\dim \mathcal{B}_+
= 15$, which proves (vii).
\end{proof}

\begin{remark} For $n=4$, $\mathcal{B}_+$ has dimension 639.
\end{remark}

The proposition above enables us to define the following quotient of
$\AAA_n$.

\begin{defi} \label{defbmw} We define the algebra $\mathcal{BMW}_n$ as
$\kk \Gamma_n / \mathcal{B} = \kk \Gamma_n / (\mathcal{B}_1 \cap \mathcal{B}_j \cap \mathcal{B}_{j^2})$.
It is a quotient of $\AAA_n$.
\end{defi}

\subsection{A natural embedding}

Let $(T_w, w \in \mathfrak{S}_n)$ denote the standard basis of the Hecke algebra under consideration (see \cite{HUMPHREYS}) 
and 
$\ell : \mathfrak{S}_n \to \Z_{\geq 0}$
the Coxeter length.
For $\alpha \in \mu_3(\kk)$, we let
$E_n(\alpha) = \sum_{w \in \mathfrak{S}_n} \alpha^{\ell(w)} T_w$.
In particular
$E_3(\alpha) = \alpha^3 s_1 s_2 s_1 + \alpha^2 s_1 s_2 + \alpha^2 s_2 s_1 + \alpha s_1 + \alpha s_2 + 1 = s_1 s_2 s_1 + \alpha^2 s_1 s_2 + \alpha^2 s_2 s_1 + \alpha s_1 + \alpha s_2 + 1$.
We recall from (\cite{GrahamLehrer}, \S 4.3) that the Temperley-Lieb
algebra $TL_n(1,j)$ is $H_n(1,j) / E_3(j^2) = H_n(1,j) / E_3(j^{-1})$ (notice that
a slight renormalization of the Artin generators is
needed from the original formulations there).
It has dimension the $n$-th Catalan number $C_n = \frac{1}{n+1} \left( \begin{array}{c} 2n \\ n \end{array} \right)$.

Let $\{ \alpha,\beta,\gamma \} = \mu_3$. We introduce
the involutive automorphism $\tau_{\gamma}$ of $\kk \Gamma_n$ defined by $s_i \mapsto
\gamma^2 s_i^{-1}$. It maps $(s_i + \alpha)(s_i + \beta)$
to $s \alpha \beta(s+ \beta)(s+ \alpha)$, hence induces
an involutive automorphism $\tau_{\alpha,\beta}$ of $H_n(\alpha,\beta)$.
The automorphism $\varphi$ induces isomorphisms $\tilde{\varphi} : H_n(\alpha,
\beta) \to H_n(j^2 \alpha,j^2 \beta)$ making the natural diagram commute.
$$
\xymatrix{ \kk \Gamma_n \ar[d] \ar[r]^{\varphi} & \kk \Gamma_n  \ar[d] \\
H_n(\alpha,\beta) \ar[r]_{\tilde{\varphi}} & H_n(j^2 \alpha,j^2 \beta)
}
$$
{}
$$
\xymatrix{
& \mathcal{H}_n \ar[r] & H_n(j,j^2) \ar@/^1pc/[rr]^{\tilde{\varphi}}  &\times & H_n(1,j) \ar@/^1pc/[rr]^{\tilde{\varphi}} &\times & H_n(1,j^2) \ar@/_3pc/[llll]_{\tilde{\varphi}} \\  
\AAA_n \ar[ur] \ar[rr] & & \kk \Gamma_n/\mathcal{B}_1  \ar[u]  \ar@/_1pc/[rr]_{\tilde{\varphi}} & \times & \kk \Gamma_n/\mathcal{B}_j  \ar[u]  \ar@/_1pc/[rr]_{\tilde{\varphi}}  & \times & \kk \Gamma_n/\mathcal{B}_{j^2} \ar[u] \ar@/^3pc/[llll]^{\tilde{\varphi}}\\
}
$$
Let $ITL_n^j(1,j)$ denote the ideal of $H_n(1,j)$ generated
by $E_3(j^2) = E_3(j^{-1})$. Then $ITL_n^1(1,j) = \tau_{1,j} ITL_n^+(1,j)$
is the ideal generated by $\tau_{1,j}E_3(j^{-1}) = E_3(1)$. A straightforward computation
shows more generally that $\tau_{\gamma}E_3(\alpha^{-1}) = E_3(\beta^{-1})$.
We define more generally

\begin{defi} For $\{ \alpha, \beta , \gamma \} = \mu_3(\kk)$,
we define $ITL_n^{\alpha}(\alpha,\beta) = ITL_n^{\alpha}(\beta,\alpha)$
as the 2-sided ideal of $H_n(\alpha,\beta)$ generated by $E_3(\alpha^{-1})$.
\end{defi}

With this definition, we have
$\tau_{\alpha,\beta}(ITL_n^{\alpha}(\alpha,\beta)) = 
ITL_n^{\beta}(\beta,\alpha)$. Moreover,
we have $\varphi(E_3(x)) = E_3(jx)$,
hence $\tilde{\varphi}$ maps
$ITL_n^{\alpha}(\alpha,\beta) \subset H_n(\alpha,\beta)$
to the ideal of $H_n(j^2 \alpha,j^2 \beta)$ generated
by $E_3(j \alpha^{-1}) = E_3((j^2 \alpha)^{-1})$, that is
$ITL_n^{j^2 \alpha}(j^2\alpha,j^2 \beta)$.

\begin{lemma} \label{lemideauxTL} In $H_n(\alpha,\beta)$, let $M_n(\alpha)$ and $M_n(\beta)$
denote the kernels of the natural morphisms $H_n(\alpha,\beta) \to \kk$
defined by $s_i \mapsto \alpha$ and $s_i \mapsto \beta$, respectively. We have
\begin{enumerate}
\item $ITL_n^{\alpha}(\alpha,\beta ) \subset M_n(\alpha) \cap M_n(\beta)$
\item $ITL_n^{\alpha}(\alpha,\beta ) + ITL_n^{\beta}(\alpha,\beta )= M_n(\alpha) \cap M_n(\beta)$ for $n \geq 5$.
\item For all $n$, $M_n(\alpha) \cap M_n(\beta)$ is generated by $s_1 s_2 + s_2 s_1$.
\end{enumerate}
\end{lemma}
\begin{proof}
Part (i) comes from the fact that $E_3(\alpha^{-1})$ is mapped to $0$
under both $s_i \mapsto \alpha$ and $s_i \mapsto \beta$, as is easily checked.
We now deal with part (ii).
For $n = 5$ 
We check by computer that $s_1 s_2 s_4 + 1 \in I = ITL_n^{\alpha}(\alpha,\beta ) + ITL_n^{\beta}(\alpha,\beta )$
when $n = 5$, hence for $n \geq 5$. It follows that $H_n(\alpha,\beta)/I$
is a quotient of $\kk \Gamma_n/N$, where $N$ is the normal subgroup
of $\Gamma_n$ generated by $w = s_1 s_2 s_4$. Note that $w \in \Gamma_n^0
= \Ker(\Gamma_n \onto C_3)$. In particular, for $n = 5$, $w$
belongs to $\Sp_4(\F_3)$, and one easily check that $N = \Sp_4(\F_3) = \Gamma_n^0$
in this case, by quasi-simplicity of $\Sp_4(\F_3)$. By Theorem \ref{theoprelimgamma} (vi) it follows
that $N = \Gamma_n^0$ for all $n \geq 5$. Thus $H_n(\alpha,\beta)/I$
is a quotient of $\kk C_3 = \kk [s]/(s^3 -1)$ of dimension at least 2,
and even of $\kk [s]/(s+ \alpha)(s+ \beta)$, which has dimension 2. It follows
that $H_n(\alpha,\beta)/I$ has dimension 2 hence $I = M_n(\alpha) \cap M_n(\beta)$
hence (ii). In order to prove (iii), we first note that $x = s_1 s_2 + s_2 s_1$
is mapped to 0 under the maps $s_i \mapsto \alpha$ and $s_i \mapsto \beta$,
hence $(x) \subset K = M_n(\alpha) \cap M_n(\beta)$. It is
then sufficient to show that $H_n(\alpha,\beta)/(x)$ has dimension 2. From the
presentation of $H_n(\alpha,\beta)$ one gets that adding $s_1 s_2 = s_2 s_1$
implies $s_i = s_j$ for all $i,j$ hence $H_n(\alpha,\beta)/(x) = \kk[s]/(s+ \alpha)
(s+ \beta)$ has dimension 2. This proves (iii).

\end{proof}

\begin{remark}
In the characteristic 0 (semisimple) case with generic parameters, the sum of the
two copies of the Temperley-Lieb ideals is the whole Hecke algebra for $n \geq 5$,
as the corresponding quotient has irreducible representations labelled
by the Young diagrams with at most 2 rows and 2 columns, and there are
clearly no such diagram of size more than 4.
\end{remark}

\begin{lemma} In $H_n(1,j)$, we have
$ 
r_w^+ \equiv j^2 E_3(j^2) ,
r_w^- \equiv j  E_3(j^2) ,
\varphi(r_w^+) \equiv 0 ,
\varphi(r_w^-) \equiv 0 ,
\varphi^2(r_w^+) \equiv j E_3(1)$ and
$\varphi^2(r_w^-) \equiv j^2 E_3(1)$.

\end{lemma}
\begin{proof} Straightforward computation from the equations
$s_i^2 + j^2 s_i + j = 0$ and $s_i^{-1} = j^2 s_i + j$.
\end{proof}

\begin{lemma} \label{lemimagesB} Let $n \geq 3$ and $\pi : \kk \Gamma_n \onto \mathcal{H}_n \into H_n(j,j^2)
\times H_n(1,j) \times H_n(1,j^2)$. Then
$\pi(\mathcal{B}_1) \subset 0 \times ITL_n^j \times ITL_n^{j^2}$,
$\pi(\mathcal{B}_j) \subset ITL_n^j \times 0 \times ITL_n^1$,
$\pi(\mathcal{B}_{j^2}) \subset ITL_n^{j^2} \times ITL_n^1 \times 0 $
\end{lemma}
\begin{proof}
Let $p_{\gamma} : \mathcal{H}_n \to H_n(\alpha,\beta)$, for
$\{ \alpha,\beta,\gamma \} = \mu_3$. The induced map from $\AAA_n$,
also denoted by $p_{\gamma}$,
factors through $\kk \Gamma_n/ \mathcal{B}_{\gamma^{-1}}$.
We have $p_1(\mathcal{B}_1) = 0$,
$p_{j^2} (\mathcal{B}_1) = ITL_n^j(1,j)$ by the lemma.
We have $p_{\gamma} \circ \varphi = \tilde{\varphi} \circ p_{\gamma j}$
hence $p_{\gamma} \circ \varphi^2 = \tilde{\varphi}^2 \circ p_{\gamma j^{-2}}$
and, using the lemma and the commutative diagrams above,
 $p_j(\mathcal{B}_1) = \tilde{\varphi}(p_{j^2}(\varphi^{-1} (\mathcal{B}_1)))
=  \tilde{\varphi}(p_{j^2}(\mathcal{B}_{j^2})) = \tilde{\varphi}(ITL_n^1(1,j))
= ITL_n^{j^2}(1,j^2)$.

\end{proof}

\begin{prop} \label{propbBBB} Recall $\bb = s_1 s_2^{-1} + s_2 s_1^{-1} + s_1^{-1} s_2
+s_2^{-1} s_1$.
\begin{enumerate}
\item $\bb \in (\mathcal{B}_1 + \mathcal{B}_j) \cap (\mathcal{B}_1 + \mathcal{B}_{j^2})
\cap (\mathcal{B}_j + \mathcal{B}_{j^2})$ for $n \geq 4$
\item In $H_4(1,j)$, one has $ITL_4^1 \cap ITL_4^j = \{ 0 \}$
\item For $n \geq 4$, the inclusions of Lemma \ref{lemimagesB} are equalities.
\item For $n \geq 5$, $\dim \kk \Gamma_n / (\mathcal{B}_1 + \mathcal{B}_j) = 2 \dim TL_n - 1$.
\item $\mathcal{B}_1 + \mathcal{B}_j \cap \mathcal{B}_{j^2}
= (\mathcal{B}_1 + \mathcal{B}_j) \cap (\mathcal{B}_1 + \mathcal{B}_{j^2})$ for $n = 4$.
\item For $n \geq 4$, $\bb \in \mathcal{B}_1 + \mathcal{B}_j \cap \mathcal{B}_{j^2}$.
\end{enumerate}
\end{prop}
\begin{proof}
For proving (i) one needs to check that $\bb \in \mathcal{B}_1 + \mathcal{B}_j$,
as $\varphi(\bb) = \bb$, and one needs to do it only for $n=4$, which
follows from a computer check.
(ii) also follows from a computer check.
As a consequence of $(\bb) \subset \mathcal{B}_1+\mathcal{B}_j$, 
$\kk \Gamma_n / \mathcal{B}_1+\mathcal{B}_j$ is a quotient
of $\kk \Gamma_n / (\bb)$, that is of the ternary Hecke
algebra.
Letting again $\pi : \kk \Gamma_n \onto \mathcal{H}_n 
\subset H_n(1,j)\times H_n(1,j^2) \times H_n(j,j^2)$ denote
the natural map, we have $\pi(\mathcal{B}_1 + \mathcal{B}_j)
= \pi(\mathcal{B}_1)  + \pi(\mathcal{B}_j) \subset
ITL_n^j \times ITL_n^j \times (ITL_n^1 + ITL_n^{j^2})$.
When $n = 4$ a computer check shows that the two sides of this
inclusion have the same dimension (which is 40). Since
$ITL_4^1 \cap ITL_4^{j^2} = \{ 0 \}$ by (ii), this implies
that the inclusions of Lemma \ref{lemimagesB} are equalities for $n = 4$,
say for $\mathcal{B}_1$. This means that $\pi(\mathcal{B}_1)$ contains,
for $n = 4$ hence for all $n \geq 4$, the elements $(0,E_3(j^{-1}),0)$
and $(0,0,E_3(j^{-2}))$ ; it follows that, for $n \geq 4$, $\pi(\mathcal{B}_1)$ contains
$ 0 \times ITL_n^j \times ITL_n^{j^2}$, hence is equal to it. Since $\pi$ commutes
with $\varphi$ this implies (iii) also for all the $\mathcal{B}_{\gamma}$.
For (iv), let $\pi : \kk \Gamma_n \onto \kk \Gamma_n/(\bb) = \mathcal{H}_n$ denote the
natural projection. Since $\bb \in \mathcal{B}_1 + \mathcal{B}_j$,
the dimension of  $\kk \Gamma_n / (\mathcal{B}_1 + \mathcal{B}_j)$ is
$$\dim \mathcal{H}_n /\pi (\mathcal{B}_1 + \mathcal{B}_j) =
-3 + \dim \frac{H_n(j,j^2) \times H_n(1,j) \times H_n(1,j^2)}{ITL_n^j \times ITL_n^j \times (ITL_n^1 + ITL_n^{j^2})}
 = -3 + 2 \dim TL_n +2$$
for $n \geq 5$ by Lemma \ref{lemideauxTL}, which proves (iv). (v) is proved by a direct computer
check, and (vi) is a trivial consequence of (v) and (i).
\end{proof}
\begin{remark}
 1) $\bb \not\in \mathcal{B}_1 + \mathcal{B}_j$ for $n = 3$.
2) Using a computer one can prove that the natural map $\AAA_3 \to \AAA_4$
is injective.
\end{remark}

It follows from the proposition that $\pi(\mathcal{B}_1+ \mathcal{B}_j) =
ITL_n^j \times ITL_n^j \times (ITL_n^1 + ITL_n^{j^2})$ for all $n \geq 4$.
By Lemma \ref{lemideauxTL}, for $n \geq 5$ this is
$ITL_n^j \times ITL_n^j \times M_n(1) \cap M_n(j^2)$,
and likewise $\pi(\mathcal{B}_1+ \mathcal{B}_{j^2}) = ITL_n^{j^2} \times (ITL_n^1+ITL_n^j)
\times ITL_n^{j^2}$. Letting $\cap ITL$ denote $ITL_n^{\alpha}(\alpha,\beta)
\cap ITL_n^{\beta}(\alpha,\beta)$, we have
$\pi \left( (\mathcal{B}_1+ \mathcal{B}_{j^2}) \cap (\mathcal{B}_1+ \mathcal{B}_{j}) \right) = \pi(\mathcal{B}_1+ \mathcal{B}_{j^2}) \cap \pi(\mathcal{B}_1+ \mathcal{B}_{j})
= \cap ITL_n \times ITL_n^j \times ITL_n^{j^2}$, because $(\mathcal{B}_1+ \mathcal{B}_{j^2}) \cap (\mathcal{B}_1+ \mathcal{B}_{j})$
contains $\Ker \pi$ for $n \geq 4$ by (1).
Also, $\pi(\mathcal{B}_j) \cap \pi(\mathcal{B}_{j^2}) = \cap ITL_n \times 0 \times 0$,
hence $\pi(\mathcal{B}_1) + \pi(\mathcal{B}_j) \cap \pi(\mathcal{B}_{j^2}) = 
\pi \left( (\mathcal{B}_1+ \mathcal{B}_{j^2}) \cap (\mathcal{B}_1+ \mathcal{B}_{j}) \right)$.
This implies $\pi(\mathcal{B}_1 + \mathcal{B}_j \cap \mathcal{B}_{j^2})
\subset \cap ITL_n \times ITL_n^j \times ITL_n^{j^2}$. We check by computer
that the dimensions on both sides are equal for $n=4$. This proves
that $\pi(\mathcal{B}_1 + \mathcal{B}_j \cap \mathcal{B}_{j^2})$
contains $(0,E_3(j^{-1}),0)$,$(0,0,E_3(j^{-2})$.

In order to have the property
that $\mathcal{B}_1 + \mathcal{B}_j \cap \mathcal{B}_{j^2} 
= (\mathcal{B}_1+ \mathcal{B}_{j^2}) \cap (\mathcal{B}_1+ \mathcal{B}_{j})$
it would be sufficient to control $\cap ITL_n$ in the sense that,
if $\pi(\mathcal{B}_j \cap \mathcal{B}_{j^2}) = \cap ITL_n$ for
some $n$, and $ITL_m$ for $m \geq n$ is generated by elements in
$\cap ITL_n$, this would prove $\pi(\mathcal{B}_j \cap \mathcal{B}_{j^2}) = \cap ITL_m$ 
for all $m \geq n$. This at first seems not be such an obstacle as,
in the semisimple case, $\cap ITL$ is generated by $ab \in \cap ITL_5$
(or $ba$)
with $a = E_3(1)$ and $b = 1 + j^2 s_3 + j^2 s_4 + j s_3 s_4 + j s_4 s_3
+ s_3 s_4 s_3$ a conjugate of $E_3(j^2)$, and $ab$ is clearly
in the image of $\mathcal{B}_j
\cap \mathcal{B}_{j^2}$. However, by computer calculation, we get that
the situation is much more complicated in our case, as 
shown by the next lemma, which gathers the result of computer calculations.

\begin{lemma} Inside $H_n(1,j)$, we have
\begin{enumerate}
\item $\dim \cap ITL_5 = 38$.
\item For $n = 5$, $(ab) = (ba)$ and $\dim (ab) = 36$.
\item $\cap ITL_5 = (ab) \oplus \kk E_5(1) \oplus \kk E_5(j^2)$
\item $\dim \cap ITL_6 = 458$
\item For $n=6$, $\dim (ab) = 454$, the ideal
generated by $\cap ITL_5 \subset \cap ITL_6$ has dimension 456
and contains $E_6(1), E_6(j^2)$.
\item $\dim \cap ITL_7 = 4184$
\item For $n=7$, $\dim (ab) = 4180$, and $\cap ITL_7$ is generated
by $\cap ITL_6 \subset \cap ITL_7$.
\end{enumerate}
\end{lemma}

The fact that $\cap ITL_7$ is generated by $\cap ITL_6$ is checked as
follows : we find randomly a (complicated) element in $\cap ITL_6$ which it generates as an ideal,
and check that this element also generates $\cap ITL_7$.
Note that the following always holds true.
\begin{lemma} For all $n \geq 3$ and $\alpha,\beta \in \mu_3(\kk)$,
we have $E_n(\alpha^{-1}) \in H_n(\alpha,\beta) E_3(\alpha^{-1})$. 
\end{lemma}
\begin{proof}
Let $h = \sum_{w \in D} \alpha^{-\ell(w)} T_w$ with $\ell : \mathfrak{S}_n \to \Z$
the Coxeter length and $D$ the representative system of $\mathfrak{S}_n/
\mathfrak{S}_3$ consisting of $\mathfrak{S}_3$-reduced elements on the
right so that any element $\sigma \in \mathfrak{S}_n$ writes uniquely
$\sigma = w \sigma'$ with $w \in D, \sigma' \in \mathfrak{S}_3$ and $\ell(\sigma)
= \ell(w) + \ell(\sigma')$ (see \cite{HUMPHREYS} \S 1.10). Then clearly
$E_n(\alpha^{-1}) = h E_3(\alpha^{-1})$.
\end{proof}

\begin{lemma} Let $n \geq 5$. Then
$$
\dim \mathcal{BMW}_n = 3(   \dim BW_n -  2 \dim TL_n + 2) - \dim 
\frac{(\mathcal{B}_1+\mathcal{B}_j)\cap(\mathcal{B}_1+\mathcal{B}_{j^2})}{
\mathcal{B}_1+\mathcal{B}_j\cap\mathcal{B}_{j^2}}$$

\end{lemma}
\begin{proof}

Recall from Proposition \ref{propbBBB} that
$\dim \kk \Gamma_n /(\mathcal{B}_1 + \mathcal{B}_j) = 2 \dim TL_n - 1$.
We apply Lemma \ref{lemchinoistop} with $A = \mathcal{B}_+$,
$I = \mathcal{B}_1$, $J = \mathcal{B}_j$, $K = \mathcal{B}_{j^2}$.
We get $\dim A / (I \cap J \cap K) = \dim \Image d_1
= \dim \Ker d_2 - \dim (K+I) \cap (K+J) /(K + I \cap J)$, and, since $d_2$
is onto, $\dim \Ker d_2 = 3 \dim \mathcal{B}_+/\mathcal{B}_1 - \dim \mathcal{B}_+/(\mathcal{B}_1 + \mathcal{B}_j)
= 
3 \dim \kk \Gamma_n/\mathcal{B}_1 - 3 \dim \kk \Gamma_n/(\mathcal{B}_1 + \mathcal{B}_j)
$. 
 
Since $\dim \kk \Gamma_n / \mathcal{B}_+ = 3$,
we get $\dim \mathcal{BMW}_n = \dim \kk \Gamma_n / I \cap J \cap K
= 3 + \dim \mathcal{B}_+/(I\cap J\cap K)
= 3 + 3 \dim BMW_n - 3 (2 \dim TL_n - 1) - \dim 
\frac{(\mathcal{B}_1+\mathcal{B}_j)\cap(\mathcal{B}_1+\mathcal{B}_{j^2})}{
\mathcal{B}_1+\mathcal{B}_j\cap\mathcal{B}_{j^2}}$
whence the conclusion.

\end{proof}

\begin{remark} In particular, since $\dim TL_5 = 42$ and $\dim BMW_5 = 945$ one gets
$\dim \mathcal{BMW}_5 = 3 \times 863 - \dim 
\frac{(\mathcal{B}_1+\mathcal{B}_j)\cap(\mathcal{B}_1+\mathcal{B}_{j^2})}{
\mathcal{B}_1+\mathcal{B}_j\cap\mathcal{B}_{j^2}}$,
to be compared with $\dim \F_2 K_5 = 3 \times 863$ (see proposition \ref{dimA5p2}).
\end{remark}

\section{Markov traces}

\subsection{Definitions and conditions for $n = 3$}

In this section we deal with Markov traces. We let $K_n = K_n(1)$, and denote
$K_{\infty}$ the direct limit of the $K_n$ under the natural morphisms
$K_n \to K_{n+1}$. Letting $A = \Z[u,v]$, we denote $AB_{\infty}$, $A \Gamma_{\infty}$
the direct limits of the group algebras $A B_n$, $A \Gamma_n$, respectively.

\begin{defi} A Markov trace is a pair $(t,R)$, where $R$ is a $\Z[u,v]$-module
and $t \in \Hom_{A}(A B_{\infty}, R)$ satisfying
\begin{itemize}
\item $t(xy) = t(yx)$ for all $x,y \in AB_{\infty}$
\item $t(xs_n) = u t(x)$ for all $x,y \in AB_{n-1}$
\item $t(xs_n^{-1}) = v t(x)$  for all $x,y \in AB_{n-1}$
\end{itemize}
A Markov trace is said to factorize through a quotient $H$ of the $A$-algebra $A B_{\infty}$ if
it lies in the image of $\Hom_A(H,R) \to \Hom_A(A B_{\infty},R)$.
\end{defi}

We now assume that $t$ is a Markov trace that factors through
$K_{\infty}$. This means that it factors through $A \Gamma_{\infty}$,
and that $t(g_1 \qq g_2) = 0$ for
all $g_1,g_2 \in \Gamma_{\infty}$, or equivalently that $t(\qq g) = 0$
for all $g \in \Gamma_{\infty}$, and finally these conditions for $g \in
\Gamma_3$ reduce to $t(\qq) = t(\qq s_1) = t(\qq s_1^2) = 0$. A direct
computation shows that these equations imply the following.

\begin{lemma} If $t$ is a Markov trace that factors through $K_{\infty}$,
then $4(u^2+v) t(1) = 4(v^2+u) t(1) = 0$ and $t(z_3)  = -(1+ 6uv) t(1)$
\end{lemma}

Notice that a Markov trace factorizing through $K_{\infty}$
takes values in $A t(1) \subset R$, and that, as a consequence of
Proposition \ref{propfunar}, it is uniquely determined by the value of
$t(1) \in R$.

It should be noted that $\{ z_3 \}$ is the only conjugacy class in $\Gamma_3$ that does not meet any
$g s_2^{\eps}$ for $g \in \Gamma_2$ and $\eps \in \{0,1,2 \}$. Let $A \Gamma_{\infty}$ denote
the direct limit of the $A \Gamma_n$. Of course a Markov trace on $K_{\infty}$ induces
a Markov trace on $A \Gamma_{\infty}$. A Markov trace on $A \Gamma_{\infty}$ then induces elements
$\tau_n \in \Hom_A(A \Gamma_n,R)$ for all $n$ (recall from Theorem \ref{theoprelimgamma} that $A \Gamma_{\infty}$ contains
the $A \Gamma_n$ for $n \leq 5$). The condition $\tau_n(xy) = \tau_n(yx)$ means that
$\tau_n$ is actually a function on the conjugacy classes of $\Gamma_n$. For instance, a consequence
of the special property of $\{ z_3 \}$ mentionned above is that any such $\tau_3$ is defined uniquely
by the values $\tau_3(1)$ and $\tau_3(z_3)$. In the following section we looked at the conjugacy classes
of $\Gamma_4$ and $\Gamma_5$, and checked whether one could define functions $\tau_4, \tau_5$
such that $\tau_4, \tau_5$ vanish on the ideal generated by $\qq$.

\subsection{Conditions for $n = 4$}

In order to shorten computations with words in the $s_i$'s, we will
use when convenient the notation $ijk...$ for $s_i s_j s_k...$, with
$\mmm i$ meaning $s_i^{-1}$ (for instance $s_1 s_2^{-1} s_3 =  1 \mmm 2 3$).

\begin{lemma} \label{relu3v3} 
If $t$ is a Markov trace that factors through $K_{\infty}$,
then $(3u^3+3v^3-5uv-1)t(1) = 0.$
\end{lemma}
\begin{proof}
We consider $x = s_2s_1^{-1} s_3s_2^{-1}$ and $y = s_2^{-1} s_1 s_3 s_2^2 s_3 s_2 s_1$
in $\Gamma_4$. In $K_{\infty}$ we have $ - s_3 s_2^2 s_3 \equiv
s_2 s_3^{-1} s_2 + s_2^{-1} s_3 s_2 + s_2 s_3 s_2^{-1} + s_2^{-1} s_3^{-1} +
s_3^{-1} s_2^{-1} + s_2 + s_3$. Then
$t(y) = 
-t(s_2^{-1} s_1 s_2 s_3^{-1} s_2s_2 s_1) -t(s_2^{-1} s_1 s_2^{-1} s_3 s_2s_2 s_1) - t(s_2^{-1} s_1 s_2 s_3 s_2^{-1}s_2 s_1) - t(s_2^{-1} s_1 s_2^{-1} s_3^{-1}s_2 s_1) 
-t(s_2^{-1} s_1 s_3^{-1} s_2^{-1}s_2 s_1) -t( s_2^{-1} s_1 s_2s_2 s_1) -t(s_2^{-1} s_1 s_3s_2 s_1)$.
We have 
$$
\begin{array}{lclclcl}
t(s_2^{-1} s_1 s_2 s_3^{-1} s_2s_2 s_1) &=& 
t(s_2^{-1} s_1s_2^{-1} s_1 s_2 s_3^{-1} )&=&
v\,t(s_2^{-1} s_1s_2^{-1} s_1 s_2 )&=&
v\,t(s_1  s_2 s_2^{-1} s_1s_2^{-1} ) \\
 &=&
v\,t(s_1 ^{-1}s_2^{-1} ) &=& v^3 t(1)\\
t(s_2^{-1} s_1 s_2^{-1} s_3 s_2s_2 s_1) &=&
t(s_2^{-1} s_1 s_2^{-1} s_1 s_2^{-1} s_3 ) &=&
u\,t(s_2^{-1} s_1 s_2^{-1} s_1 s_2^{-1}  ) &=&
u\,t(s_1 s_2^{-1} s_2^{-1} s_1 s_2^{-1}   ) \\ &=&
u\,t(s_1 s_2 s_1 s_2^{-1}   ) &=&
u\,t(s_2 s_1 s_2 s_2^{-1}   ) &=&
u\,t(s_1s_2   ) \\ &=& u^3 t(1)\\
t(s_2^{-1} s_1 s_2 s_3 s_2^{-1}s_2 s_1) &=&
t( s_1 s_2^{-1} s_1 s_2 s_3 ) &=&
u\,t( s_1 s_2^{-1} s_1 s_2  ) &=&
u\,t(  s_1 s_2  s_1 s_2^{-1} ) \\ &=&
u\,t(  s_2 s_1  s_2 s_2^{-1} ) &=&
u\,t(  s_1 s_2 ) &=& u^3t(1) \\
t(s_2^{-1} s_1 s_2^{-1} s_3^{-1}s_2 s_1) &=& 
t(s_2 s_1 s_2^{-1} s_1 s_2^{-1} s_3^{-1}) &=& 
v\,t(s_2 s_1 s_2^{-1} s_1 s_2^{-1} ) &=& 
v\,t(s_1 s_2^{-1} s_2 s_1 s_2^{-1}  )\\ &=& 
v\,t(s_1^{-1} s_2^{-1}  ) &=& v^3t(1) \\
t(s_2^{-1} s_1 s_3^{-1} s_2^{-1}s_2 s_1) &=&
t( s_1 s_2^{-1} s_1 s_3^{-1} ) &=&
v\,t( s_1 s_2^{-1} s_1  ) &=&
v\,t( s_1^{-1} s_2^{-1}   ) \\ &=& v^3t(1) \\
t( s_2^{-1} s_1 s_2s_2 s_1) &=&
t( s_2^{2} s_1 s_2^2 s_1) &=& t( z_3) &=& (-1-6uv)t(1) \\
t(s_2^{-1} s_1 s_3s_2 s_1) &=&
t(s_2 s_1s_2^{-1} s_1 s_3) &=&
u\,t(s_2 s_1s_2^{-1} s_1 ) &=&
u\,t(s_1 s_2 s_1s_2^{-1}  ) \\ &=&
u\,t(s_2 s_1 s_2s_2^{-1}  ) &=&
u\,t(s_1 s_2) &=& u^3 t(1)\\
\end{array}
$$
hence $t(y) = (-3u^3-3v^3+1+6uv)t(1)$. One has $t(x) =t(s_1^{-1} s_3) = uv\,t(1)$.
It is easily checked that $x$ and $y$ belong to
$G=\Ker (\Gamma_4 \onto \Gamma_3)$, which is an extra-special group $3^{1+2}$
which contains $z_4 = (s_1s_2s_3)^4$, hence $(G,G) = Z(G) = Z(\Gamma_4) = <z_4 >$.
We prove that $y = x z_4$.
From the braid relations
we get $(s_1 s_2 s_3)^3= 123123123 = 121121321= s_1 s_2 s_1^2 s_2 s_1 s_3 s_2 s_1$, hence $y = xz_4$
means that $s_2^2 s_1 s_3 s_2^2 = s_2 s_1^2 s_3 s_2^2 s_1 s_2 s_3 
s_1 s_2 s_1^2 s_2 s_1$ ; this comes from
the equalities $211322123121121 = 211322121321121 = 211322212321121
= 21312321121 = 211132321121 = 223221121 = 223221212 = 223222122 =
223122 = 221322$. Clearly $x \not\in Z(\Gamma_4) = Z(G)$.
For an extra-special group, the conjugacy classes
not lying in $Z(G)$ are determined by their images in $G/(G,G) = G/Z(G)$,
hence $x,y$ are conjugated in $G$ hence in $\Gamma_4$. This
proves $t(x)=t(y)$ hence $(3u^3 + 3v^3 -5 uv -1)t(1) = 0$ in $R$.
\end{proof}

\begin{lemma} 
If $t$ is a Markov trace that factors through $K_{\infty}$,
then $16t(1)=0$, $4uv\,t(1)=4t(1)$, $4u^3t(1) = 4v^3t(1) = -4t(1)$.
\end{lemma}
\begin{proof}
We recall $(32\mmm3) = (\mmm 232)$ and $(3\mmm23) = -(2\mmm 32)-(\mmm232)-(23\mmm2)-(\mmm2\mmm3)-(\mmm3
\mmm2)-(2)-(3) $ and note that 
$t(z_3) = (-1 -6uv)t(1)$, $t(12121) = t(11211)= t(11112)= t(12) = u^2t(1)$.
We will compute $t(a)$ and $t(b)$ with $a = (2\mmm312\mmm3121)$
and $b = (\mmm3231\mmm2312)$. It can be checked by hand
that, in $\Gamma_4$, we have $ac = cb$ with $c = (2\mmm13\mmm2)$, hence $t(a) = t(b)$.
\def\tt#1{t(\mathrm{{#1}})}

We first compute $t(a) = t(2\mmm312\mmm3121)$. We have
$t(2\mmm312\mmm3121) = t(21\mmm32\mmm3121) = t(21332\mmm3121) = 
t(213\mmm232121) = 
 -t(212\mmm322121)-t(21\mmm2322121)-t(2123\mmm22121)-t(21\mmm2\mmm32121)
-t(21\mmm3
\mmm22121)-t(2122121)-t(2132121) $

\begin{itemize}
\item[$\bullet$] $t(212 \mmm 322121) = t(22121212\mmm 3) = v\,t(22121212) = 
v\,t(22212121) = v\,t(12121) = u^2v\,t(1) $ 
\item[$\bullet$] $ t(21\mmm2322121) = t(2212121\mmm23) = u\,t(2212121\mmm2) =
  u\,t(\mmm22212121) =  u\,t(212121) =  u\,t(121212) = u \,t(z_3)$
\item[$\bullet$] $ t(2123\mmm22121) = t(2123121) = t(1212123) = u\,t(121212) = u \, t(z_3)$
\item[$\bullet$] $ t(21\mmm2\mmm32121) = t(212121\mmm2\mmm3) =
 v\,t(212121\mmm2) =  v\,t(\mmm2212121) = v\,t(12121) = u^2v\,t(1)$
\item[$\bullet$] $ t(21\mmm3 \mmm22121) = t(21\mmm3 121) = 
 t(12121\mmm3 ) = v\,t(12121 ) = u^2v\,t(1)$ 
\item[$\bullet$] $ t(2122121) = t(2122212) = t(2112) = t(1122) = v^2t(1)$
\item[$\bullet$] $ t(2132121) = t(2121213) = u\,t(212121) = u\,t(121212) = u\,t(z_3)$
\end{itemize}
hence
$t(2\mmm312\mmm3121) = (-3u^2v -3 u\,t(z_3) - v^2)t(1)
= (-3 u^2v + 3u(1+6uv)-v^2)t(1) = (3u + 15u^2 v - v^2)t(1)$.

We now compute $t(b) = t(\mmm3231\mmm2312)$. We have
$t(\mmm3231\mmm2312)= t(\mmm3213\mmm2312) = -t(\mmm3212\mmm3212)-
t(\mmm321\mmm23212)-t(\mmm32123\mmm212)-t(\mmm321\mmm2\mmm312)-
t(\mmm321\mmm3
\mmm212)-t(\mmm321212)-t(\mmm321312)$ and
\begin{itemize}
\item[$\bullet$] $t(\mmm3212\mmm3212) = t(\mmm3121\mmm3212) = 
t(1\mmm32\mmm31212) = t(1332\mmm31212) = t(13\mmm2321212)$
\item[$\bullet$] $t(\mmm321\mmm23212) = t(\mmm3\mmm1213212) = 
t(\mmm1\mmm3231212) = t(\mmm133231212) = t(\mmm132321212) = 
t(\mmm123221212) = t(221212\mmm123) = u \,t(221212\mmm12) =   u \,t(\mmm12221212) =  
 u \, t(212) =   u \, t(122) = u^2v\,t(1)$  
\item[$\bullet$] $t(\mmm32123\mmm212) =t(\mmm31213\mmm212) = t(1\mmm3231\mmm212) 
=t(133231\mmm212) =
t(132321\mmm212) =t(123221\mmm212) = t(221\mmm212123) = u t(221\mmm21212) = 
u t(2221\mmm2121) = u t(1\mmm2121)= u t(1\mmm2212) = u t(112) = u^2v\,t(1)$
\item[$\bullet$] $t(\mmm321\mmm2\mmm312) = t(\mmm3\mmm121\mmm312) = t(\mmm1\mmm32\mmm3112) = t(\mmm1332\mmm3112) = 
t(\mmm13\mmm232112)$
\item[$\bullet$] $t(\mmm321\mmm3\mmm212) = t(\mmm32\mmm31\mmm212)
 = t(332\mmm31\mmm212) = t(3\mmm2321\mmm212)$
\item[$\bullet$] $t(\mmm321212) = t(21212\mmm3) = v\,t(21212) = v\,t(22122) = v\,t(22221) =  v\,t(21) = v\,t(12) = u^2v t(1)$
\item[$\bullet$] $t(\mmm321312) = t(\mmm323112) = t(3323112) = t(3232112) = t(2322112) = 
t(2211223) = u t(221122) =  u t(112222) = u t(112) = u^2v\,t(1)$
\end{itemize}
We have $t(21212) = t(22122)  = t(12222)  = t(12) = u^2t(1)$,
$t(221212\mmm12) = t(221212112) =  t(112221212) = 
t(111212) = t(212) =  t(122) = uv\,t(1)$, $t(13\mmm2321212) 
= t(213\mmm232121) = (3u+15u^2v-v^2)t(1)$ as we already computed,
hence $t(\mmm13\mmm232112) = t(113\mmm232112) = t(13\mmm2321121) = t(13\mmm2321212) = (3u+15u^2v-v^2)t(1)$,
$t(3\mmm2321\mmm212) = t(3\mmm23212212) =  t(3\mmm23212121) = 
 t(13\mmm2321212)  = (3u+15u^2v-v^2)t(1)$. We thus get
$t(\mmm3231\mmm2312) = (-3(3u+15u^2v-v^2) - 4 u^2v)t(1) =
(-9u + 3 v^2 -49u^2 v)t(1)$.
We have that $(\mmm3231\mmm2312)$
is conjugated to $(2\mmm312\mmm3121)$
hence $t(b) = (3u+15u^2v-v^2)t(1) = (-9u + 3 v^2 -49u^2 v)t(1)$.
Therefore $t(a)=t(b)$ means
$(64u^2v +12u-4v^2)t(1) = 0$.
Since $4v^2t(1) = -4u\,t(1)$ and $4u^2v\,t(1) = (4u^2)v\,t(1) = -4v^2t(1) = 4u\,t(1)$,
this means $(64u +12u+4u)t(1) = 0$, i.e. $80 u\,t(1) = 0$. Since $80 = 16\times5$
and we know $2^r t(1) = 0$ for some $r$, there exists $g,h \in \Z$
with $2^r g + 5 h = 1$ hence $80 h u t(1) = 16 u t(1) = 0$.
From $4v t(1)= -4u^2t(1)$
we then get $16 v\,t(1) = 0$. By Lemma \ref{relu3v3}
we have $(3u^3+3v^3-5uv-1)t(1) = 0$, whence $16u\,t(1)=16v t(1)= 0$
implies $16t(1)=0$. Moreover,
$0 = 4\times (3u^3+3v^3-5uv-1)t(1) = 
(12u^3+12v^3-20uv-4)t(1) 
= (-12uv-12uv-20uv-4)t(1) = (-44uv - 4)t(1)$ because $4u^3t(1) =
-4uv\,t(1) = 4v^3t(1)$. Since $-44t(1) = 4t(1)$. This proves $4uv\,t(1) = 4t(1)$,
and $4u^3t(1) = 4v^3t(1) = -4t(1)$.

\end{proof}

\begin{remark}
Over $A = \Z[u,v]/(16,4(u^2+v),4(v^2+u),3u^3 + 3v^3 -5 uv -1)$, one can define
a `Markov trace' for $n=4$ extending a given $\tau_3$ originating from $MT(K_{\infty},R)$,
namely a linear map $\tau_4 : A \Gamma_4 \to A$
with $\tau_4(xy) = \tau_4(yx)$ and, when $x \in A \Gamma_3$,
$\tau_4(xs_3) = u\tau_3(x)$, $\tau_4(xs_3^{-1}) = v \tau_3(x)$. This can be checked as follows : for each one
of the 24 conjugacy classes of $\Gamma_4$, one takes an element in it and
find a word in $s_1,s_2,s_3$ representing it ; we then get a value for the
Markov trace by the implicit algorithm used to prove Proposition \ref{propfunar}. This class function naturally
extends to a trace $\tau_4 : A \Gamma_4 \to A$, and we check that,
for each $g_0 \in \Gamma_3$, we have $\tau_4(g_0 s_3) = u\tau_3(g_0)$, $\tau_4(g_0 s_3^{-1}) = v\tau_3(g_0)$.
Finally, we check that this $\tau_4$ factorizes through $K_4$, that is that
$\tau_4(g_1\qq g_2) = 0$ for each $g_1,g_2 \in \Gamma_4$ and, as before,
$\qq$ is the sum of the elements of $Q_8 \subset \Gamma_3$. Since $g_1$ can be taken in
$\Gamma_4/N_{\Gamma_4}(Q_8)$ and $g_2$ can be taken in $Q_8 \backslash \Gamma_4$,
there is only 729 conditions $\tau_4(g_1 \qq g_2) = 0$ to check. Since $\tau_4$ is already a class
function this number of equations reduces drastically to 18, so we can check
that $\tau_4$ indeed factors through $K_4$.

When $n=5$, we check similarly that there is a linear map 
$\tau_5 : A \Gamma_5 \to A$ with $\tau_5(xy) = \tau_5(yx)$ and, when $x \in A \Gamma_4$,
$\tau_5(xs_4) = u\tau_4(x)$, $\tau_5(xs_4^{-1}) = v \tau_4(x)$ : the computations in \GAP\ 
take only a lot more time, and we use the software \MACAULAY\ in order to automatize
equality checking inside $A$. The conditions for $t$ to factorize through $K_5$
amount to 243 equalities in $A$, which we check to be true using \MACAULAY.
\end{remark}

The two lemmas above can be combined to show the following.

\begin{lemma}
If $t$ is a Markov trace that factors through $K_{\infty}$,
then $(u+v+1)(u+jv+j^2)(u+j^2 v + j)t(1) = (u^3+v^3 - 3uv + 1)t(1) = 0$.
\end{lemma}
\begin{proof}
$(u+v+1)(u+jv+j^2)(u+j^2 v + j) = u^3+v^3 - 3uv + 1$ holds true in $\Z[j]$,
and $(u^3+v^3 - 3uv + 1)t(1) = 0$
because $u^3+v^3 - 3uv + 1= (4u^3+4) + (4v^3 + 4)-2\times(4uv - 4)
- (3u^3 + 3v^3 - 5uv - 1) - 16$.
\end{proof}

\subsection{Markov traces modulo 4}

In this section we prove that Markov traces exist modulo 4. We let $R = (\Z/4\Z)[j]$,
that is $(\Z/4\Z)[x]/(x^2+x+1)$, and consider 
the reduction $\bar{t} : K_{\infty} \to R[u,v]\bar{t}(1)$, with values in
$R \otimes_{\Z/4\Z} (\Z[u,v] t(1) / 4t(1))$. Here we let $\mu_3  = \{ 1,j,j^2 \}$. Since $4\bar{t}(1) = 0$,
we have $0 = (3u^3 + 3 v^3 -5 uv - 1)\bar{t}(1) = -(u^3+uv+v^3+1)\bar{t}(1)
= -(u+v+1)(u+jv + j^2 )(u+j^2 v + j)\bar{t}(1)$ hence a natural map
$$ R[u,v]/(u^3 + v^3 + uv +1) \to \tilde{M} = \prod_{\gamma \in \mu_3} R[u,v] /(v + \gamma u + \gamma^2) \simeq R[u]^3.
$$
It can be checked (e.g. using \MACAULAY) that the intersection of the ideals
$(v + \gamma u + \gamma^2)$ in $R[u,v]$ is equal to their product $(u^3 + v^3 + uv +1)$,
so the above map is injective.
Now consider the Hecke algebras $H_n(\alpha,\beta)$ over $R = (\Z/4\Z)[j]$, their direct limit $H_{\infty}(\alpha,\beta)$, and
introduce their Markov trace $tr_{\gamma} : H_{\infty}(\alpha,\beta) \to R[u] \simeq R[u,v] /(v + \gamma u + \gamma^2)$
for $\{\alpha,\beta,\gamma \} = \mu_3$,
such that $tr_{\gamma}(gs_n) = utr_{\gamma}(g)$
and $tr_{\gamma}(gs_n^{-1}) = (\gamma u + \gamma^{-1})tr_{\gamma}(g) = v tr_{\gamma}(g)$
for $g \in H_n(\alpha,\beta)$.

They extend to Markov traces $K_{\infty} \to R[u,v] /(v + \gamma u + \gamma^2)$.
Then a convenient Markov trace $\bar{t} : K_{\infty} \to R[u,v]/(u^3+uv+v^3+1)$
can be defined by $\bar{t}(g) = (tr_{\gamma}(g))_{\gamma \in \mu_3}$ ; indeed,
this defines at first a map to the cyclic $R[u,v]/(u^3+uv+v^3+1)$-module generated
by $\bar{t}(1) \in \tilde{M}$, which is free of rank 1 as $R[u,v]/(u^3+uv+v^3+1) \into \tilde{M}$.
In particular, $\bar{t}(g) = 0$ for $g$ in the ideal $J_n(\alpha,\beta)$
of $R \Gamma_n$ defining $H_n(\alpha,\beta)$, for every $\alpha,\beta$.
It follows that $\bar{t}$ vanishes on $J$, hence factorizes through the direct
limit $\mathcal{H}_{\infty}$ of the $\mathcal{H}_n = R \Gamma_n / J$.
Finally the proof of Lemma \ref{lemAtoH} says that $\cc \in J$ not only modulo 2 but modulo 4,
hence $\bar{t}$ factorizes also through $K_n(1)$, so this $\bar{t}$ is indeed a
Markov trace on $K_{\infty}$.

\begin{prop} Any Markov trace $t$ on $K_{\infty}$ with $4t(1) = 0$
factorizes through $\mathcal{H}_{\infty}$, and is induced
by the Markov traces of the Hecke algebras $H_{\infty}(\alpha,\beta)$.
\end{prop}

\begin{remark}
\begin{enumerate}
\item Over $(\Z/4\Z)[j]$, and even over $\Z[j]$, denoting $\bb = s_1 s_2^{-1} - s_1^{-1} s_2 + s_2 s_1^{-1} - s_2^{-1} s_1$,
one still gets that $\bb$ belongs to the intersection of the ideals $J_n(\alpha,\beta)$. Do we still have $\mathcal{H}_n = R \Gamma_n/(\bb)$,
for $R = (\Z/4\Z)[j]$ or even $R = \Z[j]$ ?
\item A natural question is whether the Birman-Wenzl algebra is still a quotient of $R \Gamma_n/(\qq)$ when
$R = (\Z/4\Z)[j]$ ($\la = 1$, $\delta = j-j^2 = 1+2j$). The answer is no, as a straightforward though tedious calculation shows that,
over $\Z[j]$, $\qq$ is mapped inside $BW_3$ to $(1-\delta + \delta^2 - \delta^3) b_1 + (-2 \delta + \delta^2 - \delta^3) b_2
+ (\delta^2 - 3 \delta) b_3 + 2 b_4 + (2- \delta - \delta^2) b_5 + (3- \delta) b_6
+ (\delta^2 - \delta^3) b_7 + (\delta - 2 \delta^2 - \delta^3) b_8 + (2 \delta + \delta^2 + \delta^3) b_9 +
(\delta - \delta^3) b_{10} + 2 \delta^3 b_{11} + (\delta - \delta^2 - 2 \delta^3) b_{12}
+ (\delta^2 + \delta) b_{13} + (\delta- \delta^2) b_{14} + (\delta- \delta^2) b_{15}$,
which is nonzero modulo $4$.

\end{enumerate}
\end{remark}

\subsection{Comparison with the claims of \cite{FUNAR}}

In order to
make the comparison with \cite{FUNAR} easier, we switch our notations
to the ones there. We first briefly review the
setting used in \cite{FUNAR}. In \cite{FUNAR},
elements $z,z' \in \C^{\times}$ are chosen, $A = A(z,z')$ is defined to be the subring of $\C$ generated by
$z,z'$, the $K_n(\gamma)$ are defined over $A$ with $\gamma \in A$,
and the direct limit $K_{\infty} = K_{\infty}(\gamma)$
of the $K_n = K_n(\gamma)$ is introduced. Let $K_n^{\ab}$ be quotient of
the module $K_n$ by the submodule $[K_n,K_n]$ spanned by the $xy-yx$ for $x,y \in K_n$,
and $K_{\infty}^{\ab}$ be the direct limit of the $K_n^{\ab}$.

For $R$ some fixed
$A$-module, the following $A$-modules
are defined :
$$
\begin{array}{lcl}
AF(K_{\infty},R) &=& \{ t \in \Hom_A(K_{\infty},R) \ | \ t(x s_n y) = zt(xy),
t(xs_n^{-1} y) = z't(xy), x,y \in K_n \}\\ 
MT(K_{\infty},R) &=& \{ t \in \Hom_A(K^{\ab}_{\infty},R) \ | \ t(x s_n y) = zt(xy),
t(xs_n^{-1} y) = z't(xy), x,y \in K_n^{\ab} \} 
\end{array}
$$
Since, for $a,b \in K_{n+1}^{\ab}$, $ab = ba$, we have
$t(x s_n y) - zt(xy) = t(yxs_n - zyx)$ and  $t(x s_n^{-1} y) - z't(xy) = t(yxs_n^{-1} - z'yx)$.
It follows that 
$MT(K_{\infty},R) = \Hom_A(L(K_{\infty}),R)$ with
$L(K_{\infty})$ the quotient of $K_{\infty}^{\ab}$ by the $A$-submodule
spanned by the $xs_n - zx, xs_n^{-1} - z'x$ for $x \in K_n$.

Then is introduced an $A$-module $M$ defined as the quotient
of $K_n$ by the $A$-submodule spanned by the $as_ib - z ab, as_i^2 b - tab$
for $a,b \in K_i$ and $i<n$ (by abusing notations, here $K_i$
means the image of $K_i$ in $K_n$), and with $t = \gamma z'$. Since
$K_{n+1}$ is the sum of the $K_n s_n^{\eps} K_n$ for $\eps \in \{0,1,2\}$
we have $AF(K_n,R) = \Hom_A(M,R)$. The author of \cite{FUNAR} incorrectly
identifies this space with $R \otimes_A M$. More generally,
most of the arguments in \cite{FUNAR} implicitely assume that
the $A$-modules involved are free,
which is incorrect in view of our results. In particular,
for a nontrivial $t \in MT(K_{\infty},R)$ to exist, it is claimed that
$z,z'$ have to be related by the relations $(z')^2 = -z, z^2 = -z'$, these coming from $t(\qq s_1) = 0$,
$t(\qq s_1^2) = 0$ (in the notations of \cite{FUNAR}, $\qq s_1 = R_0$, $\qq s_1^2 = R_1$
and $\qq = R_2$). Actually, one finds that, if $t$ is such a Markov
trace, then $t(\qq s_1^2) = 4(z^2 + z')t(1)$,
$t(\qq s_1) = 4((z')^2 + z)t(1)$ and $t(\qq) = t(z_3) + 6zz' t(1) + t(1)$,
with $z_3 = (s_1 s_2)^3$. Of course division by $4$ is not licit
in general.

\section{Appendix : the 25-dimensional representation of $S_4(3)$}

A crucial tool for investigating $K_n$ in characteristic 3 has been
the 25-dimensional irreducible representation of $S_4(3)$, denoted $\varphi_5$ in \cite{ATLASBRAUER}
(see section \ref{ssecn5p3}).
We proved and used that it is defined over $\F_3$, and we computed an explicit
matrix model for it. 
We provide in figures \ref{s1g} to \ref{s4g} the images of the Artin
generators in such a model, so that the reader
have the possibility to check some of the computations of this paper.
In order to save space, the following convention
has been adopted for representing elements in $\F_3$ : a dot $\cdot$ represents $0$, a black square $\blacksquare$ represents $-1$
and an empty square $\square$ represents $1$.

\begin{figure}
\begin{center}
\resizebox{8cm}{!}{\includegraphics{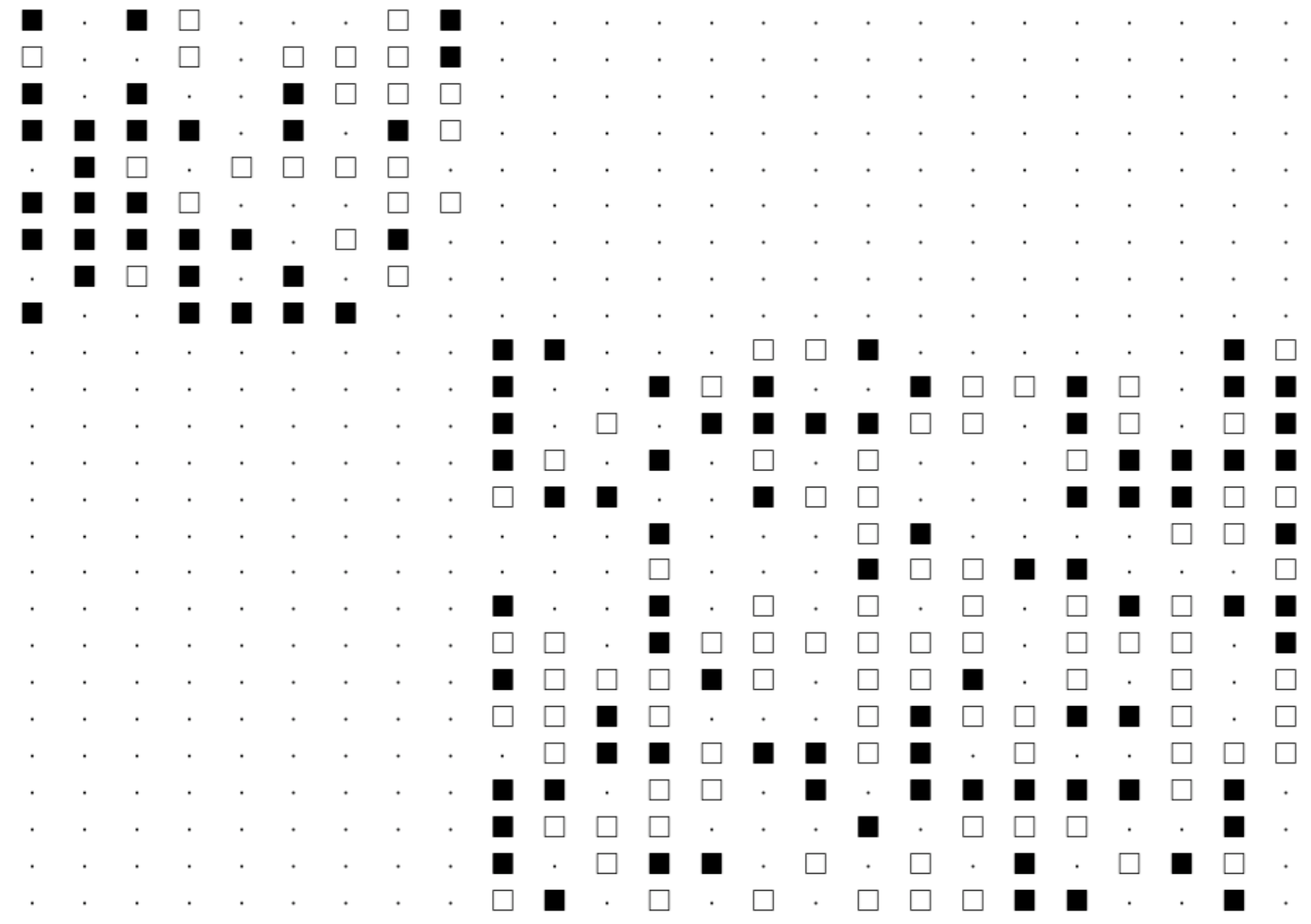}}\ \ 
\end{center}
\caption{$\varphi_5(s_1)$}
\label{s1g}
\end{figure}
\begin{figure}
\begin{center}
\resizebox{8cm}{!}{\includegraphics{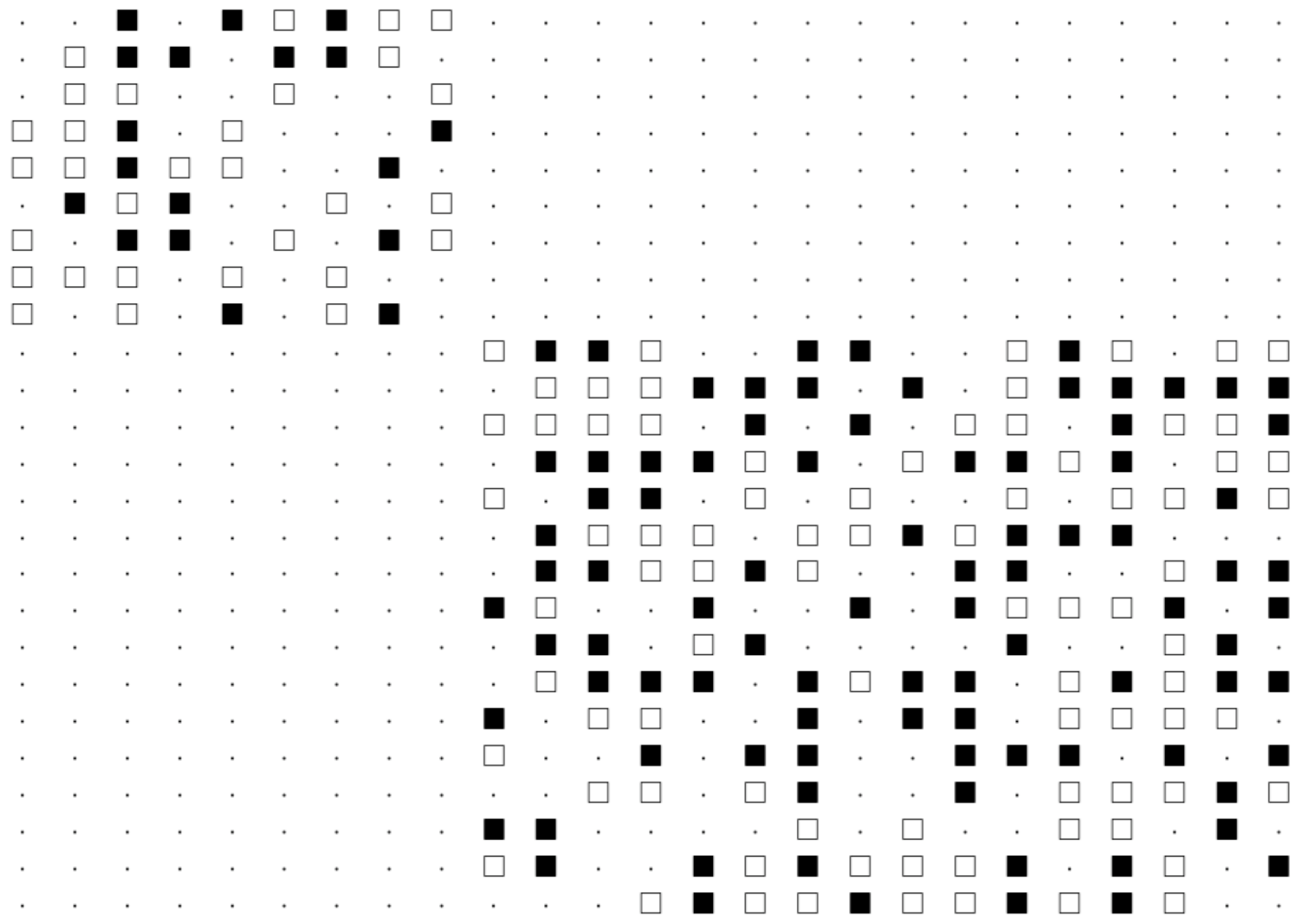}}\ \ 
\end{center}
\caption{$\varphi_5(s_2)$}
\label{s2g}
\end{figure}
\begin{figure}
\begin{center}
\resizebox{8cm}{!}{\includegraphics{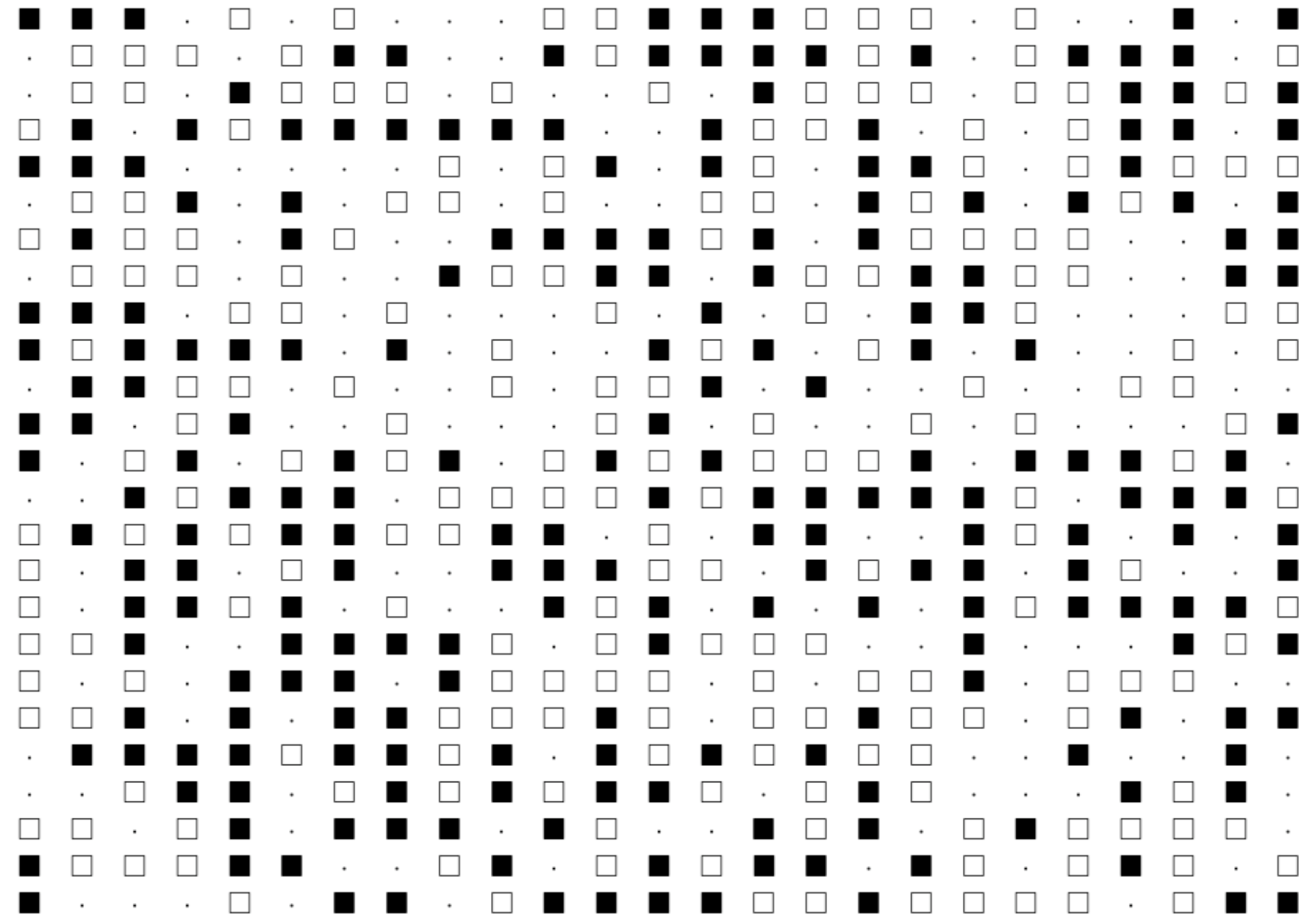}}\ \ 
\end{center}
\caption{$\varphi_5(s_3)$}
\label{s3g}
\end{figure}
\begin{figure}
\begin{center}
\resizebox{8cm}{!}{\includegraphics{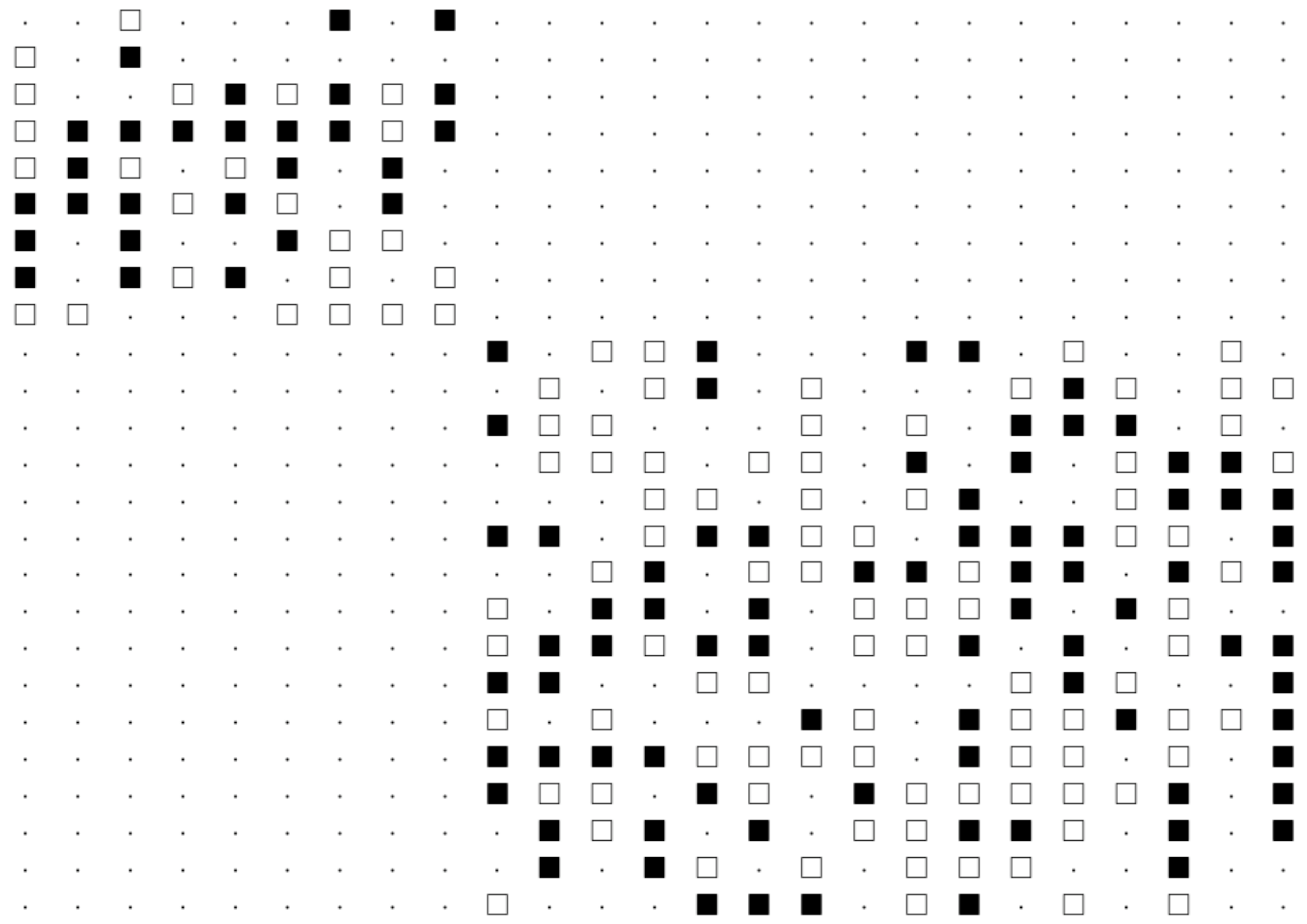}}\ \ 
\end{center}
\caption{$\varphi_5(s_4)$}
\label{s4g}
\end{figure}

\end{document}